\documentclass[12pt]{amsart}
\usepackage{statmath}

\usepackage[utf8]{inputenc}
\usepackage {amssymb}
\usepackage {amsmath}
\usepackage{amsthm}
\usepackage{graphicx}
\usepackage {amscd}
\usepackage {epic}
\usepackage{xypic}
\usepackage[colorlinks, citecolor=blue]{hyperref}
\usepackage[left=1.5in,right=1.5in,top=1.5in,bottom=1.5in]{geometry}
\usepackage{soul}
\usepackage{comment}
\usepackage{enumitem}
\usepackage{mathrsfs}

\theoremstyle{plain}
\newtheorem{theorem}{Theorem}[section]
\newtheorem{corollary}[theorem]{Corollary}
\newtheorem{lemma}[theorem]{Lemma}

\newtheorem{proposition}[theorem]{Proposition}
\newtheorem{definition-lemma}[theorem]{Definition-Lemma}
\theoremstyle{remark}
\newtheorem{remark}[theorem]{Remark}

\theoremstyle{definition}

\newcommand{\Pic}[0]{\operatorname{Pic}}

\def\ddbar{\partial\bar\partial}
\def\ve{\varepsilon}
\def\mcS{\mathcal{S}}
\def\NE{\overline{\operatorname{NE}}}
\def\NS{\operatorname{NS}}

\def\BC{\operatorname{BC}}
\def\NA{\overline{\operatorname{NA}}}
\def\Null{\operatorname{Null}}
\newcommand{\<}{\leq}
\def\>{\geq}
\newcommand{\mbQ}{\mathbb{Q}}
\newcommand{\mbR}{\mathbb{R}}

\def\mcO{\mathcal{O}}
\newcommand{\num}{\equiv}

\newcommand{\OO}{{\mathcal{O}}}
\newcommand{\Q}{{\mathbb{Q}}}
\newcommand{\C}{{\mathbb{C}}}
\newcommand{\R}{{\mathbb{R}}}

\newcommand{\mult}{{\rm mult}}
\newcommand{\Supp}{{\rm Supp}}

\newcommand{\sm}{{\rm sm}}

\newcommand{\mbC}{\mathbb{C}}
\newcommand{\mbZ}{\mathbb{Z}}

\newcommand{\bir}{\dashrightarrow}

\def\injective{\hookrightarrow}

\newcommand{\del}{\partial}
\def\lrd{\lfloor}
\def\rrd{\rfloor}
\def\oomega{\boldsymbol{\omega}}
\def\ggamma{\boldsymbol{\gamma}}

\def\bbeta{\boldsymbol{\beta}}

\def\mbN{\mathbb{N}}

\def\>{\geq}
\def\ve{\varepsilon}
\def\mcA{\mathcal{A}}
\def\mcO{\mathcal{O}}

\def\mcC{\mathcal{C}}

\def\mcL{\mathcal{L}}

\def\mcP{\mathcal{P}}

\def\mcS{\mathcal{S}}

\def\eps{\epsilon}

\def\lrd{\lfloor}
\def\rrd{\rfloor}
\def\pt{\operatorname{pt}}
\def\Ex{\operatorname{Ex}}
\def\dim{\operatorname{dim}}
\def\codim{\operatorname{codim}}

\def\sm{\operatorname{\textsubscript{sm}}}
\def\NS{\operatorname{NS}}
\def\NA{\operatorname{\overline{NA}}}

\def\NE{\operatorname{\overline{NE}}}

\def\Im{\operatorname{Im}}

\def\ext{\operatorname{ext}}
\def\exp{\operatorname{exp}}

\theoremstyle{definition}
\newtheorem{definition}[theorem]{Definition}
\theoremstyle{definition}

\numberwithin{equation}{section}

\theoremstyle{remark}
\newtheorem{claim}[theorem]{Claim}

\title{MMP for Generalized Pairs on K\"ahler 3-folds}
\author{Omprokash Das}
\address{School of Mathematics\\
Tata Institute of Fundamental Research\\
Homi Bhabha Road, Navy Nagar\\
Colaba, Mumbai 400005}
\email{omdas@math.tifr.res.in}
\email{omprokash@gmail.com}
\thanks{Omprokash Das was partially supported by the Start--Up Research Grant(SRG), Grant No. \# SRG/2020/000348 of the Science and Engineering Research Board (SERB), Govt. Of India.}

\author{Christopher Hacon}
\address{Department of Mathematics\\
University of Utah\\
155 S 1400 E\\
Salt Lake City, Utah 84112}
\email{hacon@math.utah.edu}
\thanks{Christopher Hacon and Jos\'e Ignacio Y\'a\~nez were partially supported by NSF research grants no: DMS-1952522, DMS-1801851 and by a grant from the Simons Foundation; Award Number: 256202.}

\author{Jos\'e Ignacio Y\'a\~nez}
\address{Department of Mathematics\\
University of Utah\\
155 S 1400 E\\
Salt Lake City, Utah 84112}
\email{yanez@math.utah.edu}

\begin{document}

\maketitle

\begin{abstract}
    In this article we define generalized pairs $(X, B+\bbeta)$ where $X$ is an analytic variety and $\bbeta$ is a b-(1,1) current. We then prove that almost all standard results of the MMP hold in this generality for compact K\"ahler varieties of $\dim X\<3$.  More specifically, we prove the cone theorem, existence of flips, existence of log terminal models, log canonical models and Mori fiber spaces, geography of log canonical and log terminal models, etc.
\end{abstract}

\tableofcontents
\section{Introduction}
In this article we will develop the minimal model program for generalized K\"ahler surfaces and threefolds.
Generalized pairs naturally arise in the context of Kawamata's canonical bundle formula and adjunction to lc centers, and have been playing an increasingly important role in the birational geometry of complex projective varieties (see \cite{Kaw98}, \cite{FM00}, \cite{BZ16}, \cite{Bir21} and references therein).
It is natural to hope that these results carry over to the context of K\"ahler manifolds, especially for surfaces and threefolds where the usual minimal model program is known to hold (see \cite{HP16, HP15, CHP16}, \cite{DO23}, \cite{DH20}, \cite{DH23} and references therein).
We introduce generalized K\"ahler pairs (Definition \ref{d-gp1}), in a context which is more general than the usual definition of generalized pairs from projective geometry. Roughly speaking, a generalized pair $(X/S,B+\beta)$ consists of a proper morphism $X\to S$ of normal K\"ahler varieties,
a pair $(X,B)$, and a closed (1,1) current $\beta\in H^{1,1}_{\rm BC}(X)$ which is (bimeromorphically) nef over $S$ (we refer the reader to Definition \ref{d-gp1} for the technical nuances; we will denote the corresponding closed nef b-(1,1) current by $\boldsymbol{\beta}$, but for the purposes of this introduction, we will sometimes abuse notation and just refer to $\beta=\bbeta _X$, the trace of $\bbeta$ on $X$).
Note that in the case of projective varieties one requires the more restrictive condition that $\beta$ is a $\mathbb R$-divisor (birationally nef over $S$). Thus, if $H^2(X, \mcO_X)\ne 0$ (and hence $\NS(X)_{\mbR}\neq H^{1,1}_{\BC}(X)$), this allows us more flexibility even in the projective case. This is particularly important in the K\"ahler case as there may be very few $\mathbb R$-divisors whilst $H^{1,1}_{\rm BC}(X)$ may contain many interesting classes.
For example, working in this generality allows us to: 
\begin{enumerate}
    \item Prove the finiteness of certain 3-fold minimal models (see Theorem \ref{t-finite-ltms}).
    \item Show that different 3-fold  minimal models are connected by flops (see Theorems \ref{t-finite-ltms} and \ref{t-flops}).
    \item Run the minimal model program with scaling of a K\"ahler form $\omega$ (see Theorems \ref{t-3-mmpbig} and \ref{t-3-mmppsef}). 
\end{enumerate} 
It is then possible to consider the various flavors of singularities of the minimal model program for generalized pairs (klt, lc, dlt etc.) and to show several natural properties (in all dimensions), such as the fact that generalized klt singularities are rational, and if $X$ is Stein, then
a generalized klt pair $(X,B+\bbeta)$ is equivalent to a usual klt pair $(X,B+\Delta)$ and in particular it admits a $\Q$-factorialization (see Theorem \ref{t-gkltlocal}). 
In Section \ref{s-gsmmp} we give a treatment of the generalized surface MMP including the  cone theorem, the existence of minimal models and Mori fiber spaces, and the existence of log canonical models when $K_X+B+\beta$ is big. In Section \ref{s-tgmmp} we develop the minimal model program for 3-fold generalized klt pairs.
We show that 3-fold klt flips exist.
\begin{theorem}\label{thm:3-dim-flips}
Let $(X,B+\bbeta)$ be a compact K\"ahler $\Q$-factorial 3-fold generalized klt pair, and $f:X\to Z$ a flipping contraction, then the flip $X^+\to Z$ exists.   
\end{theorem}
Proving the termination of flips in this generality however turns out to be too difficult. Instead, following the approach of \cite{BCHM10}, we show that certain generalized minimal model programs with scaling terminate. 
For example, if  $(X,B+\bbeta)$ is a compact K\"ahler $\Q$-factorial 3-fold generalized klt pair and $\beta=\bbeta _X$ is K\"ahler, then $K_X+B+t\beta$ is K\"ahler for $t\gg 0$, and a $K_X+B+\beta$ mmp with scaling of $(t-1)\beta$ is also a $K_X+B$ mmp with scaling of $t\beta$ and so, in this case, termination follows from standard results on the termination of flips for the usual klt $3$-fold pair $(X,B)$. This allows us to prove the existence of minimal and canonical models.
\begin{theorem}\label{thm:ltm}
    Let $(X,B+\bbeta)$ be a generalized compact klt K\"ahler 3-fold pair.
     \begin{enumerate}
        \item If $K_X+B+\bbeta_X$ is big, then $(X,B+\bbeta)$ has a log terminal model $f:X \dasharrow X^{\rm m}$ and a unique log canonical model $g:X^{\rm m}\to X^{\rm c}$.
     \item If $K_X+B+\bbeta _X$ is pseudo-effective and $\bbeta_X$ is big, then $K_X+B+\bbeta_X$ has a log terminal model $f:X\dasharrow X^{\rm m}$ and there is a contraction $g:X^{\rm m}\to Z$ such that $K_{X^m}+B^m+\bbeta_{X^m}\equiv g^*\omega _Z$ where $\omega _Z$ is a K\"ahler form on $Z$.
        \end{enumerate}
    \end{theorem}

    For more general minimal model programs with scaling, termination of flips is achieved by studying Shokurov polytopes and the geography of minimal models. In particular we show the following (please see Theorems \ref{t-finite} and \ref{t-finite-ltms} for a more comprehensive statement).
\begin{theorem}\label{thm:models}
    Let $X$ be a smooth compact K\"ahler 3-fold, $B$ a simple normal crossings divisor, and $\Omega$ a compact convex polyhedral set of real closed (1,1)-currents such that $[\beta]\in H^{1,1}_{\BC}(X)$ is nef and $[K_X+B+\beta]\in H^{1,1}_{\BC}(X)$ is big for all $\beta \in \Omega$. Then there exist a finite polyhedral decomposition $\Omega =\cup \Omega _i$ and finitely many bimeromorphic maps $\psi _{i,j}:X\dasharrow X_{i,j}$ such that if $\psi : X\dasharrow  Y$ is a {weak log canonical} model for some $\beta \in \Omega$, then $\psi =\psi _{i,j}$ for some $i,j$.
\end{theorem}
Building on this result, we are able to show that good minimal models are connected by flops (and in general minimal models are connected by flips, flops and anti-flips).

\begin{theorem}\label{thm:connecting-mm}
 Let $(X_i,B_i+\bbeta _{X_i} )$ be strongly $\Q$-factorial compact generalized klt K\"ahler 3-folds, where $K_{X_i}+B_i+\bbeta_{X_i}$ is nef (resp. $(X_i,B_i+\bbeta_{X_i})$ are good minimal models) for $i=1,2$ and $\phi:X_1\dasharrow X_2$ a bimeromorphic map which is an isomorphism in codimension 1. Then $\phi$ can be decomposed as flips, flops and inverse flips, see Definition \ref{def:inverse-flip} (resp. $\phi$ can be decomposed as a sequence of flips). 
\end{theorem}

When $K_X+B+\bbeta_X$ is not pseudo-effective, we show the existence of a Mori fiber space, see Theorem \ref{t-3-mmppsef}.
\begin{theorem}\label{thm:mfs}
    Let $(X,B+\bbeta)$ be a  strongly $\Q$-factorial generalized klt K\"ahler 3-fold such that $K_X+B+\bbeta_X$ is not pseudo-effective. Then we can run a $(K_X+B+\bbeta_X)$-MMP $X\dasharrow X'$ ending with a Mori fiber space $X'\to Z$.
\end{theorem}

We also establish the following cone theorem.
\begin{theorem}\label{t-3-cone+}
 Let $(X,B+\bbeta_X)$ be a generalized klt pair, where $X$ is a compact K\"ahler $3$-fold. Then there are at most countably many rational curves $\{\Gamma _i\}_{i\in I}$ such that 
 \[\overline{\rm NA}(X)=\overline{\rm NA}(X)_{K_X+B+\bbeta_X \geq 0}+\sum _{i\in I}\mathbb R ^+[\Gamma _i],\]
 and $-(K_X+B+\bbeta_X)\cdot \Gamma _i\leq 6$. Moreover, if $\bbeta$ is big, then $I$ is finite.
 \end{theorem}
We believe that the added flexibility afforded by working with nef classes in $H^{1,1}_{\rm BC}$ will be useful in a variety of contexts. For example, we use this when showing that bimeromorphic Calabi-Yau threefolds are connected by flops (Theorem \ref{t-flops}), and we expect that it will be important in the proof of the minimal model program for klt pseudo-effective K\"ahler 4-folds \cite{DH24}. Note that the case of effective klt K\"ahler 4-folds was addressed in \cite{DHP22}.

We remark that the above results are known for projective varieties of arbitrary dimension \cite{DH24a}.

This article is organized in the following manner: In Section 2 we define generalized pairs, generalized models and establish the generalized MMP for K\"ahler surfaces. We also the prove Theorem \ref{thm:3-dim-flips} in this section. Section 3 is the heart of our article, Theorem \ref{thm:ltm} is proved in Subsection \ref{subsec:ltm}, Theorem \ref{thm:mfs} is proved in Subsection \ref{subsec:mfs}, Theorem \ref{t-3-cone+} is proved in Subsection \ref{subsec:cone}, and Theorem \ref{thm:connecting-mm} is proved in Subsection \ref{subsec:flops}.\\

\noindent
{\bf Acknowledgment} We would like to thank Mihai P\u{a}un for answering our questions and the referee for carefully reading our manuscript and suggesting important improvements.

\section{Preliminaries}
An \emph{analytic variety} or simply a \emph{variety} $X$ is a reduced and irreducible complex space. A holomorphic map $f:X\to Y$ between two complex spaces is called a \emph{morphism}. A \textit{small bimeromorphic map} or a \textit{small map} is a bimeromorphic map $\phi:X\bir X'$ between two normal analytic varieties such that $\phi$ is an isomorphism in codimension $1$, i.e. there are closed analytic subsets $Z\subset X$ and $Z'\subset X'$ such that $\codim_X Z\>2$ and $\codim_{X'}Z'\>2$ and $\phi|_{X\setminus Z}:X\setminus Z\to X'\setminus Z'$ is an isomorphism. A $(1,1)$ class $\alpha\in H^{1,1}_{\BC}(X)$ is called \emph{general} (resp. \emph{very general}) if $\alpha$ is not contained in any finite union (resp. countable union) of analytic subvarieties of $H^{1,1}_{\BC}(X)$.

\begin{definition}
    Let $X$ be a normal analytic variety. The canonical sheaf $\omega_X$ is defined as $\omega_X:=(\wedge^{\dim X} \Omega^1_X)^{**}$. Note that unlike the case of algebraic varieties, $\omega_X$ here does not necessarily correspond to a Weil divisor $K_X$ such that $\omega_X\cong \mcO_X(K_X)$. However, by abuse of notation we will say that $K_X$ is a canonical divisor when we actually mean the canonical sheaf $\omega_X$. This doesn't create any problem in general as running the minimal model program involves intersecting subvarieties with $\omega_X$. 
    
   A $\mbQ$-divisor (resp. an $\mbR$-divisor) on a normal analytic variety (non necessarily compact) is a \emph{finite sum} of prime Weil divisor with $\mbQ$-coefficients (resp. $\mbR$-coefficients).
    A compact normal analytic variety $X$ is called \emph{$\mbQ$-factorial} if for every prime Weil divisor $D\subset X$ there is a $m\in\mbZ^+$ such that $mD$ is Cartier and there is a $k\in\mbZ^+$ such that $(\omega^{\otimes k}_X)^{**}$ is a line bundle on $X$.
    We say that $X$ is \emph{strongly $\mbQ$-factorial} if for every reflexive rank 1 sheaf $\mcL$, there is a positive integer $m\in\mbZ^+$ such that $\mcL^{[m]}:=(\mcL^{\otimes m})^{**}$ is a line bundle.
    
    For a normal analytic variety $X$ and an $\mbR$-divisor $B$ we say that $K_X+B$ is $\mbR$-Cartier, if locally around any point $x\in X$ we can choose a divisor $K_X$ such that $\OO _X(K_X)\cong \omega _X$ and $K_X+B$ is $\R$-Cartier. In this case, we define the singularities of the pair $(X, B)$ as in \cite{KM98}. Note that throughout this article, by a pair $(X, B)$, we will always mean that $X$ is normal,  $B\>0$ is an effective $\R$-divisor,  and $K_X+B$ is $\R$-Cartier. If $B$ is not effective, then we will refer to the corresponding singularities of $(X, B)$  as sub-klt, sub-dlt, etc.
\end{definition}

\begin{definition}
    An analytic variety $X$ is  \emph{K\"ahler} or a \emph{K\"ahler space} if there exists a positive closed real $(1, 1)$ form $\omega\in\mcA_{\mbR}^{1, 1}(X)$ such that the following holds: for every point $x\in X$ there exists an open neighborhood $x\in U$ and a closed embedding $\iota_U:U\injective V$ into an open set $V\subset\mbC^N$, and a strictly plurisubharmonic $\mathcal C^\infty$ function $f:V\to\mbR$ such that $\omega|_{U\cap X_{\sm}}=(i\del\bar{\del}f)|_{U\cap X_{\sm}}$. Here $X_\textsubscript{sm}$ is the smooth locus of $X$.
    \begin{enumerate}
        \item On a normal compact analytic variety $X$ we replace the use of N\'eron-Severi group $\NS(X)_{\mbR}$ by $H^{1, 1}_{\rm BC}(X)$, the Bott-Chern cohomology of real closed $(1, 1)$ forms with local potentials or equivalently, the closed bi-degree $(1, 1)$ currents with local potentials. See \cite[Definition 3.1 and 3.6]{HP16} for more details. More specifically, we define
\[ N^1(X):=H^{1,1}_{\rm BC}(X).\]

	\item If $X$ is in Fujiki's class $\mcC$ and $X$ has \emph{rational singularities}, then from \cite[Eqn. (3)]{HP16} we know that $N^1(X)=H^{1, 1}_{\rm BC}(X)\subset H^2(X,\mathbb R)$. In particular, the intersection product can be defined in $N^1(X)$ via the cup product of $H^2(X, \mbR)$.
    \item For the definitions of nef, pseudo-effective class, etc. see \cite[Definition 2.2]{DH20}. 
    \item We define $\NA(X)\subset N_1(X)$ to be the closed cone generated by the classes of positive closed currents of bi-dimension $(1,1)$, see \cite[Definition 3.8]{HP16}. The {\it Mori cone} $\NE(X)\subset\NA(X)$ is defined as the closure of  the cone of currents of integration $T_C$, where $C\subset X$ is an irreducible curve.

    \end{enumerate}
\end{definition}

\begin{definition}
    If $X$ is a normal K\"ahler variety and $\omega \in H^{1,1}_{\rm BC}(X)$, then we say that $\omega$ is {\it modified K\"ahler} if there exists a bimeromorphic morphism $\nu :X'\to X$ and  K\"ahler form $\omega '$ on $X'$ such that $\nu _* \omega '=\omega$. By \cite[Proposition 2.3]{Bou04}, if $X$ is compact, then this is equivalent to requiring that $\omega $ contains a K\"ahler current $T$ with Lelong number $\nu (T, D) = 0$
for all prime divisors $D$ in $X$.
\end{definition}

\begin{definition}
Let $\pi :X\to S$ be a proper morphism of normal K\"ahler varieties such that $S$ is relatively compact.
Let $\beta $ be a closed $(1,1)$ current with local potentials, i.e. a locally $\ddbar$-exact current on $X$. We say that the class $[\beta]\in H^{1,1}_{\BC}(X)$ is relatively K\"ahler (or K\"ahler over $S$) if $[\beta +\pi ^*\omega_S]\in H^{1,1}_{\BC}(X)$ is a K\"ahler class for some K\"ahler form $\omega_S$ on $S$, and we say that the class $[\beta] $ is relatively nef if $[\beta+\omega]$ is relatively K\"ahler for every relatively K\"ahler class  $[\omega]$ on $X$. Similarly, we say that $\beta$ is relatively modified K\"ahler if $\beta +\pi ^*\omega_S$ is modified K\"ahler  for some K\"ahler form $\omega_S$ on $S$.
\end{definition}
It is well known that if a class $[\beta]\in H^{1,1}_{\BC}(X)$ is relatively K\"ahler (resp. relatively nef), then its restriction to each fiber is K\"ahler  (resp. nef). By abuse of notation we will say that a closed bi-degree $(1,1)$ current $T$ with local potentials is relatively K\"ahler or relatively nef over $S$ if so is its class $[T]\in H^{1,1}_{\BC}(X)$.

%%%%%%%%%%%%%%%%%%%%%%%%%%%

\subsection{b-(1,1) Currents}
Let $X$ be a normal analytic variety.
A {\it real closed b-(1,1) current} $\boldsymbol{\beta}$ is a collection of real closed bi-degree (1,1) currents $\boldsymbol{\beta}_{X'}$ on all proper bimeromorphic models $X'\to X$  such that if $p :X_1\to X_2$ is a bimeromorphic morphism of proper models of $X$, then $p_* \boldsymbol{\beta}_{X_1}=\boldsymbol{\beta}_{X_2}$. We say that a b-(1,1) current $\bbeta$ \emph{descends} to a model $X'\to X$ if $\bbeta_{X'}$ has \textit{local potentials} on $X'$ and for any higher model $f:X''\to X'$, $\bbeta_{X''}$ also has local potentials and $[\bbeta_{X''}]=f^*[\bbeta_{X'}]$ in $H^{1,1}_{\BC}(X'')$. Here $f^*[\bbeta_{X'}]$ is defined in the following manner: choose a smooth (1,1) form $\omega$ in $[\bbeta_{X'}]\in H^{1,1}_{\BC}(X')$ and define $f^*[\bbeta_{X'}]=[f^*\omega]$. We note that, unlike differential forms, currents cannot be pulled back in general even when they have local potentials because those local potentials could be \textit{distributions} instead of \textit{functions}. If $\bbeta$ is a b-(1,1) current as above and $X'\to X$ is a model of $X$, then the current $\bbeta_{X'}$ is called the \emph{trace} of $\bbeta$ on $X'$. Moreover, we say that $\bbeta$ is a \emph{positive} b-(1,1) current if all of its traces are positive currents.

\subsubsection{b-(1,1) currents defined by positive currents}
Suppose that $\beta $ is a closed positive (1,1)-current on $X$ with local (psh) potentials, then we may define a b-(1,1) current $\bar \beta$ as follows. For any bimeromorphic morphism $\nu :X'\to X$ we let $\bar\beta_{X'}:=\nu ^*\beta$.
Explicitly, if $X=\cup U_i$ is an open cover and $\gamma _i$ are psh functions on $U_i$ such that $\beta =\partial\bar \partial \gamma _i$, then $\nu ^* \beta$ is defined by letting $U_i'=\nu ^{-1}U_i$, $\gamma _i'=\gamma _i\circ \nu|_{U_i'}$, and $\nu ^* \beta=\partial\bar \partial \gamma _i'$ on $U'_i$.
If $\mu :X'\to X''$ is another proper bimeromorphic morphism, then we let $\bar \beta_{X''}=\mu _*\beta_{X'}$.
We note that 
\begin{claim}\label{clm:b-current}
    The closed b-(1,1) current $\bar \beta_{X''}$ is well defined.
\end{claim} \begin{proof} 
    Suppose that $\tilde \nu:\tilde X\to X$ and $\tilde \mu:\tilde X\to X''$ are also proper bimeromorphic morphisms of normal complex varieties. By a standard argument, passing to a common resolution, we may in fact assume that there is  a bimeromorphic morphism $\rho:\tilde X \to X'$ such that $\tilde \nu =\nu \circ \rho $ and $\tilde \mu :\mu  \circ \rho $. Then by the projection formula we have
\[\tilde \mu_*(\tilde \nu ^* \beta)=\mu_*\rho _*(\rho ^* \nu ^* \beta)= \mu_*( \nu ^* \beta).\]

Note that any real closed smooth $(1,1)$ form $\omega$ with local potentials can be thought of as a current, and hence it also defines a b-(1,1) current. 
\end{proof}
 If $\boldsymbol{\beta}=\bar \beta$ for some closed positive (1,1)-current $\beta$ (with local potentials) on $X$, then from the definition it follows that $\bbeta$ descends to $X$. Moreover, in this case for any bimeromorphic morphism $\nu :X'\to X$ we also have that $\boldsymbol{\beta}=\overline{\boldsymbol{\beta}_{X'}}$ i.e. $\boldsymbol{\beta}$ also descends to $X'$.

\begin{remark}\label{r-gp-new} We make the following observations:
\begin{enumerate}
\item[(i)] Note that if $\gamma \in H^{1,1}_{\rm BC}(X')$ is nef, then it is pseudo-effective and so we may choose a positive  closed (1,1) current $\beta '$ on $X'$ with psh local potentials such that $[\beta ']=\gamma$ and we may then set $\boldsymbol{\beta}:=\bar \beta'$. Different choices of $\beta '$ give rise to different (non-equivalent) b-(1,1) currents.
\item[(ii)] Note that if $\boldsymbol{\beta}$ is a positive  closed b-(1,1) current that descends to $X$ and $X\dasharrow X'$ is bimeromorphic (and $X'$ is normal), then $\boldsymbol{\beta}_{X'}$ may not have local potentials, but if it does, then it has psh local potentials. 
To see this, first note that in this case $[\boldsymbol{\beta}_{X'}]\in H^{1,1}_{\BC}(X')$. Let $p:X''\to X$ and $q:X''\to X'$ be a common resolution and $U':=X'\setminus (X'_\textsubscript{sing}\cup q(\Ex(q)))$ so that $U'':=q ^{-1}U'\to U'$ is an isomorphism. 
Then $\boldsymbol{\beta}_{X'}|_{U'}=\boldsymbol{\beta}_{X''}|_{U''}$, and since $\boldsymbol{\beta}_{X''}$ is a positive current, from \cite[Proposition 4.6.3(i)]{BG13} it follows that $\boldsymbol{\beta}_{X'}|_{U'}$ has a unique extension 
$\widehat {\boldsymbol{\beta}_{X'}|_{U'}}$ to a closed positive (1,1) current on $X'$ such that $\left[\widehat {\boldsymbol{\beta}_{X'}|_{U'}}\right]=[\boldsymbol{\beta}_{X'}]$.\\
\end{enumerate}
\end{remark}

\subsection{Generalized Pairs}
\begin{definition}\label{d-gp1}
    Let $f:X\to S$ be a proper morphism of normal analytic varieties, where $S$ is relatively compact, $\nu :X'\to X$ a resolution, $B'$ a $\mathbb R$-divisor on $X'$ with simple normal crossings support such that $B:=\nu_* B'\geq 0$, and $\boldsymbol{\beta}$ a real closed b-(1,1) current. We say that $(X,B+\boldsymbol{\beta})$ is a {\it generalized pair} if 
    \begin{enumerate}
        \item $\boldsymbol{\beta}$ descends to $X'$,
        \item  $[\boldsymbol{\beta}_{X'}]\in H^{1,1}_{\rm BC}(X')$ is nef over $S$, and 
        \item $[K_{X'}+B'+\boldsymbol{\beta}_{X'}]=\nu ^*\gamma$ for some $\gamma \in H^{1,1}_{\rm BC}(X)$.
    \end{enumerate}
\end{definition}
Note that we are abusing notation as we are implicitly assuming the existence of $(X',B')$ as above. We will say that $\nu: (X',B')\to (X,B)$ is a \emph{structure morphism} or a \emph{log resolution} of $(X, B+\boldsymbol{\beta})$.

\begin{remark}\label{r-gp} We make the following observations:
\begin{enumerate}
\item[(i)]  Given $(X,B+\boldsymbol{\beta})$ and a log resolution with the above properties,  $B'$ is uniquely determined (by the negativity lemma applied to $\nu :X'\to X$).

\item[(ii)] If $S$ is a point so that $X$ is compact, then we drop $S$ and say that $(X,B+\boldsymbol{\beta} )$ is a compact generalized pair.
\item[(iii)] If $U\subset X$ is a relatively compact subset, then $(U/U,B|_U+\boldsymbol{\beta}|_U)$ is a generalized pair over $U$.
\item[(iv)] If $S=X$ and $\pi :X\to S$ is the identity (and in particular $X$ is relatively compact), then we also drop $S$ and we often abuse notation and say that $(X,B+\boldsymbol{\beta} )$ is a generalized pair.
\end{enumerate}

\end{remark}~\\

\begin{definition}\label{def:g-singularities}
\begin{enumerate}
 \item Let $P$ be a prime Weil divisor over $X$. We define the \emph{generalized discrepancy} $a(P, X, B+\boldsymbol{\beta})$ as follows: Let $\nu :X'\to X$ be a log resolution of $(X, B+\boldsymbol{\beta})$ such that $P\subset X'$ is prime Weil divisor on $X'$. Then $a(P, X,B+\boldsymbol{\beta}):=-{\rm mult}_P(B')$. Note that these can be computed locally over $X$ and hence $S$ plays no role here (and hence we drop it from the notation).

\item We say that $(X,B+\boldsymbol{\beta} )$ is \emph{generalized klt} or gklt or generalized Kawamata log terminal (resp. \emph{generalized lc} or glc or generalized log canonical) if for some log resolution $\nu ':X'\to X$, we have $\lfloor B'\rfloor\leq 0$, i.e. $a(P, X,B+\boldsymbol{\beta})>-1$ for all prime divisors $P\subset X'$ (resp. $a(P, X,B+\boldsymbol{\beta})\geq -1$ for all prime divisors $P\subset X'$). 

\item We say that $(X,B+\boldsymbol{\beta} )$ is \emph{generalized dlt} or gdlt or generalized divisorially log terminal if there is an open subset $U\subset X$
such that $(U,(B+\boldsymbol{\beta})|_U)$ is a log resolution (of itself) and $-1\leq a(P,X,B+\boldsymbol{\beta})\leq 0$ for any prime divisor $P$ on $U$
and $-1< a(P,X,B+\boldsymbol{\beta})$ for any prime divisor $P$ over $X$ with center contained in $X\setminus U$. 

\end{enumerate}
\end{definition}

\begin{remark}
\begin{enumerate}
\item The above definitions are inspired by the more traditional generalized pairs for projective varieties introduced in \cite{BZ16}. We note that if $K_X+B$ is $\R$-Cartier, then 
$a(P,X,B)\geq a(P,X,B,\bbeta )$ for all prime divisors $P$ over $X$ and equality holds for all such $P$ if and only if $\bbeta$ descends to $X$. In particular, if $(X,B+\bbeta)$ is gklt and $K_X+B$ is $\R$-Cartier, then $(X,B)$ is also klt see   Lemma \ref{l-klt}.
    \item With notations and hypothesis as in Definition \ref{d-gp1}, let $\theta$ be a closed positive (1,1) current (resp. a real closed smooth (1,1) form) on $X'$ cohomologous to $\bbeta_{X'}$. (Note that a positive current in the class $[\bbeta_{X'}]$ exists if $X$ is compact, as $[\bbeta_{X'}]$ is nef in that case. On the other hand, given any class in $H^{1,1}_{\BC}(X')$ there always exists a real closed smooth (1,1) form representing it.) Let $\boldsymbol{\theta}:=\overline\theta$ be the real closed b-(1,1) current defined by $\theta$. Then $(X, B+\boldsymbol{\theta})$ is a generalized pair and from the negativity lemma it follows that the discrepancies $a(E, X, B+\boldsymbol{\theta})=a(E, X, B+\bbeta)$ for all divisors $E$ over $X$. 

    \item With notations and hypothesis as in Definition \ref{d-gp1} if $(X, B+\bbeta)$ is a generalized pair with $S=\{\pt\}$ such that $\bbeta$ already descends to $X$ and $[\bbeta_X]\in H^{1,1}_{\BC}(X)$ is nef, then from  the negativity lemma it follows that the discrepancies $a(E, X, B+\bbeta)$ do not depend on the b-(1,1) current $\bbeta$, in fact in this case $K_X+B$ is $\mbR$-Cartier and $a(E, X, B+\bbeta)=a(E, X, B)$ for all divisors over $X$; in particular the singularities of $(X, B+\bbeta)$ are same as the singularities of $(X, B)$. In light of this fact, given a generalized pair $(X, B+\bbeta)$ we often pick a K\"ahler class $\omega\in H^{1,1}_{\BC}(X)$ and by abuse of language we denote by $(X, B+\bbeta+\omega)$ the generalized pair corresponding to $(X, B+\bbeta+\bar \omega)$. Note that $a(E, X, B+\bbeta)=a(E, X, B+\bbeta+\omega )$ for any divisor $E$ over $X$. 

    \item By abuse of notation we will often say that $\beta$ is a $(1,1)$ class in $H^{1,1}_{\rm BC}(X)$ when we actually mean $\beta$ is a real closed bi-degree $(1,1)$ current on $X$ with local potentials.
\end{enumerate}

\end{remark}

\subsection{Generalized Models}
%%%%%%%%%%%%%%%%%%%%%%%%%%%%%%%%
\begin{definition}\label{d-models}
If $(X/S,B+\boldsymbol{\beta})$ is a generalized dlt pair over $S$, then we say that a bimeromorphic map $\phi:X\dasharrow X^{\rm m}$ (proper over $S$) is a \emph{log minimal model over $S$} (resp. a \emph{log terminal model} over $S$) if (1-3) below hold (resp. (1-4) below hold).
\begin{enumerate}
    \item $(X^{\rm m},B^{\rm m}+\boldsymbol{\beta} )$ is $\mbQ$-factorial generalized dlt pair, where $B^{\rm m}=\phi _*B+E$, and $E$ is the reduced sum of all $\phi ^{-1}$-exceptional divisors,
    %\item the generalized pairs $(X,B+\boldsymbol{\beta} )$ and $(X',B'+\beta ')$ have the same nef b-current $\boldsymbol{\beta }$,
    \item $K_{X^{\rm m}}+B^{\rm m}+\boldsymbol{\beta}_{X^{\rm m}}$ is nef over $S$,
    \item $a(P,X,B,\boldsymbol{\beta} )<a(P,X^{\rm m},B^{\rm m}, \boldsymbol{\beta})$ for every $\phi$-exceptional divisor $P$, and
    \item there are no $\phi^{-1}$-exceptional divisors, i.e. $E=0$.
\end{enumerate}
We say that $\phi:X\dasharrow X^{\rm m}$ (proper over $S$) is a \emph{good log minimal model over $S$} (resp. a \emph{good log terminal model} over $S$) if (1-3) above hold (resp. (1-4) above hold) and there exists a morphism $g:X^{\rm m}\to Z$ over $S$, and a K\"ahler form $\alpha _Z$ on $Z$ such that $K_{X^{\rm m}}+B^{\rm m}+\boldsymbol{\beta}_{X^{\rm m}}\equiv g^*\alpha _Z$.

If $(X/S,B+\boldsymbol{\beta})$ is a generalized dlt pair over $S$, then we say that a bimeromorphic map $\phi:X\dasharrow X^{\rm m}$  (proper over $S$) is a \emph{weak log canonical model over $S$} (resp. a \emph{log canonical  model over $S$}) if (1-3) below hold (resp. (1-4) below hold).
\begin{enumerate}
    \item $(X^{\rm m},B^{\rm m}+\boldsymbol{\beta})$ is generalized lc pair, where $B^{\rm m}:=\phi _*B+E$, and $E$ is the reduced sum of all $\phi ^{-1}$-exceptional divisors,
   
   \item $K_{X^{\rm m}}+B^{\rm m}+\boldsymbol{\beta}_{X^{\rm m}}$ is nef over $S$, 
    \item $a(P,X,B,\boldsymbol{\beta} )\leq a(P,X^{\rm m},B^{\rm m},\boldsymbol{\beta})$ for every $\phi$-exceptional divisor $P$, and
    \item $[K_{X^{\rm m}}+B^{\rm m}+\boldsymbol{\beta}_{X^{\rm m}}]\in H^{1,1}_{\BC}(X^{\rm m})$ is a K\"ahler over $S$.
\end{enumerate}
If $X$ is proper and $S$ is a point, then we drop ``over $S$'' and simply say that we have a log minimal model, log terminal model etc.
\end{definition}
\begin{lemma}\label{l-models}
Suppose that $(X/S,B+\boldsymbol{\beta})$ is a generalized dlt pair over $S$.
\begin{enumerate}
    \item If $\phi:X\dasharrow X^{\rm m}$ is a weak log canonical model over $S$, then $a(P,X,B,\boldsymbol{\beta})\leq a(P,X^{\rm m},B^{\rm m},\boldsymbol{\beta})$ for every divisor $P$ over $X$ and $a(P,X,B,\boldsymbol{\beta})= a(P,X^{\rm m},B^{\rm m},\boldsymbol{\beta})$ for every divisor $P$ on $X^{\rm m}$.
    \item If $X\dasharrow X^{\rm m}$ and $X\dasharrow X^w$ are weak log canonical models of $(X/S,B+\beta)$ over $S$, then $(X^{\rm m},B^{\rm m}+\boldsymbol{\beta})$ and $(X^w,B^w+\boldsymbol{\beta})$ are crepant equivalent, i.e. if $p:Z\to X^{\rm m}$ and $q:Z\to X^w$ is a resolution of the induced map $X^{\rm m}\bir X^w$, then $p^*(K_{X^{\rm m}}+B^{\rm m}+\boldsymbol{\beta}_{X^{\rm m}})\equiv _S q^*(K_{X^w}+B^w+\boldsymbol{\beta}_{X^w})$.
    \item If $X\dasharrow X^{\rm m}$ and $X\dasharrow X^w$ are log canonical models of $(X/S,B+\boldsymbol{\beta})$   over $S$, then $(X^{\rm m},B^{\rm m}+\bbeta^m)$ and $(X^w,B^w+\bbeta^w)$ are isomorphic.
    \item If $X\dasharrow X^{\rm m}$ and $X\dasharrow X^w$ are log terminal models of $(X/S,B+\boldsymbol{\beta})$   over $S$, then $X^{\rm m}$ and $ X^w$ are isomorphic in codimension 1.  
    \item If $X\dasharrow X^{\rm m}$ is a minimal model and  $X\dasharrow X^w$ is a log canonical model of $(X/S,B+\boldsymbol{\beta})$   over $S$, then $X^{\rm m}\to  X^w$  is a morphism.
    \item If $(X,B+\boldsymbol{\beta})$ is generalized klt, then every log minimal model over $S$ is a log terminal model over $S$.
    \item If $f:X'\to X$ is a log resolution of $(X/S,B+\boldsymbol{\beta})$ and $K_{X'}+B^*+\boldsymbol{\beta} _{X'}=f^*(K_X+B+\boldsymbol{\beta}_X)+F$, where $B^*\>0, f_*B^*=B$ and $F\geq 0$ is $f$-exceptional such that for every $f$-exceptional divisor $P$ with $a(P,X,B+\boldsymbol{\beta})>0$ we have $P\subset {\rm Supp}(F)$. Then any log minimal model (resp. (weak) log canonical model) of $(X'/S,B^*+\boldsymbol{\beta})$ over $S$ is a log minimal model (resp. (weak) log canonical model) of $(X/S,B+\boldsymbol{\beta})$ over $S$. If moreover  $\Supp(F)=\Ex(f)$, then any log terminal model of $X'$ is a log terminal model of $X$. \end{enumerate}

\end{lemma}
Even though the proof of this lemma follows along standard arguments, we include the proof for the convenience of the reader.
\begin{proof}
(1) Let $p:Z\to X$ and $q:Z\to X^{\rm m}$ be a resolution of $\phi$. Then we can write
$F=\sum (a(P,X^{\rm m},B^{\rm m},\bbeta)-a(P,X,B,\bbeta))P$, where the sum runs over all prime divisors $P\subset Z$. Then from the definition above it follows that  
$p_*F\geq 0$.
Note that $F\equiv _S p^*(K_X+B+\bbeta_X )-q^*(K_{X^{\rm m}}+B^{\rm m}+\bbeta_{X^{\rm m}})$, and hence $F\geq 0$ by the negativity lemma, as $K_{X^{\rm m}}+B^{\rm m}+\bbeta_{X^{\rm m}}$ is nef over $S$ and thus $-F$ is nef over $X$. We also claim that $q_*F=0$. Observe that, here $B^{\rm m}=\phi_*B+E$, where $E$ is the reduced sum of $\phi^{-1}$-exceptional divisors on $X^{\rm m}$.  
Thus it is enough to show that ${\rm mult}_P(F)=0$ for any prime divisor $P$ in the support of $E$ (i.e. any $p$-exceptional divisor which is not $q$-exceptional). We have
\[-1=a(P,X^{\rm m},B^{\rm m}+\bbeta )\geq a(P,X,B+\bbeta)\geq -1,\]
where the second inequality holds because $F\geq 0$. In particular, we have
\[
a(P,X,B+\bbeta) = a(P,X^{\rm m},B^{\rm m}+\bbeta )
\]
for all prime Weil divisors $P$ on $X^{\rm m}$.\\

(2) Let $F=\sum (a(P,X^{\rm m},B^{\rm m},\bbeta)-a(P,X^w,B^w,\bbeta))P$, where the sum runs over all prime divisors $P\subset Z$. It is easy to see that \[F\equiv _S q^*(K_{X^{w}}+B^{w}+\bbeta _{X^{w}})-p^*(K_{X^{\rm m}}+B^{\rm m}+\bbeta _{X^{\rm m}})\] and $q_*F$ and $-p_*F$ are effective. Since $-F$ is $q$-nef and $F$ is $p$-nef, it follows that $F=0$ by the negativity lemma. \\

(3) Let $W$ be the normalization of the graph of $X^{\rm m}\dasharrow X^w$, and $p:W\to X^{\rm m}$, $q:W\to X^w$ the induced morphisms. Then by part (2) we have $p^*(K_{X^{\rm m}}+B^{\rm m}+\bbeta _{X^{\rm m}})\equiv q^*(K_{X^w}+B^w+\bbeta _{X^w})$. Since $p,q$ are bimeromorphic, and hence Moishezon morphisms, if $X^{\rm m}\dasharrow X^w$ is not an isomorphism, we may assume that there is a curve $C\subset W$ such that $p_*C=0$ and $q_*C\ne 0$ (or $p_*C\ne 0$ and $q_*C= 0$). But then
\begin{align*}
    0=p_*C\cdot (K_{X^{\rm m}}+B^{\rm m}+\bbeta_{X^{\rm m}} ) &=C\cdot p^*(K_{X^{\rm m}}+B^{\rm m}+\bbeta_{X^{\rm m}})\\
                                    &=C\cdot q^*(K_{X^w}+B^w+\bbeta _{X^w})\\                                                                 &=q_*C\cdot (K_{X^w}+B^w+\bbeta _{X^w})>0,
\end{align*}
which is a contradiction.\\

(4) Let $p:W\to X^{\rm m}$ and $q:W\to X^w$ be a common resolution and $P$ a divisor which is $p$-exceptional and not $q$-exceptional, then $P$ is a divisor on $X$ such that $a(P,X,B,\bbeta)<a(P,X^{\rm m},B^{\rm m},\bbeta)=a(P,X^w,B^w,\bbeta)$ where the last equality follows from (3). But then $P$ must be $X\dasharrow X^w$ exceptional, which is a contradiction. Thus $X^{\rm m}\dasharrow X^w$ extracts no divisors and by symmetry $X^{\rm m}\dasharrow X^w$ is an isomorphism in codimension 1.

(5) Let $p:W\to X^{\rm m}$ and $q:W\to X^w$ be a common resolution, then $p^*(K_{X^{\rm m}}+B^{\rm m}+\bbeta _{X^{\rm m}}=q^*(K_{X^w}+B^w+\bbeta _{X^w})$.
Let $C$ be a $p$-exceptional curve, then \[0=p^*(K_{X^{\rm m}}+B^{\rm m}+\bbeta _{X^{\rm m}}\cdot C=q^*(K_{X^w}+B^w+\bbeta _{X^w})\cdot C\] and so $C$ is $q$-exceptional. By the rigidity lemma $X^{\rm m}\to X^w$ is a morphism.

(6) Suppose that $\phi:X\dasharrow X^{\rm m}$ is a log minimal model and $P$ is a $\phi ^{-1}$ exceptional divisor. Then as $(X,B) $ is klt and $P$ is contained in the support of $B^{\rm m}$ with multiplicity 1 (as $B^{\rm m}=\phi_*B+\Ex(\phi^{-1})$), from Part (1) we have 
\[-1<a(P,X,B)=a(P,X^{\rm m},B^{\rm m})=-1,\] which is impossible, and so there are no $\phi ^{-1}$-exceptional divisors, i.e. $\phi $
is a log terminal model.\\

(7) This follows from a standard discrepancy computation which can be verified along the lines of the proof of \cite[Lemma 3.10]{HL21}. 
\end{proof}~\\

\begin{lemma}\label{l-klt}
Let $(X,B+\bbeta)$ be a generalized klt (resp. dlt) pair. If $K_X+B$ is $\mathbb R$-Cartier, %with rational singularities,
then $(X,B)$ is klt (resp. dlt).
\end{lemma}
\begin{proof} Since the statement is local on $X$, we may assume that $X$ is Stein and relatively compact.
Let $f:X'\to X$ be a log resolution and $K_{X'}+B'+\bbeta _{X'}=f^*(K_X+B+\bbeta_X)$,
where $\lfloor B'\rfloor\leq 0$, as $(X,B+\bbeta )$ is generalized klt.
Let $K_{X'}+B^\sharp:=f^*(K_X+B)$. Then \[f^*\bbeta_{X}-\bbeta_{X'}\equiv K_{X'}+B'-f^*(K_X+B)=:E,\] 
where $E\geq 0$ by the negativity lemma.
But then \[B'=E+f^*(K_X+B)-K_{X'}=B^\sharp +E\] and so $\lfloor B^\sharp \rfloor\leq \lfloor B' \rfloor\leq 0$, i.e. $(X,B)$ is klt. The statement about dlt singularities follows similarly.
\end{proof}

\begin{lemma}\label{lem:h11-surjective}
 Let $\phi:X\bir Y$ be a bimeromorphic map of normal compact K\"ahler $\mbQ$-factorial varieties. Assume that one of the following holds
 \begin{enumerate}
     \item[(i)] there is a $\mbQ$-divisor $B$ such that $(X, B)$ dlt and $\phi$ is a  finite sequence of $(K_X+B)$-flips and divisorial contractions, or
     \item[(ii)] $X$ and $Y$ both have strongly $\mbQ$-factorial klt singularities and $\phi^{-1}:Y\bir X$ does not contract any divisor.
 \end{enumerate}
  Then the linear map $\phi_*:H^{1,1}_{\BC}(X)\to H^{1,1}_{\BC}(Y)$ is well defined and surjective. 
 \end{lemma}
 
 \begin{proof}

 In case (i) both of the conclusions follow from repeated use of the cone theorem \cite[Proposition 3.1(a)]{CHP16}. So we will prove case (ii) here. Let $p:W\to X$ and $q:W\to Y$ be a resolution of the graph of $\phi$. Pick $\alpha\in H^{1,1}_{\BC}(X)$.
 We define $\phi _*\alpha=q_*(p^*\alpha )$. 
 It is easy to see that this definition doesn't depend on our choice of $W$. 

To establish surjectivity, let $\gamma\in H^{1,1}_{\BC}(Y)$. Then by \cite[Lemma 2.32]{DH20} there is a $p$-exceptional $\mbR$-divisor $F$ such that $q^*\gamma=p^*\alpha_X+[F]$ for some $\alpha_X\in H^{1,1}_{\BC}(X)$. Since $\phi^{-1}$ does not contract any divisor, $F$ is also $q$-exceptional. Therefore $\phi_*\alpha_X=q_*(p^*\alpha_X)=q_*(q^*\gamma+[F])=\gamma$, and hence $\phi_*$ is surjective.
 \end{proof}

\begin{lemma}\label{lem:exceptional-locus}
   Let $f:X'\to X$ be a proper bimeromorphic morphism of normal compact K\"ahler varieties with strongly $\mbQ$-factorial klt singularities.
 Then $\Ex(f)$ is a pure codimension $1$ subset of $X'$. 
\end{lemma}

\begin{proof}
 We can apply \cite[Lemma 2.32]{DH20} here as the necessary relative MMP in higher dimensions is established in \cite[Theorem 1.4]{DHP22}. Thus for any K\"ahler class $\omega_{X'}$ on $X'$, there is a $f$-exceptional $\mbR$-divisor $F$ so that $\omega_{X'}\num \phi^*\alpha_{X}-F$ for some $\alpha_X\in H^{1,1}_{\BC}(X)$. Then by the negativity lemma, $F$ is effective and $\Supp(F)=\Ex(f)$, and we are done.
\end{proof}

\begin{lemma}\label{lem:small-to-isomorphism}
    Let $\phi:X\bir X'$ be a small  bimeromorphic map over $Y$ of normal compact K\"ahler varieties such that $X$ and $X'$ both have strongly $\mbQ$-factorial klt singularities. Let $\omega\in H^{1,1}_{\BC}(X)$ be nef over $Y$ such that $\omega':=\phi_*\omega\in H^{1,1}_{\BC}(X')$ is  K\"ahler over 
    $Y$. Then $\phi$ is an isomorphism.
\end{lemma}

\begin{proof}
    Let $W$ be the resolution of the graph of $\phi$, and $p:W\to X$ and $q:W\to X'$ are the induced bimeromorphic morphisms. Then by \cite[Lemma 2.32]{DH20} there is a $q$-exceptional $\mbR$-divisor $E$ such that $p^*\omega=q^*\omega'+E$ where $\omega'\in H^{1,1}_{\BC}(X')$. Since $\phi$ is small, from the negativity lemma it follows that $E=0$, i.e. $p^*\omega=q^*\omega'$. If $\phi$ is a not a morphism, then there is a curve $C\subset W$ such that $p_*(C)=0$ but $q_*(C)\neq 0$ and $(f'\circ q)_*(C)=0$, where $f':X'\to Y$ is the given morphism. In particular, $0=p^*\omega\cdot C=q^*\omega'\cdot C=\omega'\cdot q_*(C)>0$,  a contradiction. Thus $\phi$ is a morphism. Then we arrive at a contradiction by Lemma \ref{lem:exceptional-locus} unless $\phi$ is an isomorphism.

\end{proof}

\begin{definition}\cite[Page 3]{Fuj22}\label{def:property-p}
 Let $X$ be a normal analytic variety and $W\subset X$ a fixed compact subset. We say that $W\subset X$ satisfies \emph{Property} \textbf{P} if the following hold:
 \begin{enumerate}
     \item[(P1)] $X$ is a Stein space.
     \item[(P2)] $W$ is a Stein compact subset of $X$.
     \item[(P3)] $\Gamma (W,\OO _X)$ is noetherian (or equivalently, for any open subset $U\subset X$ and any analytic subset $Z$ of $U$, $W\cap Z$ has finitely many connected components).
 \end{enumerate}
 
 A projective morphism $g:S\to T$ between analytic varieties is said to satisfy \emph{Property} \textbf{Q} if $S$ and $T$ are both compact.
\end{definition}

\begin{remark}\label{rmk:property-p}
    Let $X$ be a normal analytic variety and for each point $x\in X$, let $x\in U$ be a Stein open neighborhood. Since $U$ is locally compact, there is a compact neighborhood $x\in K\subset U$ of $x$. Then by \cite[Lemma 2.5]{Fuj22}, its holomorphically convex hull $\widehat{K}$ in $U$ is Stein compact. Note that from \cite[Theorem 2.10]{Fuj22} it follows that $\widehat{K}\subset U$ satisfies Property \textbf{P} if and only if $\Gamma(\widehat{K}, \mcO_U)=\varinjlim_{\widehat{K}\subset V}\Gamma(V, \mcO_V)$, where $V$ is an open subset of $U$, is a noetherian ring. But then from \cite[Lemma 2.16]{Fuj22} we see that there is a Stein compact subset $L$ such that $x\in \widehat{K}\subset L\subset U$ such that $\Gamma(L, \mcO_U)$ is noetherian. In particular, every point $x\in X$ has a Stein open neighborhood $U$ and a Stein compact subset $x\in L\subset U$ such that $U$ satisfies Property \textbf{P}. 
\end{remark}~\\

 \begin{theorem}\label{t-gkltlocal}
 Let $(X,B+\bbeta)$ be a generalized klt pair, where $X$ is a relatively compact analytic variety. Then the following hold locally over $X$. 
 \begin{enumerate}
    \item $X$ has rational singularities.
    \item There exists a small bimeromorphic morphism $\mu :X^\sharp \to X$ such that $X^\sharp$ is strongly $\mathbb Q$-factorial.
    \item If $K_{X^\sharp}+B^\sharp+\bbeta _{X^\sharp}=\mu ^*(K_X+B+\bbeta_X)$, then $\bbeta _{X^\sharp}\equiv _X \Delta ^\sharp$ so that $(X^\sharp,B^\sharp+\Delta ^\sharp)$ is klt, and
    \item if $\Delta =\mu _*\Delta  ^\sharp$, then $(X,B+\Delta)$ is klt.
   \end{enumerate}
 \end{theorem}
 \begin{proof} 
(1) immediately follows from (4) and \cite[Corollary 11.14]{Kol97}.\\

 (2-3) From Remark \ref{rmk:property-p}, {it follows that for any $x\in X$ there is a Stein compact subset $x\in W\subset X$ such that $X$ satisfies Property \textbf{P}. In what follows we work locally around $W$ i.e. we repeatedly shrink $X$ to a neighborhood of $W$ (without further mention).} Let $\nu :X'\to X$ be a projective log resolution of $(X,B+\bbeta)$ and write $K_{X'}+B'+\bbeta _{X'}=\nu ^*(K_X+B+\bbeta_X)$. Let $E=\Ex(\nu)$, and for $0<\epsilon \ll 1$ define $B^*:=(B')^{>0}+\epsilon E$ and $F:=(B')^{<0}+\epsilon E$. Then $K_{X'}+B^*+\bbeta _{X'}\equiv \nu ^*(K_X+B+\beta_X)+F$, where the support of $F$ equals the set of all $\nu$-exceptional divisors, and $(X',B^*+\bbeta _{X'})$ is generalized klt.
 In particular, $\bbeta_{X'}\equiv _X F-(K_{X'}+B^*)$, where $F-(K_{X'}+B^*)$ is an $\mathbb R$-divisor, nef over $X$. As $\nu$ is projective and $X$ is Stein, we may assume that $F-(K_{X'}+B^*)$ is big and nef (over $X$). But then $\bbeta _{X'}\equiv_X \Delta'$, where $\Delta' \>0$ is an effective $\mathbb R$-divisor such that
$(X',B^*+\Delta ')$ is klt.

We may therefore run the relative $(K_{X'}+B^*+\Delta ')$-MMP (see \cite[Theorem 1.4]{DHP22} and \cite[Theorem 1.8]{Fuj22}) and hence we may assume that we have a bimeromorphic map $\psi: X'\dasharrow X^\sharp$ such that if $F^\sharp =\psi _*F$, $B^\sharp =\psi _*B^*$, $\bbeta_{X^\sharp} =\psi _*\bbeta_{X'}$  and $\Delta ^\sharp =\psi _*\Delta '$, then
 \[ F^\sharp \equiv _X K_{X^\sharp }+B^\sharp +\bbeta_{X^\sharp} \equiv_X K_{X^\sharp }+B^\sharp +\Delta ^\sharp \] is nef over $X$ so that $F^\sharp =0$ by the negativity lemma. Therefore $\mu:X^\sharp\to X$ is a small bimeromorphic morphism, $B^\sharp=\mu ^{-1}_*B$ and $X^\sharp$ is $\mathbb Q$-factorial. 
 Clearly $(X^\sharp,B^\sharp+\Delta ^\sharp)$ is klt.
 Note that each step of the above MMP preserves the 
 numerical equivalence $\bbeta _{X^\sharp}\equiv _X\Delta ^\sharp$, and in particular $K_{X^\sharp}+B^\sharp+\bbeta _{X^\sharp}=\mu ^{-1}_*(K_X+B+\bbeta_X)$.
 
 (4) By the Base-point free theorem \cite[Theorem 8.1]{Fuj22}, we have (locally over $X$) that $K_{X^\sharp}+B^\sharp+\Delta ^\sharp \sim _{\mathbb Q,X}0 $ and the claim follows.

 \end{proof}

We have following immediate corollary.
\begin{lemma}
 \label{l-Q-fac}
Let $(X,B+\bbeta)$ be a generalized klt (resp. dlt) pair, where $X$ is compact analytic surface. Then $X$ is locally $\mathbb Q$-factorial with rational singularities, and $(X,B)$ is klt (resp. dlt).
\end{lemma}

\subsection{Existence of Flips for Generalized Pairs}\label{subsec:flips} In this subsection we prove the existence of flips for generalized klt pairs in dimension $3$.  

\begin{theorem}\label{t-Stein-flip+}
  Let $(X/S,B+\bbeta)$ be a generalized klt K\"ahler $3$-fold pair, 
  such that $K_X+B$ is $\R$-Cartier, and 
  $f:X\to Z$ is a $(K_X+B+\bbeta_X)$-negative small bimeromorphic morphism  over $S$. 
  Then $f$ is locally projective, the log canonical model  $f^+:X^+\to Z$ for $(X,B+\bbeta)$ over $Z$ exists, $f^+$ is a small and bimeromorphic morphism, and there is a $f$-exceptional rational curve $C$ such that $0>(K_X+B+\bbeta_X)\cdot C\geq -6$. \end{theorem}
  
  \begin{proof} Let $\{C_i\}_{i\in I}$ be the (finite) set of curves contracted by $f$ and $C:=\cup_{i\in I} C_i$. Assume for simplicity that $C$ is connected. It suffices to construct the
flip locally around $z=f(C)\subset Z$.  Let $z\in W\subset Z$ be a relatively compact Stein open subset. 
Shrinking $W$, we may assume that for every curve $C_i$, there is a Cartier divisor $D_i$ on $X_W:=f^{-1}W$ that intersects $C_i$ transversely and does not intersect $C_j$ for $j\ne i$. To construct $D_i$, pick a general point $x_i$ on $C_i$ and a sufficiently small neighborhood $x_i\in U_i\subset X$. We identify $x_i\in U_i$ with a locally closed analytic subvariety of $\mbC ^N$ and take the divisor $D_i$ given by a general hyperplane through $x_i$. Shrinking $W$ and intersecting $D_i$ with $X_W$, we may assume that each $D_i$ is a subvariety of $X_W$. It then follows that if $D=\sum d_iD_i$, where $d_i=[\bbeta_X] \cdot C_i$, then $D\equiv _W\bbeta_X$. 

Now let $\nu :X'\to X$ be a log resolution of the generalized pair $(X,B+\bbeta)$. Since $K_X+B$ is $\R$-Cartier, we have $[\bbeta_X]\in H^{1,1}_{\rm BC}(X)$, and so by Remark \ref{r-gp}  we may write $-E\equiv \bbeta _{X'}-\nu^*\bbeta_X$ for some $\nu$-exceptional $\R$-divisor $E$ on $X'$. Let $D':=\nu^*D-E|_{X'_W}\equiv _W \bbeta _{X'}|_{X'_W}$. We may assume that $\nu :X'_W\to W$ is projective (via Hironaka's Chow lemma  \cite[Corollary 2]{Hir75}). Since $D'$ is nef and big over $X_W$, replacing $D'$ by an $\R$-linearly equivalent divisor, we may assume that $(X'_W,B'_W+D')$ is sub-klt and hence $(X_W,B_W+D)$ is klt, since $K_{X'_W}+B'_W+D'\num \nu^*(K_{X_W}+B_W+D)$. But then the required log canonical model $ X^+_W$ exists (see \cite[Theorem 4.3]{CHP16} or \cite[Theorem 1.3]{DHP22}). 
In particular, $f^+:X^+_W\to W$ is small and projective, and since $-(K_{X_W}+B_W+D)$ is ample over $W$, $f$ is locally projective.  The existence of a $f$-exceptional rational curve $C\subset X_W$ such that $0>(K_X+B+\bbeta_X)\cdot C=(K_{X_W}+B_W+D)\cdot C\geq -6$ now  follows from \cite[Theorem 4.2]{DO23}.
  \end{proof}

  As an easy corollary, we will prove the existence of flips. Recall that if $(X/S,B+\bbeta)$ is a $\Q$-factorial compact K\"ahler generalized klt 3-fold pair, then a $(K_X+B+\bbeta_X)$-flipping contraction over $S$ is a small bimeromorphic morphism $f:X\to Z$ over $S$ such that
 $\rho (X/Z)=1$, and $-(K_X+B+\bbeta_X)$ is K\"ahler over $Z$. By definition, the flip of $f:X\to Z$, if it exists, is a small bimeromorphic morphism $f^+:X^+\to Z$ over $S$ such that $X^+$ is K\"ahler over $S$, and $K_{X^+}+B^++\bbeta_{X^+}$ is K\"ahler over $Z$. We need the following lemma first.
 \begin{lemma}\label{lem:uniqueness-of-flip}
     Let $X$ be a normal $\mbQ$-factorial (resp. strongly $\mbQ$-factorial) compact K\"ahler $3$-fold, $f:X\to Z$ a $(K_X+B+\bbeta_X)$-flipping contraction of a generalized klt pair over $S$, and $f^+:X^+\to Z$ the corresponding flip, then
     \begin{enumerate}
         \item $f^+:X^+\to Z$ is uniquely determined,
         \item $X^+$ is $\Q$-factorial (resp. strongly $\mbQ$-factorial), and
         \item $\rho (X^+/Z)=1$, when $X$ and $X^+$ are strongly $\mbQ$-factorial.
     \end{enumerate}
 \end{lemma}
 \begin{proof}
     Suppose that $f':X'\to Z$ is another flip of $f:X\to Z$, then $X^+\dasharrow X'$ is a small bimeromorphic map over $Z$. Let $Y$ be the normalization of the graph and $p:Y\to X^+$ and $q:Y\to X'$ are the induced morphisms, then from the negativity lemma it follows easily that $q^*(K_{X^+}+B^++\bbeta_{X^+})=p^*(K_{X'}+B'+\bbeta_{X'})$. Let $C\subset Y$ be a $p$-exceptional curve. Then $q_*C\ne 0$ and $(f'\circ q)_*C=0$. Thus we have
     \[0< C\cdot q^*(K_{X^+}+B^++\bbeta_{X^+})=C\cdot p^*(K_{X'}+B'+\beta_{X'})= 0\] which is a contradiction.
     Therefore, there are no such curves and hence $X^+\dasharrow X'$ is a morphism. Similarly, it follows that $X'\dasharrow X^+$ is a morphism and hence $X^+\cong X'$. In particular, (1) holds.

     Let $G^+$ be a prime Weil divisor on $X^+$ and $G$ its strict transform on $X$. Then $G$ is $\mbQ$-Cartier, as $X$ is $\mbQ$-factorial. For any point $p\in X^+$ we must show that there is a neighborhood of $p$ on which $G^+$ is $\Q$-Cartier.  This is clear if $p$ is not contained in the flipped locus ${\rm Ex}(f^+)$, so assume that $p\in {\rm Ex}(f^+)$ and let $q=f^+(p)$. Working locally over a neighborhood $q\in W\subset Z$ as in the proof of Theorem \ref{t-Stein-flip+}, we may assume that $K_{X_W}+B_W+D$ is klt for some effective $\mbR$-divisor $D\geq 0$ on $X_W$ such that $D\equiv _W \bbeta_X|_{X_W}$ and that $X^+\to W$ is the relative log canonical model for $K_{X_W}+B_W+D$. 
     Since $-(K_{X_W}+B_W+D)$ is ample over $W$, we may pick an effective $\R$-divisor $0\<H\sim _{\R,W}\epsilon G_W-\frac 12(K_{X_W}+B_W+D)$ for $0<\epsilon \ll 1$ such that $(X_W, B_W+D+H)$ is klt and $-(K_{X_W}+B_W+D+H)$ is ample over $W$.
     Then $X^+\to W$ is also the relative log canonical model for $K_{X_W}+B_W+D+H$ over $W$ and so $K_{X^+_W}+B^+_W+D^++H^+\sim _{\R,W} \frac 1 2(K_{X^+_W}+B^+_W+D^+)+\epsilon G^+_W$ is $\R$-Cartier for $0<\epsilon \ll 1$, and hence $G^+_W$ is $\Q$-Cartier and (2) is proven. For the strongly $\Q$-factorial case, we refer the reader to \cite[Lemma 2.5]{DH20}. 
     
     (3) now follows from \cite[Lemma 2.32]{DH20}.

 \end{proof}
  \begin{corollary}\label{c-3fold-flips}
      Let $(X/S,B+\bbeta)$ be a $\Q$-factorial compact K\"ahler generalized klt 3-fold pair, and $f:X\to Z$  a $(K_X+B+\bbeta_X)$-flipping contraction over $S$. Then $f$ is locally projective, the flip $f^+:X\to Z$ for $K_X+B+\bbeta_X$ over $Z$ exists (and is unique), and there is an $f$-exceptional rational curve $C$ such that $0>(K_X+B+\bbeta_X)\cdot C\geq -6$.
  \end{corollary}
  \begin{proof}
      Follows immediate from Theorem \ref{t-Stein-flip+} and Lemma \ref{lem:uniqueness-of-flip}.
  \end{proof}

\begin{proof}[Proof of Theorem \ref{thm:3-dim-flips}]
   This follows from Corollary \ref{c-3fold-flips}.   
\end{proof}~\\

\begin{lemma}\label{l-kmmp}
    Let $\pi :X\to S$ be a proper morphism of compact complex varieties such that $X$ is K\"ahler of dimension at most $3$.
    If $(X,B+\bbeta)$ is a generalized dlt pair and $\phi :X\dasharrow X'$ is a $(K_X+B+\bbeta_X)$-flip, flipping contraction or divisorial contraction, then $X'$ is K\"ahler.
\end{lemma}
\begin{proof}
    Let $\omega$ be a K\"ahler form such that $\gamma =K_X+B+\bbeta_X +\omega$ is a supporting hyperplane for the $(K_X+B+\bbeta_X)$-negative extremal ray. If $f:X\to Z$ is the corresponding contraction, then $\gamma _Z=K_Z+B_Z+\bbeta _Z+\omega _Z=f_*(K_X+B+\bbeta_X +\omega)$ is generalized dlt and hence $Z$ has rational singularities. But then, by the proof of \cite[Corollary 3.1]{CHP16}, $\gamma _Z$ is K\"ahler (over $S$).
    Suppose now that $f:X\to Z$ is a flipping contraction and let $f^+:X^+\to Z$ be the flip, then $-\omega ^+=-\phi _*\omega$ is K\"ahler over $Z$ and so, for any $0<\epsilon \ll 1$, \[K_{X^+}+B_{X^+}+\bbeta _{X^+}+(1-\epsilon)\omega ^+\equiv {f^+}^*\gamma _Z-\epsilon \omega ^+\]
    is K\"ahler on $X^+$.
\end{proof}

\subsection{Generalized Surface MMP}\label{s-gsmmp}
We begin by recalling the following well known fact.
\begin{lemma}\label{l-nnef}
If $\alpha\in H^{1,1}_{\rm BC}(X)$ is pseudo-effective but not nef on a normal compact K\"ahler surface $X$, then $\int _C\alpha <0$ for some curve $C\subset X$.\end{lemma}
\begin{proof}
Follows immediately from \cite[Theorem 2.36]{DHP22}.
\end{proof}

\begin{lemma}\label{l-nef}
Let $f:X\to Y$ be a proper bimeromorphic morphism of normal compact K\"ahler surfaces with rational singularities. If $\alpha\in H^{1,1}_{\rm BC}(X)$  is nef and $\alpha _Y:=f_*\alpha$, then $\alpha _Y$ has local potentials and the class $\alpha _Y\in H^{1,1}_{\rm BC}(Y)$ is nef.
\end{lemma}

\begin{proof}
Passing to a resolution of singularities of $X$ we may assume that $X$ is smooth. Now recall that by the Hodge index theorem the intersection matrix of the set of all $f$-exceptional curves is a negative definite matrix. Therefore there is an $f$-exceptional $\mbR$-divisor $E$ on $X$ such that $\alpha+E\equiv _Y 0$.
By \cite[Lemma 3.3]{HP16}, $\alpha +E=f^*\alpha _Y$ for some $\alpha_Y\in H^{1,1}_{\BC}(Y)$, and thus $\alpha_Y=f_*(f^*\alpha _Y)=f_*(\alpha+E)=f_*\alpha$. From the the negativity lemma it follows that $E\geq 0$. Thus $\alpha _Y$ is pseudo-effective, and so by Lemma \ref{l-nnef}, it suffices to check that $\alpha _Y|_C$ is pseudo-effective, i.e. that $\int _C \alpha_Y \geq 0$ for all curves $C\subset Y$. If $C'=f^{-1}_* C$, then we have
\[\int _C \alpha_Y =\int _{C'} \alpha +(E\cdot C')\geq 0,\]
since $C'$ is not contained in the support of $E$ and $\alpha $ is nef.
\end{proof}~\\

An immediate corollary of this lemma is the following. 
\begin{corollary}\label{cor:g-beta-nef}
 If $(X, B+\bbeta)$ is a compact generalized lc pair such that $X$ is a compact K\"ahler surface with rational singularities, then $\bbeta_X$ has local potentials on $X$ and $[\bbeta_X]\in H^{1,1}_{\BC}(X)$ is nef.
\end{corollary}~\\

\begin{definition}\label{def:ns-group}
    Let $X$ be a compact analytic variety. The N\'eron-Severi $\mbR$-vector space of $X$ is defined as:
    \[ 
    \NS(X):=\Im(\Pic(X)\to H^2(X,\mbR)), \qquad \NS(X)_\mbR=\NS(X)\otimes _{\mathbb Z}\R.
    \]
\end{definition}

\begin{lemma}\label{lem:ns-h11}
    Let $X$ be a normal compact K\"ahler variety with rational singularities. If $H^2(X, \mcO_X)=0$, then $X$ is projective and $\NS(X)_\mbR=H^{1,1}_{\BC}(X)$. 
\end{lemma}

\begin{proof} This is well know, see e.g. \cite[Proposition 5.13]{Graf18}.

\end{proof}

Next, we will establish the cone theorem and existence of minimal models (and Mori fiber spaces) for generalized pairs in dimension $2$. This will be used in the rest of the article without further reference. 

\begin{lemma}\label{lem:dlt-surface-cone}
    Let $(X, B)$ be a dlt pair, where $X$ is a compact K\"ahler surface. Then there exists at most countably many rational curves $\{\Gamma_i\}_{i\in I}$ such that $0<-(K_X+B)\cdot\Gamma_i\<4$ and
    \[ 
    \NA(X)=\NA(X)_{(K_X+B)\>0}+\sum_{i\in I}\mbR^+\cdot[\Gamma_i].
    \]    
\end{lemma}

\begin{proof}
    From Lemma \ref{l-Q-fac} it follows that $X$ has $\mbQ$-factorial rational singularities. First assume that $K_X+B$ is pseudo-effective. Then from Lemma \ref{l-nnef} it follows that $K_X+B$ is nef if and only if $(K_X+B)\cdot C\>0$ for every curve $C\subset X$. Let $K_X+B\num \sum_{i\in I}\lambda_iC_i+\beta$ be the Boucksom-Zariski decomposition as in \cite{Bou04}, where $\lambda_i\>0$ for all $i\in I\subset \mbN$ (a finite subset) and  $\beta\cdot C\>0$ for every curve $C\subset X$. Now if $K_X+B$ is not nef, then there is a curve $\Gamma\subset X$ such that $(K_X+B)\cdot\Gamma<0$. This implies that $(\sum_{i\in I}\lambda_i C_i)\cdot\Gamma<0$, in particular, $\Gamma=C_i$ for some $i\in I$ and $\Gamma^2<0$. Then the rest of proof works similarly as in the proof of \cite[Theorem 1.31]{DO23}. The length bound $0>(K_X+B)\cdot\Gamma\>-4$ follows from \cite[Theorem 1.23]{DO23}.
    
    Now assume that $K_X+B$ is not pseudo-effective. Then $K_X$ is not pseudo-effective. Let $\nu:\widetilde X\to X$ be the minimal resolution of singularities of $X$.  Then $K_{\tilde X}$ is not pseudo-effective,  and hence $H^2(\tilde X,\mcO_{\tilde X})=H^0(\tilde X, K_{\tilde X})^*=0$. Since $X$ has rational singularities, we also have $H^2( X,\mcO_{ X})\cong H^2(\tilde X, 
    \mcO_{\tilde X})=0$. Thus by Lemma \ref{lem:ns-h11}, $X$ is projective with $\NS(X)_\mbR=H^{1,1}_{\BC}(X)$. Consequently, we have $\NE(X)=\NA(X)$, and the cone theorem is well known in this case.

\end{proof}

\begin{definition}\label{def:null-locus}
    Let $X$ be a normal compact K\"ahler analytic variety and $\alpha\in H^{1,1}_{\BC}(X)$ a nef and big class. We define the null locus $\Null(\alpha)$ of $\alpha$ as the union of closed analytic subvarieties $V\subset X$
 such that $\dim V>0$ and $\alpha^{\dim V}\cdot V=0$, i.e.
 
 \[\Null(\alpha)=\bigcup_{\substack{V\subset X,\\ \dim V>0,\\ \alpha^{\dim V}\cdot V=0}} V.\]
\end{definition}
By \cite[Theorem 4.21]{HP24} it follows that $\Null(\alpha)$ is a closed analytic subset of $X$.

 Next we prove the transcendental base-point free theorem in dimension $2$,
 similar to the one in \cite[Theorem 1.7]{DH20}. 

\begin{theorem}\label{thm:tbft}
     Let $(X, B)$ be a generalized pair, where $X$ is a compact K\"ahler surface, and $\alpha\in H^{1,1}_{\BC}(X)$ a nef class. Assume that one of the following conditions hold: 
     \begin{enumerate}
         \item $(X, B)$ is klt and $\alpha-(K_X+B)$ is nef and big, or
         \item $(X, B)$ is dlt and $\alpha-(K_X+B)$ is K\"ahler.
     \end{enumerate}
     Then there is a projective surjective morphism with connected fibers $f:X\to Y$ to a normal compact K\"ahler variety $Y$ with $\mbQ$-factorial rational singularities and a K\"ahler class $\omega_Y\in H^{1,1}_{\BC}(Y)$ such that $\alpha=f^*\omega_Y$. 
 \end{theorem}

\begin{proof}
    By Lemma \ref{l-Q-fac}, $X$ has $\mbQ$-factorial rational singularities. Now if $(X, B)$ is dlt, then $(X, (1-\eps)B)$ is klt for any $0<\eps<1$. Moreover, in this case $\alpha-(K_X+(1-\eps)B)$ is K\"ahler for $0<\eps\ll 1$. Therefore in both cases (1) and (2) above we may assume that $(X, B)$ is klt and $\alpha-(K_X+B)$ is nef and big. Now we will consider two cases depending on whether $K_X+B$ is pseudo-effective or not. If $K_X+B$ is not pseudo-effective, then from the minimal resolution $\nu:X'\to X$ we see that $K_{X'}$ is not pseudo-effective, and hence $H^2(X, \mcO_X)\cong H^2(X', \mcO_{X'})\cong H^0(X', K_{X'})^*=0$, where the first isomorphism holds due to the rational singularities of $X$. Then by Lemma \ref{lem:ns-h11}, $X$ is projective and $\NS(X)_{\mbR}=H^{1,1}_{\BC}(X)$. In this case $\alpha$ is represented by a nef $\mbR$-Cartier divisor and the contraction theorem follows from the standard base-point free theorem for $\mbR$-Cartier divisors, for example see \cite[Exercise 5.9]{HK10}. So from now on we will assume that $K_X+B$ is pseudo-effective.

  Let $\omega$ be a K\"ahler class such that $\alpha-(K_X+B)-\omega$ is also big. Now let  $\alpha-(K_X+B)-\omega\equiv N+\gamma$ be the pushforward of the Boucksom-Zariski decomposition (see \mbox{\cite[Definition 3.7]{Bou04}}) of the pullback of $\alpha-(K_X+B)-\omega$ on some resolution of $X$, where $N$ is an effective $\mbR$-divisor and $\gamma$ is an $(1,1)$ class. Then from \mbox{\cite[Proposition 2.4]{Bou04}} and Lemma \mbox{\ref{l-nef}} it follows that $\gamma$ is a nef class.

    Pick $\delta >0$ such that $(X,B+\delta N)$ is klt. We write \[\alpha =K_X+B+\delta N+(1-\delta)(\alpha -(K_X+B))+\delta \omega+\delta\gamma =K_X+\Delta +\beta,\] where
    $\Delta:=B+\delta N$, and $\beta:=(1-\delta)(\alpha -(K_X+B))+\delta(\omega+\gamma)$ is a K\"ahler class. Recall that we have $\alpha=K_X+\Delta +\beta$ is nef and big, and so the null locus $\Null(\alpha)$ is an $1$-dimensional analytic subset of $X$. We will run $\alpha$-trivial $(K_X+\Delta)$-MMP. Note that since $\alpha$ is nef but not K\"ahler, $\alpha^\bot\cap\NA(X)$ is an extremal face of $\NA(X)$. Therefore, if there is a curve $C\subset X$ such that $\alpha\cdot C=0$, then $(K_X+\Delta )\cdot C=-\beta \cdot C<0$, and so, by Lemma \ref{lem:dlt-surface-cone},  $\alpha^\bot\cap\NA(X)$ contains a $(K_X+\Delta)$-negative extremal ray $R$ of $\NA(X)$. By \cite[Theorem 8.4]{Fuj19} or \cite[Theorem 1.32]{DO23} we can contract this ray $R$ and obtain a compact K\"ahler klt surface pair $(X', \Delta ')$. Note that if $g:X\to X'$ is the contraction morphism, then by \cite[Lemma 3.3]{HP16}, there is a nef and big class $\alpha'\in H^{1,1}_{\BC}(X')$ such that $\alpha=g^*\alpha'$.
    We claim that $\beta':=g_*\beta\equiv \alpha '-(K_{X'}+\Delta')$ is K\"ahler. Indeed, notice that since $\beta'$ is big and nef (by Lemma \ref{l-nef}), by \cite[Theorem 2.29]{DHP22} it is enough to show that $\beta'\cdot C>0$ for all curves in $X'$. Now if $\Gamma$ is the unique $g$-exceptional curve, then by the negativity lemma we have $\beta-g^*\beta'=g^*(K_{X'}+\Delta')-(K_X+\Delta)=-aE$ for some $a>0$. Therefore for any curve $C\subset X'$ if $\bar C\subset X'$ is strict transform of $C$, then we have $\beta'\cdot C=g^*\beta'\cdot \bar C=(\beta+aE)\cdot\bar C>0$, and hence $\beta'$ is K\"ahler.
    
   Thus continuing the above process finitely many times (as the Picard number of $X$ drops after each contraction), we arrive at a klt surface pair, say $(\bar X, \bar \Delta )$ with composite morphsim 
   $f:X\to \bar X$ such that $\bar\alpha:=f_*\alpha$ contains no $\alpha$-trivial curves, and hence $\bar\alpha$ is a K\"ahler class. Moreover, since every step of the above process is $\alpha$-trivial, we also have $\alpha=f^*\bar\alpha$. Finally, since $(\bar X, \bar \Delta)$ is klt, $\bar X$ has rational singularities. Thus we are done by setting $Y:=\bar X$ and $\omega_Y:=\bar \alpha$. 
\end{proof}

\begin{corollary}\label{c-s-gcone}
Let  $(X,B+\bbeta)$ be a  generalized dlt pair, where $X$ is a compact K\"ahler surface. Then the following holds:
\begin{enumerate}    
\item There are at most countably many curves $\{\Gamma _i\}_{i\in I}$ such that $0>(K_X+B+\bbeta _X)\cdot \Gamma _i\geq -4$ and  \[\overline{\rm NA}(X)=\overline{\rm NA}(X)_{K_X+B+\bbeta _X\geq 0}+\sum _{i\in I}\mathbb R [\Gamma _i].\] 

\item If $F$ is a face spanned by a finite set of $(K_X+B+\bbeta_X)$-negative extremal rays, then there is a contraction $f:X\to Y$ contracting curves $C\subset  X$ if and only if $[C]\in F$, and either $Y$ is a point, or a smooth projective curve or a normal $\mbQ$-factorial surface with rational singularities.

\item If $(X,B+\bbeta)$ is a generalized klt and $B+\bbeta_X$ or $K_X+B+\bbeta_X$ is big, then $I$ is finite.  \end{enumerate} \end{corollary}

 \begin{proof}
 (1) By Lemma \ref{l-klt}, $(X,B)$ is dlt with rational $\Q$-factorial singularities.
 By Corollary \ref{cor:g-beta-nef}, $\bbeta_X$ is nef and so $\overline{\rm NA}(X)_{K_X+B\geq 0}\subset \overline{\rm NA}(X)_{K_X+B+\bbeta_X \geq 0}$.
 Thus by Lemma \ref{lem:dlt-surface-cone} we have
 \[\overline{\rm NA}(X)=\overline{\rm NA}(X)_{K_X+B \geq 0}+\sum _{i\in I}\mathbb R [\Gamma _i]=\overline{\rm NA}(X)_{K_X+B+\bbeta_X \geq 0}+\sum _{i\in I}\mathbb R [\Gamma _i].\]

 (2) 
Observe that $F$ is a $(K_X+B)$-negative extremal face of $\NA(X)$, as $[\bbeta_X]$ is nef by Corollary \ref{cor:g-beta-nef}. Then by Lemma \ref{lem:dlt-surface-cone} and some standard arguments (for example, see the proof of \cite[Corollary 6.9]{Deb01}) it follows that  $F$ has a nef supporting class, say $\alpha\in H^{1,1}_{\BC}(X)$ such that $\alpha^\bot\cap\NA(X)=F$ and $\alpha=K_X+B+\omega$ for some K\"ahler class $\omega$. Then by Theorem \ref{thm:tbft}, we can contract this face and the contraction satisfies the required properties. \\

 (3) We claim that if $\psi \in H^{1,1}_{\rm BC}(X)$ is a  big class,
 then there are at most finitely many curves $C\subset X$ such that $\int _C\psi<0$. To see this, note that for some K\"ahler form $\omega$, the class $[\psi -\omega]$ is still big. Let $\psi-\omega\num Z+\gamma$ be the pushforward of the Boucksom-Zariski decomposition of the pullback of $\psi-\omega$ on some resolution of $X$, where $Z\geq 0$ is an effective $\mbR$-divisor and $\gamma$ is a nef class (see \cite[Proposition 2.4]{Bou04} and Lemma \ref{l-nnef}). But then one sees that if $\int _C\psi<0$, then $C$ is contained in the support of $Z$.
 Thus, if $K_X+B+\bbeta_X$ is big, then the claim immediately holds. 
 
 Suppose now that $B+\bbeta_X$ is big, then we may write $B+\bbeta_X\num Z+\omega +\gamma$ as above. Thus \[ B+\bbeta_X\equiv \left((1-\epsilon )B+\epsilon Z\right)+((1-\epsilon )\bbeta_X +\epsilon (\omega +\gamma))\] where $(X,(1-\epsilon )B+\epsilon Z)$ is klt and  $(1-\epsilon )\bbeta_X +\epsilon (\omega +\gamma)$ is K\"ahler for all $0<\epsilon \ll 1$. The finiteness of $K_X+B+\bbeta_X$ negative extremal rays now follows from a standard argument.
 \end{proof}~\\

\begin{lemma}\label{l-mmps}
Let $(X,B+\bbeta)$ be a generalized  klt (resp. dlt) pair, where $X$ is a compact K\"ahler surface. Then we can run the $(K_X+B+\bbeta_X)$-MMP 
\[X=X_0\to X_1\to \ldots \to X_n\]
so that:
 \begin{enumerate}
     \item each $(X_i,B_i+\bbeta _{X_i})$ is a generalized  klt (resp. dlt) K\"ahler surface with $\mathbb Q$-factorial rational singularities,
     \item if $K_X+B+\bbeta_X$ is pseudo-effective, then $K_{X_n}+B_n+\bbeta _{X_n}$ is nef, and
     \item if $K_X+B+\bbeta_X$ is not pseudo-effective, then there is a $(K_{X_n}+B_n+\bbeta _{X_n})$-Mori fiber space $f:X_n\to Z$.
 \end{enumerate}\end{lemma}

 \begin{proof}
 By repeated application of Corollary \ref{c-s-gcone} we can run the $(K_X+B+\bbeta_X)$-MMP. Since the Picard number $\rho (X)$ drops by 1 with each step of the MMP, this process ends after finitely many steps $X\to X_1\to \ldots \to X_n$. Note that this MMP will end either with a minimal model so that $K_{X_n}+B_n+\bbeta_{X_n}$ is nef, or a Mori fiber space $f:X_n\to Z$ such that $-(K_{X_n}+B_n+\bbeta _{X_n})$ is relatively K\"ahler.
 
 (1) This follows easily from the negativity lemma. 
 
 (2)  If $f:X_n\to Z$ is a $(K_{X_n}+B_n+\bbeta _{X_n})$-Mori fiber space, then the general fibers of $f$ intersect $K_{X_n}+B_n+\bbeta _{X_n}$ negatively. However, since the composite morphism $\phi:X\to X'$ is an isomorphism over the general fiber of $f$, it follows that $K_{X}+B+\bbeta _{X}$ intersects the general fiber negatively and hence is not pseudo-effective.

 (3) If $K_{X_n}+B_n+\bbeta_{X_n}$ is nef, and if $\phi:X\to X_n$ is the composite morphism, then from the negativity lemma it follows that $K_X+B+\bbeta_X$ is pseudo-effective.

 \end{proof}~\\

 \begin{theorem}\label{t-lcms}
 Let $(X,B+\bbeta)$ be a generalized klt pair, where $X$ is a compact K\"ahler surface. If  $K_X+B+\bbeta_X$ is big, then $(X,B+\bbeta )$ has a log canonical model.
 \end{theorem}
 
 \begin{proof}
 By running a $(K_X+B+\bbeta_X)$-MMP, we may assume that $\alpha=K_X+B+\bbeta_X$ is nef and big (see Lemma \ref{l-mmps}).  Choose a K\"ahler class $\omega$ such that $K_X+B+\bbeta_X-\omega$ is also a big class. Let $K_X+B+\bbeta_{X}-\omega\equiv D+\gamma$ be the pushforward of the Boucksom-Zariski decomposition of the  pullback of $K_X+B+\bbeta_{X}-\omega$ on some resolution of $X$, where $D$ is an effective $\mbR$-divisor, and $\gamma$ is a pseudo-effective $(1,1)$ class. From \mbox{\cite[Proposition 2.4]{Bou04}} and Lemma \mbox{\ref{l-nef}} it follows that $\gamma$ is a nef class.

 Now choose $0<\ve\ll 1$ such that $(X, B+\ve D)$ is klt. Then 
 \[ (1+\ve)\alpha=(K_X+B+\ve D+\bbeta_X)+\ve(\gamma+\omega). \] 

Now since $\alpha$ is nef and big, and $X$ is a surface, it follows that $\Null(\alpha)$ is $1$-dimensional; in particular, $\Null(\alpha)$ contains finitely many curves. If $C\subset \Null(\alpha)$ is a curve, then $\alpha\cdot C=0$ and thus from the above equation we have $(K_X+B+\ve D+\bbeta_X)\cdot C<0$, and hence $(K_X+B+\ve D)\cdot C<0$ by Corollary \ref{cor:g-beta-nef}. 
Since $(X, B+\ve D)$ is klt, this curve $C$ can be contracted. Repeating this process finitely many times (since $\Null(\alpha)$ contains finitely many curves) we obtain a projective bimeromorphic morphism $f:X\to Z$ to a normal compact surface $Z$ with rational singularities such that $\alpha=f^*\alpha_Z$ and $\Null(\alpha_Z)=\emptyset$, where $\alpha_Z:=f_*(K_X+B+\bbeta_X)=:K_Z+B_Z+\bbeta_Z$. Then from \cite[Theorem 2.29]{DHP22} it follows that $\alpha_Z$ is a K\"ahler class. Thus $(Z, B_Z+\bbeta _Z)$ is the log canonical model of $(X, B+\bbeta _X )$.

 \end{proof}
 \begin{remark}
     Note that by \cite[Example 6.2]{LP20}, it is not the case that all generalized pairs have a good minimal model, however it is known that if $\beta$ is an $\mathbb R$-divisor and $K_X+B$ is pseudo-effective, then good minimal models exist \cite[Corollary C]{LP20}. It would be interesting to know if good minimal models exist for generalized klt  K\"ahler surface pairs $(X,B+\bbeta)$ such that $K_X+B$ is pseudo-effective and $[\bbeta_X] \in H^{1,1}_{\rm BC}(X)$.
 \end{remark}

 \subsection{Relative MMP for $3$-Folds} Using \cite[Theorem 5.2]{DHP22} we will show that we can run a relative MMP for \textit{proper} morphisms between K\"ahler varieties.
 
 \begin{theorem}\label{thm:relative-mmp}
  Let $(X, B)$ be a $\mbQ$-factorial dlt pair, where $X$ is a compact K\"ahler $3$-fold. Let $f:X\to Z$ be a proper morphism to a normal compact K\"ahler variety. Then we can run a $(K_X+B)$-MMP over $Z$ which terminates with either a log terminal model over $Z$ or a Mori fiber space over $Z$. 
 \end{theorem}
 
 \begin{proof}

 Let $\omega_Z$ be a K\"ahler class on $Z$. We may assume that $K_X+B$ is not nef over $Z$. Then $K_X+B+tf^*\omega_Z$ is not nef on $X$ for any $t\>0$.  From the cone theorem \cite[Theorem 5.2]{DHP22} we know that there are at most countably many rational curves $\{C_i\}_{i\in I}$ such that  $0>(K_X+B)\cdot C_i\>-6$ for all  $i\in I$ and
 \[
 \NA(X)=\NA(X)_{(K_X+B)\>0}+\sum_{i\in I}\mbR^+\cdot[C_i].
 \]

 We claim that there is an $i\in I$ such that $f_*C_i=0$. If not, then $f^*\omega_Z\cdot C_i=\omega_Z\cdot f_*C_i>0$ for all $i\in I$, since $\omega_Z$ is a K\"ahler class on $Z$. Since the classes $[C_i]$ are contained in a  discrete lattice of $H^4(X, \mbZ)$, it follows that there is an $\epsilon >0$ such that $\omega_Z\cdot f_*C_i\geq \epsilon $ for all $i\in I$. Then for some $t_0\gg 0$ we may assume that $t_0f^*\omega_Z\cdot C_i\geq 7$ for all $i\in I$. Thus $(K_X+B+t_0f^*\omega_Z)\cdot C_i>0$ for all $i\in I$, and hence $K_X+B+t_0f^*\omega_Z$ is nef on $X$, a contradiction. Now we contract an extremal ray $R=\mbR^+[C_i]$ such that $f_*C_i=0$ using \cite[Theorem 1.7]{DH20} and obtain a morphism $g:X\to Y$ to a normal K\"ahler variety $Y$. Then from the rigidity lemma it follows that there is a unique morphism $h:Y\to Z$ such that $f=h\circ g$. We now replace $X$ by the corresponding divisorial contraction or flip \cite[Theorems 1.1, 1.2]{DH20}. Repeating this process we construct a MMP over $Z$. Termination of flips follows from \cite[Theorem 1.12]{DO23}. 
 
 \end{proof}~\\

 %%%%%%%%%%%%%%%%%%%%%%%%%%%%%%%%%%
 \section{Threefold generalized MMP}\label{sec:g-mmp-dim-3}
 \subsection{Running the MMP for $\mbR$-Cartier Divisors}\label{s-tgmmp}
     Throughout this section we will repeatedly use the results of \cite{DH20} on the 3-fold MMP for $\mbQ$-factorial compact K\"ahler klt pairs $(X,B)$. Note that in this reference, the results are stated for the case that $K_X+B$ is $\mathbb Q$-Cartier, however, they also hold when $K_X+B$ is an $\mathbb R$-Cartier divisor.
     This is because if $K_X+B$ is an $\mathbb R$-Cartier divisor, then it can be approximated by a sequence of klt $\mathbb Q$-Cartier divisors $K_X+B_n$ (for example, if $X$ is $\Q$-factorial, let $B_n=\frac 1 n\lrd nB\rrd$). The cone theorem for $K_X+B$  is easily seen to follow from the cone theorem (cf. \cite[Theorem 5.2]{DHP22}) applied to the sequence of $\mathbb Q$-Cartier divisors $K_X+B_n$. If $\Gamma$ is a $(K_X+B)$-negative extremal ray, then it is also a $(K_X+B_n)$-negative extremal ray for any $n\gg 0$ and so the contraction of $\Gamma$, $c_\Gamma :X\to Y$ exists by \cite[Theorem 1.5]{DH20}. Similarly, if $X\to Y$ is a $(K_X+B)$-flipping contraction, then it is also a $(K_X+B_n)$-flipping contraction and hence the flip $X^+\to Y$ exists \cite[Theorem 4.3]{CHP16}. The termination of flips follows by the usual arguments (see \cite[Theorem 3.3]{DO23}).
 
\begin{lemma}\label{l-3-nef}
Let $f:X'\to X$ be a proper bimeromorphic morphism of normal compact analytic varieties of dimension $3$, and $\beta '\in H^{1,1}_{\BC}(X')$ is a nef class such that $\beta :=f_*\beta '$ represented by current with local potentials, i.e. $\beta\in H^{1,1}_{\BC}(X)$.  If $\beta$ is not nef, then $\beta \cdot C<0$ for some curve $C\subset X$ contained in the indeterminacy locus of $f^{-1}$. 
 \end{lemma}
 \begin{proof} If $\beta$ is not nef, then $\beta|_V$ is not pseudo-effective for some subvariety $V\subset X$, by \cite[Thm. 2.36 and Rmk. 2.37]{DHP22}.
 Suppose that $V$ is not contained in the indeterminacy locus of $f^{-1}$ and let $V':=f^{-1}_*V$, then $\beta '|_{V'}$ is pseudo-effective and hence so is $\beta|_V=(f|_V)_*(\beta'|_{V'})$.
 Therefore, we may assume that $V=C$ is one of the finitely many curves contained in the indeterminacy locus of $f^{-1}$. Since $\beta |_C$ is not pseudo-effective and $\dim C=1$, we have $\beta\cdot C <0$.
 \end{proof}

 \begin{lemma}\label{lem:length-bound-by-modified-kahler}
 Let $X$ be a normal compact K\"ahler $3$-fold and $\omega$ a modified K\"ahler class on $X$. Then for any countable collection of non-numerically equivalent curves $\{C_i\}_{i\in I}$,  there is a positive real number $b>0$ such that $\omega\cdot C_i\>b$ for all but finitely many curves. Moreover, if $(X, B)$ is a log canonical pair for some $\mbR$-divisor $B\>0$ and $\{C_i\}_{i\in I}$ are all the rational curves generating the $(K_X+B)$-negative extremal rays of $\NA(X)$, then there are only finitely many curves $\{C_j\}_{j\in J}$, $J\subset I$, such that $(K_X+B+\omega)\cdot C_j<0$ for all $j\in J$.
 \end{lemma}
 
 \begin{proof}
 Let $f:X'\to X$ be a resolution of singularities of $X$ and $\omega'$ a K\"ahler class on $X'$ such that $f_*\omega'=\omega$. Then $f^*\omega=\omega'+E$, where $E$ is a $f$-exceptional divisor. From the negativity lemma it follows that $E$ is effective. Since $\dim X=3$ and $E$ is $f$-exceptional, $\dim f(\Supp E)\<1$. Therefore there can be at most finitely many curves $\{C_j\}_{j\in J}$, $J\subset I$, contained in $f(\Supp E)$. In particular, $\omega\cdot C_i= f^*\omega\cdot C'_i=(\omega'+E)\cdot C'_i>0$ for all $i\in I\setminus J$, where $C'_i$ is the strict transform of $C_i$. Note that these $C'_i$ are also not numerically equivalent. Moreover, since $\omega'$ is a K\"ahler class, there is a positive real number $b>0$ such that $\omega'\cdot C'_i\>b$ for all $i\in I\setminus J$. In particular, $\omega\cdot C_i\> \omega'\cdot C'_i\>b$ for all $i\in I\setminus J$.
 
 If $\{C_i\}_{i\in I}$ are generators of the $(K_X+B)$-negative extremal rays $R_i$.
 If $\nu:X'\to X$ is a dlt model (see e.g. \cite[Claim 6.11]{DH23}) so that $\nu ^*(K_X+B)=K_{X'}+B'$ where $(X',B')$ is dlt, then it is easy to see that there are $(K_{X'}+B')$-negative extremal rays $R'_i=\mathbb R ^+[C'_i]$ spanned by curves $C'_i$ such that $R_i=\mathbb R ^+[\nu _* C'_i]$. We may then assume that $C_i=\nu _*C'_i$.
 By \cite[Corollary 5.3]{DHP22} we may assume that $(K_X+B)\cdot C_i=(K_{X'}+B')\cdot C'_i\>-6$ for all $i\in I$. Therefore, if $(K_X+B+\omega)\cdot C_i<0$ for some $i\in I\setminus J$, then $\omega'\cdot C'_i\<\omega\cdot C_i<-(K_X+B)\cdot C_i\<6$. Since $\omega'$ is a K\"ahler class, it follows that there are only finitely many $(K_X+B+\omega)$-negative extremal rays. 
 \end{proof}

 \subsection{Existence of Log Terminal Models}\label{subsec:ltm}
 In this subsection we will establish the existence of log terminal models and log canonical models, and prove Theorem \ref{thm:ltm}.
 
 In the following two results we will show that we can run a MMP with scaling (which terminates after finitely many steps) when $K_X+B+\bbeta_X$ is pseudo-effective and $\bbeta_X$ is a modified K\"ahler class.
 \begin{proposition}\label{pro:mmp-with-scaling}
 Let $(X, B)$ be a $\mbQ$-factorial compact K\"ahler $3$-fold klt pair. Let $\omega\in H^{1,1}_{\BC}(X)$ be a modified K\"ahler class such that $K_X+B+\omega$ is pseudo-effective  (resp. $K_X+B+\omega$ is not pseudo-effective) and $K_X+B+(1+t)\omega$ is nef for some $t\>0$. Then we can run a $(K_X+B+\omega)$-MMP with scaling of $t\omega$ which terminates with a log terminal model (resp. with a Mori fiber space). 
 \end{proposition}
 We remark that we are not assuming that $(X, B+\boldsymbol{\omega})$ is a generalized pair in this proposition. The proof relies instead on the fact that this MMP is in fact a $(K_X+B)$-MMP and modified K\"ahler classes are preserved under steps of the MMP.
 \begin{proof}
 Let $\lambda:=\inf\{t\>0\;:\; K_X+B+(1+t)\omega\mbox{ is nef}\}$.
 \begin{claim}\label{clm:curve-at-nef-threshold}
 If $\lambda >0$, then there exists a $(K_X+B)$-negative extremal ray $\mathbb R^+[C]$ such that $(K_X+B+(1+\lambda) \omega )\cdot C=0$.
 \end{claim}
 \begin{proof}[Proof of Claim \ref{clm:curve-at-nef-threshold}]
 By \cite[Theorem 5.2]{DHP22}, there are countably many $(K_X+B)$-negative extremal rays generated by curves $\{C_i\}_{i\in I}$ such that $0>(K_X+B)\cdot C_i\geq -6$. 
 Since $\omega$ is a modified K\"ahler class, by Lemma \ref{lem:length-bound-by-modified-kahler} there is a finite subset $I^-\subset I$ 
 such that $(K_X+B+\omega)\cdot C_i< 0$ if and only if $i\in I^-$.   Note that $(K_X+B+\omega)\cdot C_i\geq  0$ and $\omega \cdot C_i>0$ for any $i\in I^+:=I\setminus I^-$, and hence 
 $(K_X+B+(1+s)\omega)\cdot C_i>0$ for $s>0$ and any $i\in I^+$.
 Let $I_0$ be the set of $i\in I^-$ such that  $(K_X+B+(1+\lambda)\omega)\cdot C_i=0$.
 
 We claim that $I_0\neq\emptyset$.
 To see this, suppose that $I_0=\emptyset$, then $(K_X+B+(1+\lambda)\omega)\cdot C_i>0$ for all $i\in I^-$. Since $I^-$ is a finite set, then there is a positive real number $b>0$ such that $(K_X+B+(1+\lambda)\omega)\cdot C_i>b$ for any $i\in I^-$, and there is a positive real number $c>0$ such that $\omega\cdot C_i\<c$ for all $i\in I^-$. Choose a positive real number $0<\delta< {\rm min}\{\lambda , \ b/c\}$, then 
 \begin{equation}\label{eqn:small-nef}
     (K_X+B+(1+\lambda-\delta)\omega)\cdot C_i\geq b-\delta c>0 \quad\mbox{for all } i\in I^-.
 \end{equation}
 As observed above, since $s:=\lambda-\delta>0$, then
  \eqref{eqn:small-nef} holds also for all $i\in I^+$ and hence for all $i\in I$.

Finally, observe that \[K_X+B+(1+\lambda-\delta)\omega=\frac{\delta}{1+\lambda}(K_X+B)+\left(1-\frac{\delta}{1+\lambda}\right)(K_X+B+(1+\lambda)\omega)\] 
and so $K_X+B+(1+\lambda-\delta)\omega$ is non-negative on $\overline{\rm NA}(X)_{K_X+B\geq 0}$.
Since, by \cite[Theorem 5.2]{DHP22}, \[\overline{\rm NA}(X)=\overline{\rm NA}(X)_{K_X+B\geq 0}+\sum _{i\in I}\mathbb R ^+\cdot [C_i],\]
then $K_X+B+(1+\lambda-\delta)\omega$ is non-negative on $\overline{\rm NA}(X)$ and so $K_X+B+(1+\lambda-\delta)\omega$ is nef, which is a contradiction to the definition of $\lambda$.

 \end{proof}

 Now, let $R=\mbR^+[C]$ be a $(K_X+B)$-negative extremal ray such that $(K_X+B+(1+\lambda) \omega )\cdot C=0$, and in particular, $\omega \cdot C>0$. Then, by \cite[Theorem 1.7]{DH20}, we can contract this ray and obtain a morphism $f:X\to Y$ to a normal compact K\"ahler variety $Y$ with rational singularities. 
 If $f$ is a Mori fiber space,  $K_{X}+B+\omega$ is not pseudo-effective and the claim is proved. Otherwise $f$ is bimeromorphic.
 If $f$ is a flipping contraction, then let $f':X'\to Y$ be the associated flip (and if $f$ is a divisorial contraction, let $X'=Y$), and $B', \omega '$ the pushforwards of $B$ and $\omega$ on $X'$. 
 
 Note that $K_{X'}+B'+(1+\lambda)\omega '$ is nef and $\omega'$ is modified K\"ahler. We now let $\lambda':=\inf\{t\>0\;:\; K_{X'}+B'+(1+t)\omega'\mbox{ is nef}\}$ and repeat the process. Note that $ \lambda \geq \lambda'\geq 0$ and the process terminates as there is no infinite sequence of steps for any $(K_X+B)$-MMP by \cite[Theorem 1.12]{DO23}.
 Thus we eventually end up either with a $(K_X+B+(1+\lambda)\omega)$-trivial Mori fiber space for some $\lambda >0$, in which case $K_X+B+\omega$ is not pseudo-effective, or we end up with a $(K_X+B+\omega)$-minimal model, in which case $K_X+B+\omega$ is pseudo-effective.
 
 \end{proof}

\begin{corollary}\label{cor:rel-mmp-with-scaling}
 Let $(X, B)$ be a $\mbQ$-factorial compact K\"ahler $3$-fold klt pair and $\pi:X\to S$ a proper surjective morphism to a K\"ahler variety. Let $\omega\in H^{1,1}_{\BC}(X)$ be a modified K\"ahler class over $S$, $K_X+B+\omega$ is pseudo-effective (resp. not  pseudo-effective) over $S$, and $K_X+B+(1+t)\omega$  is K\"ahler over $S$ for some $t\>0$. Then we can run a $(K_X+B+\omega)$-MMP over $S$ with scaling of $t\omega$ which terminates with a log terminal model over $S$ (resp. a Mori-fiber space over $S$). 
 \end{corollary}

 \begin{proof}
 Replacing $\omega$ by $\omega +\pi ^* \omega _S$ where $\omega _S$ is a suitable K\"ahler class on $S$, we may assume that $\omega\in H^{1,1}_{\BC}(X)$ is a modified K\"ahler class and $K_X+B+(1+t)\omega$ is K\"ahler.
 Now assume that $K_X+B+\omega$ is pseudo-effective. 
 Let $\{C_i\}_{i\in I}$ be the set of curves generating all $(K_X+B)$-negative extremal rays of $\NA(X)$.
 Since $\omega$ is modified K\"ahler, from the proof of Lemma \ref{lem:length-bound-by-modified-kahler} it follows that for almost all curves $\Sigma\subset X$, $\omega\cdot \Sigma\geq 0$ holds, in other words, there are only finitely many curves $\Sigma _1,\ldots ,\Sigma _k\subset X$ such that $\omega\cdot \Sigma_i<0$ for all $i=1,\ldots, k$.

 Then from  \cite[Theorem 5.2]{DHP22} it follows that $0<-(K_X+B)\cdot C_i\<6$ for all $i\in I$. Since $K_X+B+(1+t)\omega$ is K\"ahler, it follows that $\omega\cdot C_i>0$ for all $i\in I$. In particular, we have $(K_X+B+(1+\mu)\omega)\cdot C_i\geq -6$ for any $0\leq \mu \leq t$ and for all $i\in I$.
 Pick a K\"ahler class $\eta _S$ on $S$ such that $C\cdot \eta _S>6$ for any curve $C$ on $S$. Let $\omega':=\omega +\pi ^*\eta_S$, then $\omega '$ is a modified K\"ahler class on $X$, $K_X+B+\omega'$ is pseudo-effective, and $K_X+B+(1+t)\omega'$ is K\"ahler.
 By Proposition \ref{pro:mmp-with-scaling}, we may run the $(K_X+B+\omega')$-MMP with scaling of $t\omega '$. 
 Let \[\lambda:=\inf\{s>0\;:\; K_X+B+\omega'+s\omega' \mbox{ is nef}\}.\] 
 Then by Claim \ref{clm:curve-at-nef-threshold} there is a $(K_X+B)$-negative extremal ray spanned by a curve $C_i$ such that $(K_X+B+(1+\lambda)\omega')\cdot C_i=0$. We claim that $\pi_*C_i=0$.  If not, i.e. if $\pi_* C_i\ne 0$, then we have
 \[0=(K_X+B+(1+\lambda)\omega)\cdot C_i+(1+\lambda)\pi ^* \eta_S \cdot C_i>-6+(1+\lambda)6>0\]
 which is a contradiction. Therefore $\pi_* C_i= 0$ and so the corresponding flip or divisorial contraction is a step of the $(K_X+B)$-MMP over $S$.
 Since there is no infinite sequence of $(K_X+B)$-flips (see \cite[Theorem 3.3]{DO23}), we may repeat this procedure finitely many times until we obtain a $(K_X+B+\omega)$-minimal model over $S$.

  Finally, suppose that $K_X+B+\omega$ is not pseudo-effective, then arguing as above, the minimal model program $f:X\dasharrow X'$ ends with a $K_X+B+(1+\lambda)\omega$ Mori fiber space $g:X'\to Z$  for some $\lambda >0$ (cf. Proposition \ref{pro:mmp-with-scaling}). 
 Since each step of this MMP is over $S$, so is the morphism $g$. In particular, it follows that $K_X+B+\omega$ is not pseudo-effective over $S$.
 \end{proof}

 We now prove the existence of log canonical models when $K_X+B+\bbeta_X$ is big and generalized klt. This is result is of fundamental importance and will be used repeatedly in the rest of the article.
 \begin{theorem}\label{t-3-mmp}
 Let $(X,B+\bbeta)$ be a generalized klt pair, where $X$ is a compact K\"ahler $3$-fold. Assume that $K_X+B+\bbeta_X$ is big. Then the following hold:
 \begin{enumerate}
 \item  $(X,B+\bbeta _X)$ has a (unique) log canonical model.
 \item There exists a log terminal model and all such  models  admit a morphism to the log canonical model.
     \item Assume that $X$ is strongly $\Q$-factorial,  $[K_X+B+\bbeta_X]\in H^{1,1}_{\rm BC}(X)$ is very general  and $\phi :X\dasharrow X^{m}$ is a strongly $\Q$-factorial  $(K_X+B+\bbeta_X)$  log terminal model, then $X^{\rm m}$ coincides with the log canonical model.
 \end{enumerate}
 \end{theorem}
 \begin{proof} We will first prove (1) and (2). We begin with the following reduction.
 
 \begin{claim} We may assume that $(X,B)$
 is log smooth and $\bbeta_X$ is a K\"ahler class.\end{claim}
\begin{proof}
Let $f:X'\to X$ be a structure morphism of the generalized pair $(X, B+\bbeta)$. 
Since $K_X+B+\bbeta_X$ is big, by \cite[Theoreme 1.4]{Bou02} and passing to a higher resolution if necessary, we may assume that $f^*(K_X+B+\bbeta_X)\equiv F'+\omega'$, where $\omega'$ is a K\"ahler form and $F'\>0$ is an effective $\mbQ$-divisor. Let $F+\omega:=f_*(F'+\omega ')$, then $F\geq 0$ and $\omega $ is modified K\"ahler.
For any $0<\epsilon \ll 1$, $(X,B+\epsilon F+\bbeta +\epsilon \bar \omega ')$ is generalized klt and \[K_X+B+\epsilon F+\bbeta _X +\epsilon \omega\equiv (1+\epsilon)(K_X+B+\bbeta_X).\] Thus, replacing $(X,B+\bbeta)$ by $(X, B+\epsilon F+\bbeta +\epsilon \bar \omega ' )$, we may assume that $\bbeta _{X'}$ is K\"ahler
for some  log resolution $f:X'\to X$ of the generalized pair $(X, B+\bbeta)$. 

Let $E\geq 0$ be an effective $\mbQ$-divisor such that $-E$ is $f$-ample and $E=\Ex(f)$. By Lemma \ref{l-models}(5) (see also \cite[Lemma 3.6.9]{BCHM10}), a log terminal model (resp. the log canonical model) of $K_{X'}+(B')_{\>0}+\eps E+\bbeta_{X'}$ (for $0<\eps\ll 1$) is also a log terminal model (resp. the log canonical model) of $K_X+B+\bbeta_X$. 
Thus replacing $(X,B+\bbeta)$ by $(X', (B')_{\geq 0}+\epsilon E+ \bbeta)$, we may assume that $(X,B)$
 is log smooth and $\bbeta_X$ is a K\"ahler class. 
\end{proof}
 
Then $K_X+B+(1+t)\bbeta_X$ is K\"ahler for $t\gg 0$ and $K_X+B+\bbeta_X$ is pseudo-effective, and thus by Proposition \ref{pro:mmp-with-scaling}, we can run the $(K_X+B+\bbeta_X)$-MMP with scaling of $t\bbeta_X$. We obtain a log terminal model
 $\phi:X\dasharrow X^{\rm m}$ such that $\alpha^{\rm m}:=K_{X^{\rm m}}+B^{\rm m}+\bbeta _{X^{\rm m}}=\phi_*(K_X+B+\bbeta_X)$ is nef and big, and $\bbeta_{X^{\rm m}}$ is a modified K\"ahler class. Moreover, since we are running an MMP with scaling, we also have that $K_{X^{\rm m}}+B^{\rm m}+(1+\epsilon)\bbeta _{X^{\rm m}}$ is nef (and big) for all $0\leq \epsilon \ll 1$. 
 \begin{claim}
 After a finite sequence of $\alpha^{\rm m}$-trivial steps of the $(K_{X^{\rm m}}+B^{\rm m})$-MMP $X^{\rm m}\dasharrow X^n$, we may assume that $(K_{X^n}+B^n)\cdot C\geq 0$ for any $\alpha^n$-trivial curve $C\subset X^n$ and that ${\rm Null}(\alpha ^n)$ does not contain any surface.
 \end{claim}
 \begin{proof}
 The proof follows exactly as in the proof of \cite[Theorem 6.4]{DH20} where it is shown that we may flip and contract all $(K_{X^{\rm m}}+B^{\rm m})$-negative extremal rays that are $\alpha^{\rm m}$-trivial. Note that in \cite[Theorem 6.4]{DH20} it is assumed that $\bbeta_{X^{\rm m}}$ is nef and big, but the arguments of the proof only use that $\bbeta_{X^{\rm m}}$ is modified K\"ahler. 
 \end{proof}

\begin{claim}\label{c-m}
There is a proper bimeromorphic contraction $\pi:X^n\to Z$ contracting ${\rm Null}(\alpha ^n)$ such that $\mu:X^{\rm m}\dasharrow Z$ is also a morphism.
\end{claim}
\begin{proof}
The morphism $\pi:X^n\to Z$ contracting ${\rm Null}(\alpha ^n)$ exists by \cite[Proposition 6.2]{DH20}. Following the proof of \cite[Theorem 6.4]{DH20}, we argue that $\mu: X^{\rm m}\to Z$ is also an $\alpha ^{\rm m}$-trivial morphism. 
To this end, we consider the above sequence of $\alpha ^{\rm m}$-trivial steps of the $K_{X^{\rm m}}+B^{\rm m}$-MMP, say
\[X^{\rm m}\dasharrow X^1 \dasharrow \ldots \dasharrow X^n.\]
Proceeding by descending induction on $is$, it suffices to show that if $\mu ^i:X^i\to Z$  is a proper bimeromorphic contraction contracting $\Null(\alpha ^i)$, then $\mu ^{i-1}:X^{i-1}\dasharrow Z$  is also a proper bimeromorphic contraction contracting ${\rm Null}(\alpha ^{i-1})$.
If $f_{i-1}:X^{i-1}\dasharrow X^i$ is a divisorial contraction, then the claim is clear, so assume that $f_{i-1}$ is a flip and let $\pi:X^{i-1}\to W$, $\pi ^+:X^i\to W $ be the corresponding flipping and flipped contractions. 
Since $f_{i-1}$ is $\alpha ^{i-1}$-trivial, then $\alpha ^i\cdot C=0$ for any curve $C$ contracted by $\pi ^+$. Since $\mu ^i$ contracts ${\rm Null}(\alpha ^i)$, it also contracts $C$. Then from the rigidity lemma \cite[Theorem 4.1.13]{BS95} it follows that $W\dasharrow Z$ is a morphism and hence so is $\mu ^{i-1}:X^{i-1}\dasharrow Z$. It is easy to see that $\mu ^{i-1}$ contracts precisely the null locus ${\rm Null}(\alpha ^{i-1})$.
The claim follows.
\end{proof}

Let $C\subset X^{\rm m}$ be a curve contracted by $\mu:X^{\rm m}\to Z$. It is easy to see (passing through the graph of $X^{\rm m}\bir X^n$) that $\alpha^{\rm m}\cdot C=0$. Now recall that $K_{{X^{\rm m}}}+B^{\rm m}+(1+\eps)\bbeta_{X^{\rm m}}$ is nef. So from  \begin{align*}
    \eps\bbeta_{X^{\rm m}}    &= (K_{X^{\rm m}}+B^{\rm m}+(1+\eps)\bbeta_{X^{\rm m}})-(K_{X^{\rm m}}+B^{\rm m}+\bbeta_{X^{\rm m}})\\
                        &= (K_{X^{\rm m}}+B^{\rm m}+(1+\eps )\bbeta_{X^{\rm m}})-\alpha^{\rm m}
\end{align*}
 it follows that $\bbeta_{X^{\rm m}}\cdot C\geq 0$ for all curves $C\subset X^{\rm m}$ contracted by $\mu:X^{\rm m}\to Z$. Thus $-(K_{X^{\rm m}}+B^{\rm m})$ is $\mu$-nef-big, as $-(K_{X^{\rm m}}+B^{\rm m})|_{X^{\rm m}_z}\equiv\bbeta_{X^{\rm m}}|_{X^{\rm m}_z}$ for all $z\in Z$. % and hence (locally over $Z$) there is a klt pair $(X',\Delta)$ such that $-(K_{X'}+\Delta) $ is ample over $Z$. 
 Then by \cite[Lemma 8.8]{DHP22}, $Z$ has rational singularities. Now since $Z$ is in Fujiki's class $\mathcal C$, by \cite[Lemma 3.3]{HP16} there exists a $(1, 1)$ class $\alpha _Z\in H^{1,1}_{\rm BC}(Z)$ such that $\alpha ^{\rm m}\equiv \mu ^* \alpha _Z$. One then easily checks that $\Null(\alpha _Z)=\emptyset $ and so $\alpha _Z$ is K\"ahler by \cite[Theorem 2.29]{DHP22}.  
 Thus $K_Z+B_Z+\beta_Z:=\mu_*(K_{X^{\rm m}}+B^{\rm m}+\bbeta_{X^{\rm m}})$ is a log canonical model of $K_X+B+\bbeta_X$. The uniqueness of log canonical models follows by (3) of Lemma \ref{l-models}; this proves (1). 
 
 (2) The fact that log terminal models admit a morphism to the log canonical model follows from the Claim \ref{c-m}  above.
  
 (3) Finally, suppose that $[K_X+B+\bbeta_X]$ is very general in $H^{1,1}_{\BC}(X)$ and $\pi:X^{\rm m}\to Z$ is the morphism from a log terminal model $X^{\rm m}$ to the log canonical model $Z$. Since $X$ and $X^m$ are strongly $\Q$-factorial, it follows from Lemma \ref{lem:h11-surjective}  that the induced morphism $\phi_*: H^{1,1}_{\BC}(X)\to H^{1,1}_{\BC}(X^{\rm m})$ is surjective, and in particular, the class of $K_{X^{\rm m}}+B^{\rm m}+\bbeta_{X^{\rm m}}$ is very general in $H^{1,1}_{\BC}(X^{\rm m})$.
Let $C$ be a curve contracted by $\pi$, then $(K_{X^{\rm m}}+B^{\rm m}+\bbeta _{X^{\rm m}})\cdot C=0$, contradicting the fact that $[K_{X^{\rm m}}+B^{\rm m}+\bbeta_{X^{\rm m}}]$ is very general. Therefore $\pi$ is a quasi-finite proper morphism with connected fibers,  and hence an isomorphism.\\

 \end{proof}~\\
 We will also need the following relative version of Theorem \ref{t-3-mmp}.
 \begin{theorem}\label{t-3-mmp-rel}
 Let $(X,B+\bbeta)$ be a compact K\"ahler $3$-fold generalized klt pair, where $\pi :X\to S$ is a morphism to a compact K\"ahler variety and $K_X+B+\bbeta_X$ is big over $S$. Then the following hold:
 \begin{enumerate}
 \item  $(X,B+\bbeta _X)$ has a (unique) log canonical model $X\dasharrow X^{\rm c}$ over $S$. 
 \item There exists a log terminal model $X\dasharrow X^{\rm m}$ over $S$ such that $K_{X^{\rm m}}+B^{\rm m}+\bbeta _{X^{\rm m}}+p^*\omega _S$ is nef for some K\"ahler form $\omega _S$ on $S$ (where $p:X^{\rm m}\to S$ is the induced morphism) and  there is a morphism $X^{\rm m}\to X^{\rm c}$ over $S$. 
 \item If $X\dasharrow X^{\rm m}$  
 is a $K_{X^{\rm m}}+B^{\rm m}+\bbeta _{X^{\rm m}}$ log terminal model over $S$, then $K_{X^{\rm m}}+B^{\rm m}+\bbeta _{X^{\rm m}}+p^*\omega _S$ is nef for some K\"ahler form $\omega _S$ on $S$, and  there is a morphism $X^{\rm m}\to X^{\rm c}$ over $S$.
 \end{enumerate}
 \end{theorem}
 \begin{proof} Adding a sufficiently large multiple of a pullback of a  K\"ahler form $\omega _S$ from $S$, we may assume that $K_X+B+\bbeta_X$ is big. Proceeding as in the proof of Theorem \ref{t-3-mmp}, replacing $X$ by a higher model, we may assume that $\bbeta_X $ is K\"ahler so that $K_X+B+(1+t)\bbeta_X$ is also K\"ahler for $t\gg 0$. 
As in the proof of Corollary \ref{cor:rel-mmp-with-scaling}, after adding the pullback of a sufficiently large multiple of a K\"ahler form $\omega _S$ on $S$,
we run the $(K_X+B+\bbeta_X)$-MMP with scaling of $t\bbeta_X$ which turns out to be a MMP over $S$, and we obtain a log terminal model $X\dasharrow X^{\rm m}$ over $S$ such that $K_{X^{\rm m}}+B^{\rm m}+\bbeta _{X^{\rm m}}$ is nef and hence also nef over $S$. 
By Theorem \ref{t-3-mmp} there is a log canonical model $\psi:X^{\rm m}\to X^{\rm c}$ for $K_{X^{\rm m}}+B^{\rm m}+\bbeta _{X^{\rm m}}+p ^*\omega _S$, where $\omega_S$ is a K\"ahler class on $S$, and $p:X^{\rm m}\to S$ is the induced morphism. Note that $\psi$ is a bimeromorphic morphism, so its fibers are covered by curves and $\psi:X^{\rm m}\to X^{\rm c}$ contracts $(K_{X^{\rm m}}+B^{\rm m}+\bbeta_{X^{\rm m}}+p ^*\omega _S)$-trivial curves.
Since $K_{X^{\rm m}}+B^{\rm m}+\bbeta _{X^{\rm m}}$ is nef and $\omega _S$ is K\"ahler, any such curve must be vertical over $S$ and hence by the rigidity lemma (see \cite[Lemma 4.1.13]{BS95}), there is a morphism $X^{\rm c}\to S$ so that $\psi:X^{\rm m}\to X^{\rm c}$ is the log canonical model for $K_{X^{\rm m}}+B^{\rm m}+\bbeta _{X^{\rm m}}$ over $S$.
Thus (1) and (2) hold.

Suppose now that $X\bir X^{\rm m}$ is any log terminal model of $K_X+B+\bbeta_X$ over $S$.
We begin by showing the following.
\begin{claim}
    There exists a K\"ahler form $\omega _S$ on $S$ such that $K_{X^{\rm m}}+B^{\rm m}+\bbeta _{X^{\rm m}}+p ^*\omega _S$ is nef.
\end{claim}
\begin{proof}
    If $K_{X^{\rm m}}+B^{\rm m}+\bbeta_{X^{\rm m}}$ is nef, then the claim is obvious. Otherwise, let $X\dasharrow X^n$ be the log terminal model of $K_{X}+B+\bbeta_X$ over $S$ constructed in (2). Then $K_{X^n}+B^n+\bbeta_{X^n}+q^*\omega _S$ is nef for some K\"ahler class $\omega _S$ on $S$, where $q:X^n\to S$ is the induced morphism. Now, by Lemma \ref{l-models}  $X^{\rm m}\dasharrow X^n$ is an isomorphism in codimension 1 between log terminal models of $K_X+B+\bbeta_X$ over $S$, and hence it easily follows from the negativity lemma that
    if $r:W\to X^{\rm m}$ and $s:W\to X^n$ is a common resolution, then $r^*(K_{X^{\rm m}}+B^{\rm m}+\bbeta_{X^{\rm m}}+p^*\omega _S)=s^*(K_{X^n}+B^n+\bbeta_{X^n}+q^*\omega _S)$.
    Since $K_{X^n}+B^n+\bbeta_{X^n}+q^*\omega _S$ is nef, so is $K_{X^{\rm m}}+B^{\rm m}+\bbeta_{X^{\rm m}}+p^*\omega _S$.
\end{proof}
By Lemma \ref{l-models}(5), it follows that $X^{\rm m}\to X^{\rm c}$ is a morphism and hence (3) also holds.
 \end{proof}

 \begin{theorem}\label{t-Qfac} 
 Let $(X,B+\bbeta)$ be a generalized klt pair, where $X$ is a compact K\"ahler $3$-fold. Then the following hold: 
 \begin{enumerate}     
 \item $X$ has rational singularities,  
 \item there exists a small bimeromorphic morphism $\nu :X^q\to X$ such that $X^q$ is strongly $\mathbb Q$-factorial, and
 \item there exists a bimeromorphic morphism $\nu :X^t\to X$ such that $X^t$ is strongly $\mathbb Q$-factorial and $(X^t,B^t+\bbeta )$ is a generalized terminal pair such that 
 $K_{X^t}+B^t+\bbeta _{X^t}=\nu^*(K_X+B+\bbeta_X)$. 
 \end{enumerate}
 \end{theorem}
 Note that a local version of (2) was proven in Theorem \ref{t-gkltlocal}.
 \begin{proof} 
 (1) follows from Theorem \ref{t-gkltlocal}.
 
 (2)   Let $f :X'\to X$ be a projective log resolution of the generalized pair $(X,B+\bbeta)$.
 Fix $0<\epsilon \ll 1$ and let $\phi:X'\dasharrow X^q$ be a log terminal model of $K_{X'}+f ^{-1}_*B+(1-\epsilon){\rm Ex}(f)$ over $X$ which exists by Theorem \ref{t-3-mmp-rel}.
 Since $K_{X'}+f ^{-1}_*B+{(1-\epsilon)\rm Ex}(f)+\bbeta _{X'}\equiv _X F$ where $F\geq 0$ and ${\rm Supp}(F)={\rm Ex}(f)$, it follows that 
 $F^q=\phi _* F\geq 0$ is $f^q:X^q\to X$ exceptional and $F^q$ is nef over $X$ and so by the negativity lemma, $F^q=0$.
 Therefore $f^q$ is a small bimeromorphic morphism and $X^q$ is strongly $\mathbb Q$-factorial (cf. \cite[Lemma 2.5]{DH20}).
 
 The proof of (3) is also standard (see for example \cite[Corollary 1.4.3]{BCHM10}) and follows similarly to the proof of (2) and so we omit it.

  \end{proof}~\\

 The following theorem is a variant of the base-point-free theorem \cite[Theorem 1.7]{DH20}.
 \begin{theorem}\label{t-3-mmp1}
 Let $(X,B+\bbeta)$ be a generalized klt pair, where $X$ is a compact K\"ahler 3-fold. Assume that $K_X+B+\bbeta_X$ is nef but not big and $\bbeta$ descends to a big (and nef) class on some log resolution of $(X, B+\bbeta)$. Then
   there is a morphism $\phi:X\to Z$ to a normal K\"ahler variety $Z$ such that $K_{X}+B+\bbeta _X =\phi^*\alpha _Z$, where $\alpha _Z$ is a K\"ahler class on $Z$.
  
 \end{theorem}
 \begin{proof}
 Note that if $f:Y\to X$ is a bimeromorphic morphism 
 and $f':Y\to Z$ a proper morphism (not necessarily bimeromorphic) of normal compact K\"ahler varieties such that $f^*\alpha ={f'}^*\alpha _Z$, where $\alpha _Z$ is a K\"ahler class on $Z$, then $f'$ contracts all $f$-vertical curves and so by the rigidity lemma (see \cite[Lemma 4.1.13]{BS95}) there is a morphism $g:X\to Z$  such that $g\circ f=f'$ and $\alpha =g^*\alpha _Z$. Therefore, by Theorem \ref{t-Qfac}  we may assume that $X$ is strongly $\mathbb Q$-factorial and $(X,B+\bbeta)$ has generalized terminal singularities. Let $\nu:X'\to X$ be a log resolution of $(X, B+\bbeta)$ such that $K_{X'}+B'+\bbeta_{X'}=\nu^*(K_X+B+\bbeta_X)$ and $[\bbeta_{X'}]\in H^{1,1}_{\BC}(X')$ is big (and nef). Replacing $X'$ by a higher model, we may assume by \cite[Theoreme 1.4]{Bou02} that $\bbeta_{X'}\equiv F+\omega '$ where $F\>0$ is an effective $\mbQ$-divisor and $\omega'$ is a K\"ahler form. Pick $\epsilon >0$ such that $(X',B'+\epsilon F)$ is sub-klt. Define $B^*:=f_*(B'+\eps F)$ and $\bbeta^*:=(1-\eps)\bbeta_{X'}+\eps\omega'$. Then $(X, B^*+ \bbeta^*)$ is a generalized pair and $\bbeta^*_{X'}$ is a K\"ahler class. Note that $K_X+B^*+\bbeta^*_X\num K_X+B+\bbeta_X$; thus replacing $(X, B+\bbeta)$ by $(X, B^*+\bbeta^*)$ we may assume that $\bbeta_X$ is a modified K\"ahler class.

 Now, if $K_X$ is pseudo-effective, then $K_X+B+\bbeta_X$ is big, which is a contradiction. Therefore $K_X$ is not pseudo-effective, and hence $X$ is uniruled.
 \begin{claim}
 Let $\pi:X\dasharrow T$ be the MRC fibration. Then we may assume that $\dim T=2$. 
 \end{claim}\begin{proof}
 Since $X$ is uniruled, $\dim T\leq 2$.
 If $\dim T\leq 1$, then by \cite[Lemma 2.42]{DH20}, $X$ is projective and $H^2(X, \mathcal O _X)=0$. Thus by  Lemma \ref{lem:ns-h11}, every $(1,1)$ class on $X$ is represented by a $\mbR$-Cartier divisor. In particular, $(X,B+\bbeta_X)$ is numerically equivalent to a traditional generalized pair for projective varieties, i.e. $[\bbeta_{X'}]=c_1(N')$, where $N'$ is an ample $\mbR$-divisor on $X'$. We may assume that $(X',B'+N')$ is sub-klt.

 But then $(X,\Delta :=f_*(B'+N'))$ is klt such that $\Delta\>0$ is big and $K_X+B+\bbeta_X \equiv K_X+\Delta$.
 The conclusion now follows from the usual base-point free theorem for $\mbR$-divisors, for example see \cite[Theorem 3.9.1]{BCHM10}.
 Therefore we may assume that $\dim T=2$.
 \end{proof}

 \begin{claim}\label{clm:pseff-not-big}
 Let $F$ be a general fiber of $\pi:X\bir T$, then $F\cong \mathbb P^1$ and $(K_X+B+\bbeta_X)\cdot F=0$.
 \end{claim}
 \begin{proof}
 Let $g:Y\to X$ be a log resolution of $(X, B+\bbeta)$ which also resolves the map $\pi:X\bir T$. Write
 \[
 K_Y+B_Y+\bbeta_Y=g^*(K_X+B+\bbeta_X)+E,
 \]
 where $B_Y\>0, E\>0, g_*B_Y=B, g_*E=0$, $B_Y$ and $E$ do not share any component, and $\bbeta_Y$ is nef.
 
 Observe that the general fibers of $\pi\circ g$ and $\pi$ are isomorphic. 
 Now since $(K_X+B+\bbeta_X)$ is pseudo-effective, so is $K_Y+B_Y+\bbeta_Y$, and thus $(K_Y+B_Y+\bbeta_Y)\cdot F\geq 0$. If 
 $(K_X+B+\bbeta_X)\cdot F>0$, then $(K_Y+B_Y+\bbeta_Y)\cdot F=(K_X+B+\bbeta_X)\cdot F>0$, and thus
 $(K_Y+B_Y+t\bbeta_Y) \cdot F>0$ for some $1>t>0$.
Then by \cite[Theorem]{Gue20}, $K_Y+B_Y+t\bbeta_Y$ is pseudo-effective and so 
 $K_Y+B_Y+\bbeta_Y+(1-t)\Ex(f)$ is big, since $\bbeta_X$ is big. In particular, $K_X+B+\bbeta_X=g_*(K_Y+B_Y+\bbeta_Y+(1-t)\Ex(f))$ is big, a contradiction.
 \end{proof}

 Now, as in the proof of \cite[Theorem 5.2]{DH20} we will analyze the nef dimension of $K_X+B+\bbeta_X$. Since a dense open subset of $X$ is covered by $(K_X+B+\bbeta_X)$-trivial curves, we see that the nef dimension $n(K_X+B+\bbeta_X)\<2$. If $n(K_X+B+\bbeta_X)=0$, then $K_X+B+\bbeta_X\num 0$ and we are done by choosing $Z:=\mbox{Specan}(\mbC)$. If $n(K_X+B+\bbeta_X)=1$, then there is a smooth projective curve $C$ and a morphism $\phi:X\to C$ such that $K_X+B+\bbeta_X=\phi^*\alpha_C$, where $\alpha_C\in H^{1,1}_{\BC}(C)$, (see \cite[2.4.4]{BCEKPRSW02} and \cite[Theorem 3.19]{HP15}). Since the nef dimension $n(\phi^*\alpha_C)=1$, it follows that $\alpha_C$ is a K\"ahler class and we are done. The final case is $n(K_X+B+\bbeta_X)=2$.  In this case, by an  argument identical to the one in \cite[Theorem 5.5]{DH20}, we find the required morphism $\phi:X\to Z$.
 
Note that as observed above, in \cite[Theorem 6.4]{DH20}, $\bbeta_X$ is assumed to be nef and big, however 
$\bbeta_X$ being modified K\"ahler is enough for the proof in \cite{DH20}.

 \end{proof}

 \begin{theorem}\label{t-3-mmp2}
 Let $(X,B+\bbeta)$ be a $\mbQ$-factorial generalized klt pair, where $X$ is a compact K\"ahler 3-fold, such that $K_X+B+\bbeta_X$ is pseudo-effective but not big, and $\bbeta$ descends to a big (and nef) class on some log resolution of $(X, B+\bbeta)$. Then there is a log terminal model $f:X\dasharrow X^{\rm m}$ and a morphism $g:X^{\rm m}\to Z$ such that $K_{X^{\rm m}}+B^{\rm m}+\bbeta_{X^{\rm m}} =g^*\alpha _Z$, where $\alpha _Z$ is a K\"ahler class on $Z$.
 \end{theorem}
 
 \begin{proof}
 Let $\nu :X'\to X$ be a log resolution of $(X,B+\bbeta)$. Since $\bbeta _{X'}$ is nef and big, passing to a higher resolution, we may assume that $\bbeta _{X'}\equiv \omega '+F$, where $F\geq 0$ is an effective $\mbQ$-divisor and $\omega$ is a K\"ahler form.
 We write $K_{X'}+B'+\bbeta _{X'}=\nu ^*(K_X+B+\bbeta_X)$.
 We let $B^*:=\nu ^{-1}_*B+ (1-\epsilon){\rm Ex}(\nu)$ and $\bbeta ^*:=(1-\delta)\bbeta+\delta \bar \omega'$ for some $0<\eps, \delta<1$. Then for $0<\delta \ll \epsilon \ll 1$, we have that $\bbeta ^*_{X'}$ is a K\"ahler form, $(X',B^*+\bbeta ^*)$ is generalized klt, and 
 \[K_{X'}+B^*+\bbeta ^*_{X'}\equiv K_{X'}+B'+\bbeta _{X'}+E,\]
 where $E\geq 0$ is a $\mbR$-divisor such that ${\rm Supp}(E)={\rm Ex}(\nu)$. By Lemma \ref{l-models}, we may replace $(X,B+\bbeta )$ by $(X',B^*+\bbeta ^*)$ and hence we may assume that $\bbeta _{X}$ is K\"ahler. In particular $K_X+B+(1+t)\bbeta_X$ is nef for any $t\gg 0$. By Proposition \ref{pro:mmp-with-scaling} there is a log terminal model $f:X\dasharrow X^{\rm m}$. The existence of the morphism $g:{X^{\rm m}}\to Z$ such that $K_{X^{\rm m}}+B_{X^{\rm m}}+\bbeta_{X^{\rm m}} \equiv g^*\alpha _Z$, where $\alpha _Z$ is a K\"ahler class on $Z$, follows from Theorem \ref{t-3-mmp1}.
     
 \end{proof}

\begin{proof}[Proof of Theorem \ref{thm:ltm}]
   It follows from combining Theorems \ref{t-3-mmp} and \ref{t-3-mmp2}.
\end{proof}

 Next we will establish an analog of \cite[Corollary 1.1.5]{BCHM10} for log canonical models.

  \begin{theorem}\label{t-finite}
  Let $X$ be a normal $\mbQ$-factorial compact K\"ahler $3$-fold and $\nu :X'\to X$ a resolution. Let $(X,B)$ be a pair and $\Omega'$ a compact convex polyhedral set of real closed positive $(1,1)$ currents on $X'$  such that for every $\beta'\in \Omega'$, $(X,B+\bbeta)$ is a generalized klt pair and $\nu$ is a log resolution of $(X, B+\bbeta)$, where $\bbeta=\bar  \beta '$. Assume that one of the following conditions hold:
 \begin{enumerate}
     \item[(i)] $K_X+B+  \bbeta _X$ is big for every $\beta'\in \Omega'$ (and $\bbeta=\bar  \beta '$),  or
     \item[(ii)] there is a bimeromorphic morphism $\pi : X\to S$ of normal compact K\"ahler $3$-folds.
 \end{enumerate} 
 Then there exists a finite polyhedral decomposition $\Omega ' =\cup \Omega' _i$ and finitely many bimeromorphic maps $\psi _{i}:X\dasharrow X_{i}$ (resp. finitely many bimeromorphic maps $\psi _{i}:X\dasharrow X_{i}$ over $S$) such that if $\psi :X\dasharrow Y$ is a log  canonical model for $K_X+B+\nu _*\beta'$ (resp. a log canonical model for $K_X+B+\nu _*\beta'$ over $S$) for some $\beta ' \in \Omega'_i$, then $\psi =\psi_{i}$.
 
 \end{theorem}
 Note that a compact convex polyhedral set is a convex hull of finitely many vectors. Then by finite polyhedral decomposition $\Omega '=\cup \Omega' _i$ we simply mean that each $\Omega' _i$ is a subset of $\Omega'$ defined by finitely many affine linear equations and inequalities such that $\Omega' _i\cap \Omega' _j=\emptyset$ for $i\ne j$.
 
 \begin{proof} We will prove both cases (i) and (ii) simultaneously. We will use the convention that in case (i), $S={\rm Specan}(\mathbb C)$, and we remark that in case (ii) the condition that $K_X+B+\bbeta_X$ is big over $S$ is automatic as $\pi$ is bimeromorphic.
 We will use induction on the dimension of $\Omega'$. We will abuse notation and denote $\bbeta _X$ by $\beta$.
 If $\dim\Omega'=0$, then $\Omega'=\{\beta'_0\}$ for some $\beta'_0$ such that $(X, B+\beta_0=B+\nu _*\beta'_0)$ is a generalized klt pair and $K_X+B+\beta_0$ is big (over $S$). In this case the existence of the required log canonical model follows by Theorems \ref{t-3-mmp} and \ref{t-3-mmp-rel}. 
 
 Since $\Omega'$ is compact, it is enough to prove the statement locally in a neighborhood of each point $\beta'\in\Omega'$. Fix a point $\beta'_0\in\Omega'$ and let $\beta _0=\nu _*\beta '_0\in \Omega :=\nu _*\Omega '$. By  Theorems \ref{t-3-mmp} and \ref{t-3-mmp-rel},
  there is a $(K_X+B+\beta_0)$-log terminal model $\phi:X\dasharrow X^{\rm m}$ (over $S$) 
 and a log canonical model $\psi:X^{\rm m}\to X^{\rm c}$ (over $S$). Since $a(E, X, B+\beta_0)<a(E, X^{\rm m}, B^{\rm m}+\beta_0^{\rm m})$ for all $\phi$-exceptional divisors $E$ of $X$ (where $B^{\rm m}+\beta^{\rm m}_0=\phi_*(B+\beta_0)$), shrinking $\Omega'$ (to a smaller polytope without changing its dimension) around $\beta '_0$ we may assume that if $\beta:=\nu _*\beta '$ and $\beta^{\rm m}:=\phi_*\beta$, then $a(E, X, B+ \beta)<a(E, X^{\rm m}, B^{\rm m}+\beta^{\rm m})$ for all  $\beta'\in\Omega'$ and for all $\phi$-exceptional divisors $E$ of $X$. In particular, if $\phi^{\rm m}:X^{\rm m}\to \bar X^{\rm m}$ is a log canonical model for $K_{X^{\rm m}}+B^{\rm m}+\beta^{\rm m} $ (over $S$), then $\phi ^{\rm m}\circ \phi:X\bir \bar X^{\rm m}$ is a log canonical model for $K_{X}+B+\beta$ (over $S$).

Now let $\Omega^{\rm m}:=\phi_*\Omega$. Note that $\Omega^{\rm m}$ is a compact convex polyhedral subset of $H^{1,1}_{\BC}(X^{\rm m})$, since $\phi_*$ is a linear map. Then, by induction, there is a finite polyhedral decomposition $\partial\Omega^{\rm m}=\cup_{i=1}^k \mcP_i$  of the boundary $\partial\Omega^{\rm m}$ of $\Omega^{\rm m}$

 and finitely many meromorphic maps $\phi _{i}:X^{\rm m}\dasharrow X_{i}$ (over $X^{\rm c}$) $1\<i\<k$ such that if $f:X^{\rm m}\dasharrow Y$ is a log canonical model of $K_{X^{\rm m}}+B^{\rm m}+\beta^{\rm m}$ over $X^{\rm c}$ for some $\beta^{\rm m}\in \mcP_i$, then $f=\phi _{i}$ (note that, as $K_X+B+\beta _0$ is big over $S$, $\psi:X^{\rm m}\to X^{\rm c}$ is bimeromorphic). Recall that $\beta_0^{\rm m}:=\phi_*\beta_0\in \Omega^{\rm m}$. Choose $\beta^{\rm m}_1\in\partial\Omega^{\rm m}$ such that $\beta^{\rm m}_1\neq \beta_0^{\rm m}$. 
 For $0<\lambda\<1$ we define
 \begin{equation}\label{eqn:polytope}
 \beta_\lambda^{\rm m}:=(1-\lambda) \beta_0^{\rm m}+\lambda\beta^{\rm m}_1.
 \end{equation}
 Recall that $K_{X^{\rm m}}+B^{\rm m}+\beta^{\rm m}_0=\psi^* \omega \equiv _{X^{\rm c}}0$ for some  class $\omega$ on $X^{\rm c}$ which is K\"ahler (over $S$). Since $\phi_{i}:X^{\rm m}\bir X_{i}$ is a  log canonical model of $K_{X^{\rm m}}+B^{\rm m}+\beta^{\rm m}_1$ over $X^{\rm c}$ for some $i$, then from \eqref{eqn:polytope} we have
\begin{equation}\label{eqn:polytope-2}
    K_{X_{i}}+B_{i}+\beta_{\lambda, i}\num (1-\lambda)\psi^*_{i}\omega+\lambda(K_{X_{i}}+B_{i}+\beta_{1, i}),
\end{equation} 
 where $\psi_{i}:X_{i}\to X^{\rm c}$ is the induced bimeromorphic morphism.

 Since \begin{equation}\label{eqn:polytope-2}
    \frac 1{\lambda}(K_{X_{i}}+B_{i}+\beta_{\lambda, i})= (K_{X_{i}}+B_{i}+\beta_{1, i})+\frac{1-\lambda}{\lambda}\psi^*_{i}\omega,
\end{equation}
it follows that $K_{X_i}+B_i+\beta_{\lambda, i}$ is K\"ahler over $S$ for all $0<\lambda\ll 1$; in particular, 
$\phi_{i}$ is a log canonical model of $K_{X^{\rm m}}+B^{\rm m}+\beta^{\rm m}_\lambda$ over $S$ for all $0< \lambda\ll 1 $.
 \begin{claim}\label{clm:models-over-S}
 There exists a constant $\bar \lambda >0$ such that for every $\beta^{\rm m}_1\in\partial\Omega^{\rm m}$ there exists a $\phi_{i}:X^{\rm m}\bir X_{i}$ for some $i\in \{1,2,\ldots,k\}$ such that $\phi_{i}$ is a log canonical model for $K_{X^{\rm m}}+B^{\rm m}+\beta^{\rm m}_\lambda$ (over $S$) for all $0\leq\lambda< \bar \lambda $.
 \end{claim}
 \begin{proof}

  Since $\beta ' \in \Omega '$ is nef over $S$, then if $\beta _i=(\phi _i\circ \phi\circ \nu)_*\beta '$ is not nef over $S$, there must be a curve $C$ (vertical over $S)$ contained in the indeterminacy locus of $X_i\dasharrow X'$ such that $\beta _i\cdot C<0$. Since the indeterminacy locus is of dimension at most $1$ ($\codim \geq 2$), there are only finitely many irreducible curves in the following set: 
  \[\mathcal C _i:=\{ C\subset X_i\;|\; p_{i,*}C=0,\ \beta _i\cdot C<0, \mbox{ for some } \beta _i\in \mathcal Q_i:=\phi_{i,*}\mathcal P _i\},\]
where $p_i:X_i\to S$ is the induced morphism.
 By compactness of $\mathcal Q _i$, there exists an integer $M_i> 0$ such that $\beta _i\cdot C\geq -M_i$ for any $C\in \mathcal C _i$ and $\beta _i \in \mathcal Q _i$.
 It follows that $\beta _i\cdot C\geq -M_i$ for every curve $C$ on $X_i$. Let $\delta >0$ be a constant such that $\omega  \cdot D\geq \delta $ for every $S$-vertical curve $D$ on $X^c$, where $\omega =K_{X^c}+B^c+\beta _0^c$ is K\"ahler over $S$. 
 Let \[s:={\rm sup}\{\lambda> 0\;|\; K_{X_{i}}+B_{i}+\beta_{\lambda,i} \mbox{ is K\"ahler over } S\}. \] 
From \eqref{eqn:polytope-2} it follows that $s>0$.
 Let $M={\rm max}\{M_i\;:\; 1\<i\<k\}$, then $\beta_i\cdot C\geq -M$ for all curves $C\subset X_i$ and for all $i=1,2,\ldots, k$. We claim that $s\geq \frac \delta {\delta +M+6}$.

By contradiction assume that $s<\frac \delta {\delta +M+6}$. By Theorem \ref{t-3-mmp-rel}, there is a relative log canonical model $\eta:X_i\to Y$ over $S$ for $K_{X_{i}}+B_{i}+\beta_{s,i}$. From our assumptions it follows that $\eta$ is bimeromorphic and hence its exceptional locus is covered by curves. If $\Sigma$ is a $\eta$-exceptional curve, then $\Sigma$ is not $\psi _i$-exceptional as otherwise $(K_{X_i}+B_i+\beta _{\lambda ,i})\cdot \Sigma=0$ for all $\lambda$ contradicting the fact that the corresponding class is K\"ahler over $S$ for $0<\lambda \ll 1$.

If $\Sigma$ is a curve spanning a $(K_{X_i}+B_i)$-negative extremal ray over $Y$ then, by what we observed above, the curve $\psi_i(\Sigma)\subset X^c$ is vertical over $S$. Thus $\psi _i^*\omega \cdot \Sigma >0$ and so $ \psi _i^* \omega\cdot \Sigma=\omega \cdot  \psi _{i,*}\Sigma \geq \delta$.
 We may also assume that $0>(K_{ X_i}+ B_i)\cdot \Sigma \geq -6$. But then 
 \[0= (K_{ X_{i}}+B_{i}+\beta_{s,i})\cdot \Sigma =(1-s)\psi _i^*\omega \cdot \Sigma+s(K_{ X_{i}}+ B_{i}+ \beta_{1,i})\cdot \Sigma\geq (1-s)\delta -s(6+M) >0.\]
This is impossible and so, by the cone theorem, $K_{X_i}+ B_i$ is nef over $Y$. 
Suppose again that $\Sigma$ is a $\eta$-exceptional curve, then as observed above, $\Sigma$ is not $\psi _i$-vertical and so  $\omega \cdot  \psi _{i,*}\Sigma \geq \delta$.
Since $(K_{X_{i}}+B_{i})\cdot \Sigma \geq 0$ and $s<\frac \delta {\delta +M+6}<\frac{\delta}{\delta+M}$, then \[0= (K_{X_{i}}+B_{i}+\beta_{s,i})\cdot \Sigma =(1-s)\omega \cdot \psi _{i,*}\Sigma+s(K_{X_{i}}+B_{i}+\beta_{1,i})\cdot \Sigma\geq (1-s)\delta -sM >0,\]
this is impossible. Therefore, there are no $\eta$-exceptional curves i.e. $X_i=Y$ is a log canonical model over $S$ for $K_{X_i}+B_i+\beta _{s,i}$.
This contradicts the definition of $s$ and it follows that 
$s\geq \frac{\delta}{\delta+M+6}$.

Now fix a $\bar\lambda$ satisfying $0<\bar \lambda < \frac \delta {\delta +M+6}$; this proves our Claim \ref{clm:models-over-S}.
 \end{proof}

 Note that as observed above, $X\dasharrow X_{i}$ is a log canonical model for $K_X+B+\beta _\lambda$ (over $S$) for all $\beta _1\in \mathcal P _i$ and $0\leq\lambda \leq \bar \lambda$. The decomposition $\partial \Omega ^{\rm m} =\cup \mathcal P _i$ induces a corresponding decomposition of $\Omega ^{\rm m}-\{\beta ^{\rm m}_0\}=\cup \Omega ^{\rm m}_i$ where each $\Omega ^{\rm m}_i $ is the polytope spanned by $\beta_0$ and $\mathcal P _i$ excluding $\beta ^{\rm m}_0$, and $\Omega ^{\rm m}_0:=\{\beta _0^{\rm m}\}$ is a 0-dimensional polytope. We then obtain a decomposition  $\Omega' =\cup \Omega' _i$, where $\Omega _i$ is the inverse image of $\Omega^{\rm m}_i$. 
 Finally we replace $\Omega '$ by $\Omega' \cap \{\beta' \in \Omega' \;:\;\ ||\beta' -\beta' _0||\leq \bar \lambda\}$ for some fixed norm $||\cdot||$. This completes the proof.

 \end{proof}~\\

\subsection{Existence of Mori Fiber Space}\label{subsec:mfs}
In this subsection we will show that if $K_X+B+\bbeta_X$ is not pseudo-effective, then we can run an MMP which ends with a Mori fiber space.

First we will show that if $K_X+B+\bbeta_X$ is big, then we can run a terminating MMP with scaling of a very general K\"ahler class. Using this result, we will then show that we can also obtain a Mori fiber in the non pseudo-effective case.
 \begin{theorem}\label{t-3-mmpbig}
 Let $(X,B+\bbeta)$ be a  strongly $\mbQ$-factorial generalized klt pair such that $K_X+B+\bbeta_X$ is big,  where $X$ is a compact K\"ahler $3$-fold. Let $\omega$ be a  very general K\"ahler class in $ H^{1,1}_{\rm BC}(X)$ such that $K_X+B+\bbeta_X+\omega$ is a K\"ahler class. Then we can run a terminating $(K_X+B+\bbeta_X)$-MMP with scaling of $\omega$.
 \end{theorem}
 
 \begin{proof}
First, re-scaling $\omega$, we may assume that $K_X+B+\bbeta_X+\omega $ is nef but not K\"ahler. Now, to run the $(K_X+B+\bbeta_X)$-MMP with scaling of $\omega$, we will inductively construct a sequence of bimeromorphic maps $\phi _i:X_i\dasharrow X_{i+1}$ and real numbers $t_i>t_{i+1}$ for $i\geq 0$
 such that $X_0=X$ and $t_0=1$ and
 the following conditions are now satisfied for all $i\geq 0$ (where we set $t_{-1}=1$)
 \begin{enumerate}
     \item $(X_i,B_i+\bbeta_{X_i}+t_i\omega_i)$ is a generalized klt pair,
     \item $K_{X_i}+B_i+\bbeta_{X_i} +t_i\omega_i$ is nef,
     \item $K_{X_i}+B_i+\bbeta_{X_i}+(t_{i-1}-\epsilon)\omega_i$ is K\"ahler for $0<\epsilon \ll 1$ and for all $i\geq 1$. 
     \item $K_{X_i}+B_i+\bbeta_{X_i}$ is big,
     \item $X_i$ is strongly $\mathbb Q$-factorial and $\omega_i\in H^{1,1}_{\rm BC}(X_i)$ is very general.
 \end{enumerate}
 The base of the induction is clear. Assume that we have constructed $(X_i,B_i+\bbeta_{X_i}+t_i\omega_i)$ as above.
 Let \[t_{i}:={\rm inf}\{s\>0\;|\; K_{X_i}+B_i+\bbeta_{X_i} +s\omega _i\mbox{ is nef} \}.\]
 If $t_{i}=0$, then the MMP terminates and $X\dasharrow X_i$ is a log terminal model for $(X,B+\bbeta_X)$. Thus we may assume that $t_{i}>0$.
 
 Let $\psi_i:X_i\to Z_i$ be the log canonical model for $K_{X_i}+B_i+\bbeta_{X_i} +t_{i}\omega_i$ (which exists by Theorem \ref{t-3-mmp}. 
 Since $\psi _i$ is bimeromorphic, the fibers of $\psi_i$ are covered by curves. Since $\omega_i$ is very general in $H^{1,1}_{\rm BC}(X_i)$ and $(K_{X_i}+B_i+\bbeta_{X_i}+t_{i}\omega_i) \cdot C=0$ for any $\psi_i$-exceptional curve $C\subset X_i$, it follows that $\rho(X_i/Z_i)=1$ and $\psi_i$ is a contraction of a $(K_{X_i}+B_i+\bbeta_{X_i})$-negative extremal ray $R_i$ spanned by (any) one of these curves, i.e. $R_i=\mathbb R^{\geq 0}[C]$. If $\psi_i$ is a divisorial contraction, then we let $\phi _i=\psi _i, X_{i+1}=Z_i$ and \[K_{X_{i+1}}+B_{i+1}+\bbeta_{X_{i+1}}+t_{i}\omega_{i+1}%:=K_{Z_i}+B_{Z_i}+\beta _{Z_i}+t_{i+1}\omega _{Z_i}
 :=\psi _{i,*}(K_{X_i}+B_i+\bbeta_{X_i}+t_{i}\omega_i).\] 
 If $\psi_i$ is a small contraction, then it is a $(K_{X_i}+B_i+\bbeta_{X_i})$-flipping contraction (as it is $\omega_i$-positive).
 \begin{claim}
 Let $X\dasharrow X_{i+1}$ be the log canonical model of $K_X+B+\bbeta_X +(t_{i}-\epsilon)\omega$ (for any $0<\epsilon \ll 1$). Then $\phi _i :X_i\dasharrow X_{i+1}$
is the flip of  $\psi_i$.\end{claim}
\begin{proof}
 By Theorem \ref{t-finite} (and its proof), we may assume that there is an $\epsilon _0>0$ such that $X\dasharrow X_{i+1}$ is the log canonical model of $K_X+B+\bbeta_X +(t_{i}-\epsilon)\omega$ for any $0<\epsilon \leq \epsilon _0 $. In particular, $K_{X_{i+1}}+B_{i+1}+\bbeta_{X_{i+1}} +t_{i}\omega_{i+1}$ is nef and hence admits a morphism $\psi _i^+:X_{i+1}\to \bar Z$ to the log canonical model of $({X_{i+1}},B_{i+1}+\bbeta_{X_{i+1}} +t_{i}\omega_{i+1})$ (which exists by Theorem \ref{t-3-mmp}). Since $X\dasharrow X_{i+1}$  is $(K_X+B+\bbeta_X +t_{i}\omega)$-non-positive, then $X\dasharrow \bar Z$ is also the log canonical model of $K_X+B+\bbeta_X +t_{i}\omega$ and hence $\bar Z=Z_i$.
 Note that $-(K_{X_{i}}+B_{i}+\bbeta_{X_i}+(t_{i}-\epsilon)\omega_{i})$ and  $K_{X_{i+1}}+B_{i+1}+\bbeta_{X_{i+1}}+(t_{i}-\epsilon)\omega_{i+1}$ are both K\"ahler over $Z_{i}$ for $0<\eps\ll 1$,  so $X_i\dasharrow X_{i+1}$ is a $(K_{X_{i}}+B_{i}+\bbeta_{X_i}+(t_{i}-\epsilon)\omega_{i})$-flip. Since $K_{X_{i}}+B_{i}+\bbeta_{X_i}+t_{i}\omega_{i}\equiv _{Z_i}0$, it follows that 
$X_i\dasharrow X_{i+1}$ is a also $(K_{X_{i}}+B_{i}+\bbeta_{X_i})$-flip.
 \end{proof}
 It is easy to check that properties (1-5) hold for $(X_{i+1},B_{i+1}+\bbeta_{X_{i+1}}+t_{i}\omega _{i+1})$.
 Repeating the above procedure we obtain a sequence of distinct $K_{X}+B+\bbeta_X +(t_i-\epsilon_i)\omega$-log canonical models. 
  By Theorem \ref{t-finite} (applied to $\Omega ':=\{t\nu^*\omega \;|\;0\leq t\leq 1\}$ for some log resolution $\nu:X'\to X$ of $(X, B+\bbeta)$), this sequence cannot be infinite and so the above minimal model program with scaling terminates and the proof is complete.
 \end{proof}~\\

 The next result shows that if $K_X+B+\bbeta_X$ is not pseudo-effective, then we can run a terminating $(K_X+B+\bbeta_X)$-MMP with scaling of a very general K\"ahler class and end with a Mori fiber space.
 \begin{theorem}\label{t-3-mmppsef}
 Let $(X,B+\bbeta)$  be a  strongly $\mbQ$-factorial generalized klt pair, where $X$ is a compact K\"ahler $3$-fold. Assume that $K_X+B+\bbeta_X$ is not pseudo-effective, and let $\omega$ be a very general K\"ahler class in $H^{1,1}_{\rm BC}(X)$ such that $K_X+B+\bbeta_X+\omega$ is K\"ahler. Then we can run the $(K_X+B+\bbeta_X)$-MMP with scaling of $\omega$ and obtain $\phi:X\bir X'$ such that  $K_{X'}+B'+\bbeta_{X'}+\tau \omega'$ is pseudo-effective but not big for some $0<\tau <1$, and there is a Mori-fiber space $g:X'\to W$.
 \end{theorem}
 \begin{proof} 
 We define
 \begin{equation*}
     \tau:=\inf\{s\>0\;|\; K_X+B+\bbeta_X +s\omega \mbox{ is pseudo-effective}\}
 \end{equation*}
 and 
 \begin{equation*}
     t_1:=\inf\{s\>0\;|\; K_X+B+\bbeta_X+s\omega \mbox{ is nef} \}.
 \end{equation*}
 Then $K_X+B+\bbeta_X+\tau\omega$ is pseudo-effective but not big. 
 By Theorem \ref{t-3-mmp2}, there is a log terminal model $\phi:X\bir X'$ and a morphism $g:X'\to Z$ of normal K\"ahler varieties such that $K_{X'}+B'+\bbeta _{X'} +\tau \omega'=g^*\alpha _Z $, where $\alpha _Z$ is a K\"ahler class. Since $K_X+B+\bbeta_X+\tau\omega$ is not big, $g$ is not bimeromorphic. 
 
 We begin by proving that we can run a $(K_X+B+\bbeta_X)$-MMP with scaling of $\omega$ terminating with a log terminal model of $K_X+B+\bbeta_X+\tau\omega$. Indeed, if $X_i$ is a step of this MMP, then let \[t_{i+1}:=\inf\{s\>0: K_{X_i}+B_i+\bbeta_{X_i}+s\omega_i\mbox{ is nef} \}.\] If $t_{i+1}>\tau$, then $K_{X_i}+B_i+\bbeta_{X_i}+t_{i+1}\omega_i$ is big and by Theorem \ref{t-3-mmpbig} we can run this MMP. Thus as long as $t_i>\tau$, we can continue running this MMP and it will stop once we have $t_i=\tau$ for some $i$ (note that every step of this MMP is also a step of $(K_X+B+\bbeta_{X})$-MMP with the scaling of $\omega$). However, it is not clear whether this process will terminate after finitely many steps.   

Assume by contradiction that this MMP does not terminate. We claim that $\lim t_i=\tau$. If not, then let $\lim t_i=\tau_0>\tau$; note that $\tau_0=\inf\{t_i\;|\; i\>0\}$. Then every step of the above MMP is also a step of $(K_X+B+\bbeta_X+\tau_0\omega)$-MMP, but since $K_X+B+\bbeta_X+\tau_0\omega$ big (as $\tau_0>\tau)$, this MMP terminates by Theorem \ref{t-3-mmpbig}, a contradiction. Now from Lemma \ref{lem:stability-of-negative-part2} we observe that 
 \begin{equation}\label{eqn:negative-parts-equal}
    \Supp N(K_X+B+\bbeta_X+t\omega)=\Supp N(K_X+B+\bbeta_X+\tau\omega)\quad \mbox{for all } \ 0<t-\tau \ll 1. 
 \end{equation}
Thus by Theorem \ref{thm:contracted-locus-of-mmp} we may assume that $X_i\dasharrow X'$ is a small bimeromorphic map for $i\gg 0$. Now from the proof of Theorem \ref{t-3-mmpbig} it follows that $K_{X_i}+B_i+\bbeta_{X_i}+t\omega _i$ is K\"ahler for any $t>0$ satisfying $t_i>t>t_{i+1}$. 
 We may also assume that if $0<t_0-\tau \ll 1$, then \[a(E,X,B+\bbeta_X +t_0\omega )<a(E,X',B'+\bbeta_{X'} +t_0\omega' )\] for all $\phi$-exceptional divisors, and $(X',B'+\bbeta_{X'} +t_0\omega' )$ is generalized klt. Fix $t_0$ as above. Recall that there is a morphism 
 $g:X'\to Z$ such that $K_{X'}+B'+\bbeta_{X'}+\tau \omega '\equiv g^*\alpha _Z$, where $\alpha _Z$ is K\"ahler on $Z$. Let $b>0$ be a constant such that
 $\alpha _Z\cdot C>b$ for any curve $C$ on $Z$ and fix $t>0$ such that $\tau <t<\frac {bt_0+6\tau}{b+6}<t_0$.
 By Theorem \ref{t-3-mmpbig} there is a sequence of $K_{X'}+B'+\bbeta _{X'}+t\omega '$ flips $X'_j\dasharrow X'_{j+1}$ with $0\leq j\leq \bar j-1$ ending with $X'\dasharrow X'_{\bar j}$, a log terminal model of $(X',B'+\bbeta_{X'} +t\omega' )$. Then $X\dasharrow X'_{\bar j}$ is a log terminal model of $(X,B+\bbeta _X+t\omega )$.
 
 We claim that each flip $\psi_j:X'_j\dasharrow X'_{j+1}$ for $0\<j\<\bar j-1$ is $(K_{X'}+B'+\bbeta_{X'}+\tau \omega ')$-trivial. We prove this by induction on $i$. Suppose that $X'_0=X'$ and the claim holds for the first $k-1$ flips, then each flip is a flip over $Z$ and so there is a morphism $g'_k:X'_k\to Z$ such that  $K_{X'_k}+B'_k+\bbeta_{X'_k}+\tau \omega '_k=(g'_k)^*\alpha _Z$. Observe that $\psi _0,\ldots ,\psi _{k-1}$ are $K_{X'}+B'+\bbeta _{X'}+\lambda\omega '$ flips for any $\tau <\lambda\leq t_0$, as each of the them are $K_{X'}+B'+\bbeta_{X'}+\tau\omega'$ trivial. Recall that $\alpha _Z\cdot C\geq b$ for every curve $C\subset Z$. Let $X'_k\dasharrow X'_{k+1}$ be the next $K_{X'}+B'+\bbeta _{X'}+t\omega '$ flip and $C_k$ a corresponding flipping curve. Since $K_{X'_k}+B'_k+\bbeta_{X'_k}+\tau \omega '_k$ is nef and $t_0>t$, we may assume that this curve is also a $(K_{X'}+B'+\bbeta_{X '}+t_0\omega ')$-flipping curve, and hence by Corollary \ref{c-3fold-flips} we may assume that $-(K_{X'_k}+B'_k+\bbeta_{X'_k}+t_0 \omega '_k)\cdot C_k\leq 6$. Moreover, if $(K_{X'_k}+B'_k+\bbeta_{X'_k}+\tau \omega '_k)\cdot C_k>0$, then $(K_{X'_k}+B'_k+\bbeta_{X'_k}+\tau  \omega '_k)\cdot C_k\geq b$. Observe that 
 \[
 K_{X'_k}+B'_k+\bbeta_{X'_k}+t \omega '_k=\frac{t_0-t}{t_0-\tau}(K_{X'_k}+B'_k+\bbeta_{X'_k}+\tau  \omega '_k)+\frac{t-\tau}{t_0-\tau}(K_{X'_k}+B'_k+\bbeta_{X'_k}+t_0 \omega '_k).
\]
 Since $\tau <t<\frac {bt_0+6\tau}{b+6}$ and hence $b(t_0-t)-6(t-\tau)>0$, we then have  
 \[
   0>(K_{X'_k}+B'_k+\bbeta_{X'_k}+t \omega '_k)\cdot C_k\geq \frac{b(t_0-t)}{t_0-\tau}-\frac{6(t-\tau)}{t_0-\tau}>0.  
\]
Since this is impossible, we have $(K_{X'_k}+B'_k+\bbeta_{X'_k}+\tau \omega '_k)\cdot C_k= 0$ and hence $\psi _k$ is $(K_{X'_k}+B'_k+\bbeta_{X'_k}+\tau \omega '_k)$-trivial and the induction is complete.

 Since $\omega$ is very general in $H^{1,1}_{\rm BC}(X)$, and we may assume that $t_i>t>t_{i+1}$ for some $i\gg 0$, then $K_{X_i}+B_i+\bbeta_{X_i}+t\omega_i$ is K\"ahler by what we have seen above. By (5) of Lemma \ref{l-models} there is  a morphism $g_i:X_i\to Z$ such that $K_{X_i}+B_i+\bbeta _{X_i}+\tau \omega _i=g_i^*\alpha _Z$. But this leads to an immediate contradiction, since if  $X_i\dasharrow X_{i+1}$ is a flip and $\Sigma _i$ is a flipping curve for the $(K_{X}+B+\bbeta_X)$-MMP with scaling of $\omega$, then $(K_{X_i}+B_i+\bbeta _{X_i}+t_{i+1}\omega _i)\cdot \Sigma _i=0$ and $\omega _i\cdot \Sigma _i>0$ so that $(K_{X_i}+B_i+\bbeta _{X_i}+\tau \omega _i)\cdot \Sigma _i<0$, but $(K_{X_i}+B_i+\tau\omega_i)\cdot \Sigma_i=\alpha_Z\cdot g_{i, *}(\Sigma_i)\>0$.
 
 This shows that our $(K_X+B+\bbeta_X)$-MMP with scaling of $\omega$ terminates after finitely many steps producing a log terminal model of $K_X+B+\bbeta_X+\tau\omega$. Let $\phi:X\bir X'$ be the composite maps of this MMP so that $K_{X'}+B'+\bbeta_{X'}+\tau\omega':=\phi_*(K_X+B+\bbeta_X+\tau\omega)$ is nef, and by Theorem \ref{t-3-mmp1} there is a morphism $g:X'\to Z$ to a normal compact K\"ahler variety $Z$ such that $K_{X'}+B'+\bbeta_{X'}+\tau\omega=g^*\alpha_Z$, where $\alpha_Z$ is a K\"ahler class on $Z$.\\

 We will now show that we have a Mori fiber space.
 Observe that $-(K_{X'}+B')|_{X_z}\equiv (\bbeta_{X'}+\tau \omega ')|_{X_z}$ is big for general points $z\in Z$; in particular $X_z$ is Moishezon and $K_{X'}+B'$ not pseudo-effective over $Z$. Thus by Theorem \ref{thm:relative-mmp}, we can run a $(K_{X'}+B')$-MMP over $Z$ which terminates with a Mori fiber space $h:X''\to W$ over $Z$. Note that each step of this MMP is $K_{X'}+B'+\bbeta_{X'}+\tau\omega'$ trivial.

Now we will show that the induced map $\psi:X'\bir X''$ is an isomorphism. To see this, let $g':X'\to Y$ be the first contraction of the above MMP over $Z$, and $\Sigma $ is a curve contracted by $g'$. Let $C$ be a curve contained in a general fiber of $g:X'\to Z$. Then $\Sigma$ and $C$ are linearly independent in $N_1(X')$, however they are both $K_{X'}+B'+\bbeta_{X'}  +\tau \omega'$ trivial, contradicting the fact that $\omega $ is very general in $H^{1,1}_{\rm BC}(X)$. Thus $\psi:X'\bir X''$ is an isomorphism and $Z=W$. In particular, $\rho(X'/Z)=1$ and $-(K_{X'}+B')$ is $g$-ample. 
Since $\omega '$ is $g$-K\"ahler (as $\rho(X'/Z)=1$) and $K_{X'}+B'+\bbeta_{X'}  +\tau \omega'$ is $g$-trivial, it follows that $-(K_{X'}+B'+\bbeta_{X'})$ is $g$-K\"ahler. This completes our proof.

 \end{proof}

\begin{proof}[Proof of Theorem \ref{thm:mfs}]
This follows from Theorem \ref{t-3-mmppsef}.    
\end{proof}

 \subsection{Cone Theorem}\label{subsec:cone}
 In this section we will prove the cone theorem for generalized pairs in dimension $3$. We start with the following lemma.
\begin{lemma}\label{l-cone}
Let $(X,B+\bbeta)$ be a strongly $\mbQ$-factorial generalized klt pair, where $X$ is a compact K\"ahler $3$-fold. Let $\omega $ be a K\"ahler class such that $\alpha:=K_X+B+\bbeta_X +\omega$ is nef but not K\"ahler. Then there is a rational curve $C\subset X$ such that $\alpha \cdot C=0$ and $0>(K_X+B+\bbeta_X)\cdot C\>-6$.
\end{lemma}
\begin{proof}
If $K_X+B+\bbeta_X $ is nef, then $\alpha$ is K\"ahler, which is a contradiction. So we may assume that $K_X+B+\bbeta_X $ is not nef.
We may write $\omega =\eta+\omega'$, where $\eta $ and $\omega'$ are very general K\"ahler classes in $H^{1,1}_{\rm BC}(X)$.
Replacing $\bbeta$ by $\bbeta+\epsilon \bar \eta$ for $0<\eps\ll 1$ and $\omega $ by $\omega -\epsilon \eta \equiv (1-\epsilon )\omega+\epsilon \omega '$, we may assume that $K_X+B+\bbeta_X $ is not nef, $\bbeta_X$ is big, $K_X+B+\bbeta_X$ is either big or not pseudo-effective, and $\omega$ is a very general class in $H^{1,1}_{\rm BC}(X)$.
By Theorems \ref{t-3-mmpbig} and \ref{t-3-mmppsef} we can run the $(K_X+B+\bbeta_X)$-MMP with scaling of $\omega$.
Let $f:X\to Z$ be the first flipping or divisorial contraction, or fiber type contraction, and $C$ the curve spanning the corresponding extremal ray; then $\alpha \cdot C=0$. If $f$ is a flipping contraction, then the result follows from Theorem \ref{t-Stein-flip+}.

So, from now on,  assume that $f$ is either a divisorial contraction or a fiber type contraction.  
 Then there is a family of $f$-vertical curves $\{\Gamma_t\}_{t\in T}$ in $X$ such that either $\cup_{t\in T}\Gamma_t=E$ is the exceptional divisor of $f$ or $\cup_{t\in T}\Gamma_t=X$, respectively. Then in the former case $\bbeta_X|_E$ is pseudo-effective, as by definition $\bbeta_X$ is a pushforward of a nef class from a resolution of $X$. Therefore $\bbeta_X\cdot\Gamma_t=\bbeta_{X}|_E\cdot \Gamma_t\>0$ (as $\{\Gamma_t\}_{t\in T}$ is a covering family of curves in $E$); in the latter case, $\bbeta_X\cdot\Gamma_t>0$, since $\bbeta_X$ is big. Therefore $\bbeta_X \cdot C\geq  0$ for all $f$-exceptional curves in either case, and so $0>(K_X+B+\bbeta _X )\cdot C \geq (K_X+B)\cdot C $. But then $f$ is a $(K_X+B)$-negative contraction and $-(K_X+B)$ is $f$-ample (as $\rho(X/Z)=1$). Then by  \cite[Theorem 1.23]{DO23} there is a rational curve $\Gamma$ such that $f_*\Gamma=0$ and $0>(K_X+B+\bbeta _X)\cdot \Gamma \geq (K_X+B)\cdot \Gamma \geq -6$.
 
\end{proof}~\\

Now we are ready to prove the Cone Theorem \ref{t-3-cone+}.
 \begin{proof}[Proof of Theorem \ref{t-3-cone+}] 

Following the proof of \cite[Corollary 5.3]{DHP22}, it suffices to show that $X$ admits a strongly $\mbQ$-factorial small bimeromorphic  modification $\nu:X'\to X$
 and that the theorem holds for $X'$.
The existence of $\nu$ follows by Theorem \ref{t-Qfac}.
Thus we assume from now on that $X$ is strongly $\Q$-factorial.

 Let $R$ be a $(K_X+B+\bbeta_X)$-negative exposed extremal ray of $\NA(X)$, i.e. there is a nef $(1,1)$ class $\alpha$ such that $\alpha^\bot\cap\NA(X)=R$. We make the following claim.
 \begin{claim}\label{clm:nef-supporting-class}
 There is a K\"ahler class $\omega$ such that $\alpha=K_X+B+\bbeta _X +\omega $ is nef but not K\"ahler and $\alpha ^\perp \cap \overline{\rm NA}(X)=R$.
 \end{claim}
 \begin{proof}
 Fix a norm $||\cdot||$ on $N_1(X)$ and let $\mcS$ be the unit sphere in $N_1(X)$, i.e. $\mcS:=\{\gamma\in N_1(X): ||\gamma||=1\}$. Let $S:=\mcS\cap\NA(X)$; then $S$ is a compact subset of $\NA(X)$ such that for any $\gamma\in\NA(X)\setminus\{0\}, \frac{\gamma}{||\gamma||}\in S$. Moreover, from \cite[Corollary 3.16]{HP16} it follows that a class $\alpha\in H^{1,1}_{\BC}(X)$ is K\"ahler if and only if $\alpha\cdot \gamma>0$ for all $\gamma\in S$.
 
 There is a unique point $r\in R$ such that $R\cap S=\{r\}$. Let $\eta$ be a $(1,1)$ nef supporting class of $R$; then $\eta^\bot\cap\NA(X)=R$.
 
 For $\epsilon >0$, let $B_\epsilon:=\{ s\in S\;:\; ||s-r||< \epsilon\}$.
 Choosing $0<\epsilon \ll 1$ we may assume that $B_\epsilon \subset \NA(X)_{(K_X+B+\bbeta_X)<0}$. Then clearly $\eta-(K_X+B+\bbeta_X)$ is positive on $B_\eps$, i.e.
 \begin{equation}\label{eqn:kahler-on-B}
   (\eta-(K_X+B+\bbeta_X))\cdot s>0 \mbox{ for all } s\in B_\eps.    
 \end{equation}

 Now define $S_\epsilon:=S\setminus B_\eps$. Observe that $\eta\cdot s>0$ for all $s\in S_\eps$. Since $S_\eps$ compact, there exist positive real numbers $\delta >0$ and $M>0$ such that $\eta \cdot s\>\delta$ and $-(K_X+B+\bbeta_X)\cdot s\geq -M$ for all $s\in S_\epsilon$. Then for $t\gg 0$, $t\delta-M>0$, and thus
 \begin{equation}\label{eqn:kahler-on-S}
    (t\eta-(K_X+B+\bbeta_X))\cdot s\>(t\delta-M)>0 \quad\mbox{for all } s\in S_\eps. 
 \end{equation}
 Since $\eta$ is nef, from \eqref{eqn:kahler-on-B} we have $(t\eta-(K_X+B+\bbeta_X))\cdot s>0$ for all $s\in B_\eps$ and $t\>1$. Recall that $S=B_\eps\cup S_\eps$, and thus we have $(t\eta-(K_X+B+\bbeta_X))\cdot s>0$ for all $s\in S$, and hence $t\eta-(K_X+B+\eta)$ is K\"ahler for $t\gg 0$. Let $\omega:=t\eta-(K_X+B+\bbeta_X)$ for some $t\gg 0$. Then $\alpha:=t\eta=K_X+B+\bbeta_X+\omega$ proves our claim.

 \end{proof}
 It then follows from Lemma \ref{l-cone} that there is a rational curve $\Gamma\subset X$ such that $\alpha\cdot \Gamma=0$ and $-(K_X+B+\bbeta_X )\cdot \Gamma \leq 6$. Thus by the Bishop's theorem there are at most countably many $(K_X+B+\bbeta_X)$-negative exposed extremal rays $\{R_i\}_{i\in I}$ generated by rational curves $\{\Gamma_i\}_{i\in I}$ such that $-(K_X+B+\bbeta_X)\cdot\Gamma_i\leq 6$.  
 
 Let $V=\overline{\rm NA}(X)_{K_X+B+\bbeta_X \geq 0}+\sum _{i\in I}\mathbb R ^+[\Gamma _i]$. By \cite[Lemma 6.1]{HP16} it suffices to show that 
 $\overline{\rm NA}(X)=\overline V$ (note that \cite[Lemma 6.1]{HP16} is only stated for $K_X$, but the same proof works for $K_X+B+\bbeta_X$). 

 Now let $\ext(\NA(X))$ and $\exp(\NA(X))$ be the set of extremal rays and exposed extremal rays of $\NA(X)$, respectively. Since $\NA(X)$ is a strongly convex closed cone, by Theorem 1.21 and 1.23 of \cite{HW20} it follows that the convex hull of $\ext(\NA(X))$, the closure of the convex hull of $\exp(\NA(X))$ and the cone $\NA(X)$ coincide, i.e.
 \[
 \overline{\mbox{convex}\left(\exp(\NA)\right)}=\mbox{convex}\left(\ext(\NA)\right)=\NA(X).
 \]
 Thus, if the inclusion $\bar V\subset \NA(X)$ is strict, then there is an exposed extremal ray $R$ of $\NA(X)$ which is not contained in $\overline V$. Since $\NA(X)_{(K_X+B+\bbeta_X)\geq 0}$ is contained in $V$, $R$ must be $(K_X+B+\bbeta_X)$-negative. Then by the argument above there is a rational curve $\Gamma\subset X$ such that $R=\mbR^+[\Gamma]$ and $0<-(K_X+B+\bbeta_X)\cdot\Gamma\leq 6$. But then $R\subset V$ by our construction above and this is a contradiction.

 Finally, if $\nu :X'\to X$ is a log resolution and $\bbeta _{X'}$ is big, then by \cite[Def. 3.7 and Pro. 3.8]{Bou04}, we may write $\bbeta_{X'}\equiv  N+\eta$, where $N\geq 0$ is an effective $\mbR$-divisor and $\eta$ is a modified K\"ahler class. Passing to a higher model, we may assume that in fact $\eta$ is a K\"ahler class. For $0<\epsilon \ll 1$ we have that $(1-\epsilon)\bbeta _{X'}+\epsilon \eta$ is K\"ahler and $(X',B'+\epsilon N)$ is sub-klt. Let $\omega$ be a K\"ahler form on $X$, then $(1-\epsilon)\bbeta _{X'}+\epsilon \eta-\delta \nu ^*\omega $ is K\"ahler for $0<\delta \ll 1$.  If we let $\gamma := (1-\epsilon)\bbeta _{X'}+\epsilon \bar \eta-\delta \nu ^*\omega$, $\ggamma:=\overline \gamma$, and $B^\epsilon :=B+\epsilon \nu _*N $, then $(X,B^\epsilon +\ggamma)$ is generalized klt and $K_X+B+\bbeta _X\equiv K_X+B^\epsilon+\ggamma _X +\delta \omega $.
 Note that if $C$ is a curve not contained in the indeterminacy locus of $\nu ^{-1}$, then $\ggamma _X\cdot C> 0$.
 Thus, arguing as above we get \[0>(K_X+ B^\epsilon+\ggamma _X+\delta \omega )\cdot\Gamma_i \geq  (K_X+ B^\epsilon)\cdot\Gamma_i \geq -6\] for all but finitely many $i$, and hence, 
 \[ \delta \omega \cdot\Gamma_i < -(\ggamma _X+\delta \omega )\cdot\Gamma_i <6.\]
 By  Bishop's theorem such curves belong to finitely many families. 
 \end{proof}

 Next we will establish an analog of \cite[Corollary 1.1.5]{BCHM10} for log terminal models.
 
\subsection{Geography of Minimal Models}

 \begin{theorem}\label{t-finite-ltms}Let $X$ be a normal compact K\"ahler $3$-fold,  $\nu :X'\to X$ a log resolution of a klt pair $(X,B)$, and $\Omega'$ a compact convex polyhedral set of closed positive (1,1) currents on $X'$  such that for every $\beta'\in \Omega'$, $(X,B+\bbeta)$ is a generalized klt pair and $\nu$ is a log resolution of the pair $(X, B+\bbeta)$, where $\bbeta=\bar\beta '$. Assume that one of the following conditions hold:
 \begin{enumerate}
     \item[(i)] $K_X+B+\bbeta_X$ is big for every $\beta'\in \Omega'$ (and $\bbeta=\bar\beta'$), or
     \item[(ii)] there is a bimeromorphic morphism $\pi : X\to S$.
 \end{enumerate} 
 Then there exists a finite polyhedral decomposition $\Omega' =\cup \Omega' _i$ and finitely many bimeromorphic maps $\psi _{ij}:X\dasharrow X_{ij}$ (resp. finitely many bimeromorphic maps $\psi _{ij}:X\dasharrow X_{ij}$ over $S$) such that if $\psi :X\dasharrow Y$ is a weak log canonical model for $K_X+B+\bbeta_X$ (resp. a weak log canonical  model for $K_X+B+\bbeta_X$ over $S$) for some $\beta' \in \Omega'_i$ (with $\bbeta=\bar\beta'$), then $\psi =\psi_{ij}$ for some $i,j$. 
 
 \end{theorem}
\begin{proof} Arguing as in the proof of Theorem \ref{t-finite}, we will
prove both cases (i) and (ii) simultaneously. We will use the
convention that in case (i), S = Specan(C) and we remark that in case (ii) the
condition that $K_X + B + \bbeta_X$ is big over $S$ is automatic as $\pi$ is bimeromorphic.
By compactness, it suffices to prove the result on a neighborhood of any $\beta' _0\in \Omega'$. For simplicity of notation, from now on we will write $\beta$ on $X$ to denote $\bbeta_X$ and so on.
Note that $\nu :X'\to X$ is a log resolution of $(X,B+\beta)$ for any $\beta \in \Omega=\nu _*\Omega '$.
By \cite[Theorem 1.4]{Bou02}, we may assume that $\nu^*(K_X+B+\beta_0)\equiv_S \omega ' +F$ where $\omega '$ is K\"ahler and $F\geq 0$ has simple normal crossing support. Let $B':=\nu ^{-1}_*B+(1-\delta ){\rm Ex}(\nu )$ for  $0<\delta \ll 1$. Then the weak log canonical models of 
$K_{X'}+B'+\beta'$ and $K_X+B+\nu _* \beta'$ (over $S$) coincide for every $\beta '\in \Omega'$ by  Lemma \ref{l-models}(7). Replacing $(X,B)$ by $(X',B')$ and $\Omega$ by $\Omega '$  we may assume that all $\beta\in\Omega$ are nef and descend to $X$, $X$ is smooth, and  $K_X+B+\beta_0\equiv \omega  +F$, where $\omega$ is K\"ahler, $F\geq 0$ and $B+F$ has simple normal crossings support.

Pick $\delta >0$ such that $(X,B+\delta F)$ is klt and consider the linear map $L(\beta)=\frac 1{1+\delta}(\beta + \delta \beta _0)$.
Note that $L(\beta _0)=\beta _0$ and $L(\Omega )\subset \Omega$ contains a neighborhood of $\beta _0$. Since we are working locally around $\beta _0$ it suffices to prove the claim for $L(\Omega )$ instead of $\Omega$.
Since \[K_X+B+{\delta}F+\beta +{\delta}\omega \equiv K_X+B+\beta +\delta (K_X+B+\beta _0)\equiv (1+\delta)(K_X+B+L(\beta)),\]
then the weak log canonical models of $(X,B+L(\beta))$ and $(X, B+{\delta}F+\beta +{\delta}\omega)$ coincide for all $\beta \in \Omega$.
Replacing $B$ by $B+{\delta}F$, $\beta$ by $\beta +{\delta}\omega$ we may assume that $\beta =\eta +\gamma$, where  $\gamma$ is a fixed K\"ahler class and $\eta :=\beta -\gamma$ is nef for any $\beta \in \Omega$. In particular, each $\beta =\eta +\gamma$ is K\"ahler.

Let $\{\gamma _1,\ldots ,\gamma _\rho\}$ be K\"ahler forms whose classes in $H^{1,1}_{\rm BC}(X)$ form a basis of $H^{1,1}_{\BC}(X)$. For $\eps>0$ define 
\[\Omega ^{\epsilon}:=\left\{\beta+\sum _{i=1}^\rho t_i\gamma_i\;:\; \beta \in \Omega,\; |t_i|\leq \epsilon,\;  1\<i\<\rho \right\}.\] 
For $0<\epsilon \ll 1$ we may assume that $K_X+B+\beta'$ is generalized klt and big (over $S$), and $\beta'$ is K\"ahler for any $\beta '\in \Omega ^\epsilon$. By Theorem \ref{t-finite}, there exists a finite polyhedral decomposition $\Omega ^\epsilon =\cup _{j\in J}\Omega ^\epsilon_j$ and finitely many bimeromorphic maps $\psi ^\epsilon_{j}:X\dasharrow X^\epsilon_{j}$  (over $S$) such that if $\psi :X\dasharrow Z$ is a log  canonical model for $K_X+B+\beta'$ (over $S$) for some $\beta' \in \Omega ^\epsilon_j$, then $\psi =\psi^\epsilon_{j}$. Suppose now that $\phi:X\dasharrow Y$ is a weak log canonical model of $K_X+B+\beta, $ (over $S$) for some $\beta \in \Omega$, and let $\eta$ be a K\"ahler class on $Y$. 

Note that as $(X,B+\bbeta)$ is gklt, then so is $(Y,\phi _*B+\bbeta)$ and hence $Y$ has rational singularities by Theorem \ref{t-gkltlocal}. Since $K_X+B+\beta$ is big (over $S$), $\phi$ is bimeromorphic. Now since $X$ is smooth, passing through a resolution of the graph of $\phi$ we see that $\phi^*\eta$ has local potentials.
Since $\{\gamma _1,\ldots ,\gamma _\rho\}$ spans $H^{1,1}_{\rm BC}(X)$ 
 we may pick $ t_1,\ldots ,t_\rho\in\mbR$ such that $\phi^*\eta \equiv \sum _{i=1}^\rho t_i\gamma_i$ and so by \cite[Lemma 3.3]{HP16} (applied to the graph of $\phi)$ we get that $\phi _* (\sum _{i=1}^\rho t_i\gamma_i)\equiv \eta$. For any $0<\delta \ll 1$, it follows that $\phi$ is a log canonical model for $K_X+B+\beta +\sum _{i=1}^\rho (\delta t_i)\gamma_i$ and that $\beta':=\beta +\sum _{i=1}^\rho (\delta t_i)\gamma_i\in \Omega ^{\epsilon}$. But then, there exists $j\in J$ such that $\beta'\in \Omega ^{\epsilon}_j$ and hence $\phi =\psi ^\epsilon_{j}$.

We now let $\{\Omega _i\}_{i\in I}$ be the finite polyhedral decomposition induced by  $\{\Omega ^\epsilon _j\cap \Omega\}_{j\in J}$.
For each $i\in I$ we let $\{\psi _{i,j}\}=\{\psi _j^\epsilon \}_{j\in J}$.
Suppose now that $\psi:X\dasharrow Y$ is a weak log canonical model for $K_X+B+\beta$ where $\beta  \in \Omega _i$, then as observed above $\psi$ is a log canonical model for $K_X+B+\beta +\sum _{i=1}^\rho (\delta t_i)\gamma_i$ and $\beta':=\beta +\sum _{i=1}^\rho (\delta t_i)\gamma_i\in \Omega ^{\epsilon}$. Thus $\beta'\in \Omega ^{\epsilon}_j$ for an appropriate $j$ and hence $\psi =\psi _j^\epsilon \in \{\psi _{i,j}\}$  as required.

\end{proof}
\begin{proof}[Proof of Theorem \ref{thm:models}]
    Immediate from Theorem \ref{t-finite-ltms}.
\end{proof}

An immediate corollary of the above result is that if $K_X+B+\bbeta_X$ is big, then there are finitely many log minimal models.
\begin{corollary} Let $(X,B+\beta)$ be a generalized klt 3-fold and $\pi :X\to S$ a proper morphism such that either $\pi$ is bimeromorphic or $S={\rm Specan}(\C)$ and $K_X+B+\beta$ is big, then $(X,B+\beta)$ has finitely many minimal models. \end{corollary}

\subsection{Minimal Models are Connected by Flops}\label{subsec:flops}
In this section we will prove that minimal models are connected by flops.
Notice that if $(X,B+\bbeta)$ is a $\mbQ$-factorial compact K\"ahler generalized klt pair and $f_i:X\to X_i$ are log terminal models for $i=1,2$, then $X_i$ are $\mbQ$-factorial, $K_{X_i}+B_i+\bbeta _i={f_i}_*(K_X+B+\bbeta )$ is nef and $X_1\dasharrow X_2$ is an isomorphism in codimension 1. We will show that 3-fold log terminal models are connected by flips, flops and inverse flips, and in particular two generalized klt Calabi-Yau pairs are connected by flops, which generalizes a result of Koll\'ar for terminal  varieties, see \cite[Theorem 4.9]{Kol89}.\\

First we define the inverse flip.
\begin{definition}\label{def:inverse-flip}
    Let $(X, B+\bbeta)$ be a $\mbQ$-factorial compact K\"ahler generalized klt pair and $\phi:X\bir X'$ a small bimeromorphic map. If $\phi$ is a $(K_X+B+\bbeta_X)$-flip, then we call $\phi^{-1}:X'\bir X$ an \textit{inverse flip} (or anti-flip) of $K_X+B+\bbeta_X$.

\end{definition}

\begin{theorem}\label{t-flops}
 Let $(X_i,B_i+\bbeta _{X_i} )$ be compact K\"ahler strongly $\Q$-factorial generalized klt 3-folds, where $K_{X_i}+B_i+\bbeta_{X_i}$ is nef for $i=1,2$ and $\phi:X_1\dasharrow X_2$ a bimeromorphic map which is an isomorphism in codimension 1. Then the following hold:
 \begin{enumerate}
     \item $\phi$ decomposes as a finite sequence of flips, flops and inverse flips.
     \item Suppose that there is a positive constant $b>0$ such that the following holds: whenever $(K_{X_1}+B_1+\bbeta_{X_1})\cdot C>0$ for some curve $C\subset X_1$, then $(K_{X_1}+B_1+\bbeta_{X_1})\cdot C\geq b$ holds. Then $\phi$ decomposes as a finite sequence of flops.
     
      \end{enumerate}

\end{theorem}
\begin{remark}
Note that if $(X_1,B_1+\bbeta _{X_1} )$ is a good minimal model, then there is a morphism $f:X_1\to Z_1$ and a K\"ahler form $\omega _1$ on $Z_1$ such that $K_{X_1}+B_1+\bbeta_{X_1}\equiv f^* \omega _1$. Let $b:={\rm inf }\{\Sigma \cdot \omega _1\;|\;\Sigma \subset Z_1\ {\rm is\ a\ curve}\}$. So
if $(K_{X_1}+B_1+\bbeta_{X_1})\cdot C>0$ for some curve $C\subset X_1$, then $\Sigma ={f_1}_*C\ne 0$ and $(K_{X_1}+B_1+\bbeta_{X_1})\cdot C=\omega _1\cdot \Sigma \geq b$.
If instead $K_{X_1}+B_1$ is $\mathbb Q$-Cartier (and $\beta _i=0$), then $k(K_{X_1}+B_1)$ is Cartier for some $k>0$ and let $b=\frac 1 k$.
Thus, in both of these cases, the hypothesis of (2) are satisfied.
\end{remark}

\begin{proof} We remark that this proof is inspired by \cite{Kaw08} (note that in \cite{Kaw08} we have $\bbeta =0$ and $K_{X_i}+B_i$ is $\Q$-Cartier so that condition (2) holds).
Let $\omega _2$ be a K\"ahler form on $X_2$, $\oomega:=\bar \omega _2$, and  $\omega _1:=\phi ^{-1}_*\omega _2=\oomega _{X_1}$, then by Lemma \ref{lem:potentials-of-pushforward}, $\omega _1$ has local potentials, it is modified K\"ahler and $(X_1,B_1+\bbeta +\epsilon_0  \oomega )$ is generalized klt for some $0<\epsilon _0 \ll 1$, and $K_{X_1}+B_1+\bbeta_{X_1}+\eps_0\oomega_{X_1}$ is big.
We now run a $(K_{X_1}+B_1+\bbeta_{X_1}+\epsilon \omega _1)$-MMP with scaling of a sufficiently large multiple of a very general K\"ahler class, where $0<\epsilon \leq \epsilon _0$ is any fixed real number. 

By Theorem \ref{t-3-mmpbig}, this MMP terminates with a minimal model $(X^{\rm m},B^{\rm m}+\bbeta_{X^{\rm m}}+\epsilon \oomega _{X^{\rm m}})$, $\psi: X_1\dasharrow X^{\rm m}$. In particular, $K_{X^{\rm m}}+B^{\rm m}+\bbeta_{X^{\rm m}}+\epsilon \oomega _{X^{\rm m}}$ is nef and big. Since $(X_2,B_2+\bbeta_{X_2}+\epsilon \oomega _{X_2})$ is a generalized log canonical model of $(X_1,B_1+\bbeta_{X_1}+\epsilon \omega _1)$, there is a morphism $g:X^{\rm m}\to X_2$. Since $\phi$ is small, then  $g$  is a small bimeromorphic morphism of strongly $\Q$-factorial varieties, it is in fact an isomorphism by Lemma \ref{lem:exceptional-locus}. Now observe that $N(K_{X_1}+B_1+\bbeta _{X_1}+\eps\omega_1)=0$, and thus by Theorem \ref{thm:contracted-locus-of-mmp}, there are no divisorial contractions in the above MMP. So every step of this MMP is a $(K_{X_1}+B_1+\bbeta_{X_1}+\eps\omega_1)$-flip, which are in particular either flips, flops or inverse flips with respect to $K_{X_1}+B_1+\bbeta_{X_1}$ (depending on whether the $K_{X_1}+B_1+\bbeta _{X_1}+\eps\omega_1$ flipping contraction is $(K_{X_1}+B_1+\bbeta _{X_1})$-negative, trivial or positive respectively).

Suppose now that we are in case (2) and so  there is a positive constant $b>0$ such that $(K_{X_1}+B_1+\bbeta_{X_1})\cdot C\geq b$ for all curves $C\subset X_1$ such that $(K_{X_1}+B_1+\bbeta_{X_1})\cdot C>0$.
We now run a $(K_{X_1}+B_1+\bbeta_{X_1}+\epsilon \omega _1)$-MMP with scaling of a sufficiently large multiple of a very general K\"ahler class, where $0<\epsilon < b\epsilon _0/(b+6)$ is any fixed real number; this MMP terminates by Theorem \ref{t-3-mmpbig}. Now let $\mathcal K_t:=K_{X_1}+B_1+\bbeta_{X_1}+t\omega _1$.
Suppose that $C_1\subset X_1$ is a $\mathcal K_\epsilon$-flipping curve for $t=\eps$. Then $C_1\cdot \omega _1<0$ as $K_{X_1}+B_1+\bbeta_{X_1}$ is nef, and hence $C_1$ is also a  $\mathcal K_{\epsilon _0}$ flipping curve (as $\eps_0>\eps$ by assumption) and so we may assume that $0>\mathcal K_{\epsilon _0}\cdot C_1\geq -6$ by Lemma \ref{l-cone}. If $\mathcal K_0\cdot C_1>0$,  then
\[0>\mathcal K_\epsilon \cdot C_1=\left(1-\frac \epsilon {\epsilon _0}\right)\mathcal K_0\cdot C_1+\frac \epsilon {\epsilon _0} \mathcal K_{\epsilon _0}\cdot C_1\geq \left(1-\frac \epsilon {\epsilon _0}\right)b-6\frac \epsilon {\epsilon _0}>0\] 
which is impossible. Therefore $\mathcal K_0\cdot C_1=0$ and the first flip $X_1\dasharrow X_1^+$ is a $(K_{X_1}+B_1+\bbeta_{X_1})$-flop.
It follows that $K_{X^+_1}+B^+_1+\bbeta_{X^+_1}$ is nef. Suppose that $C\subset X^+_1$ is a curve such that $(K_{X^+_1}+B^+_1+\bbeta_{X^+_1})\cdot C>0$, then we claim that in fact $(K_{X^+_1}+B^+_1+\bbeta_{X^+_1})\cdot C\geq b$ and hence we may continue the procedure inductively. Thus we obtain a sequence of flips for the $(X_1,B_1+\bbeta_{X_1}+\epsilon \omega _1)$ MMP with scaling which are also $(K_{X_1}+B_1+\bbeta_{X_1})$-flops connecting $X_1$ and $X_2$.

To see the claim, let $p:Y\to X_1$ and $q:Y\to X_1^+$ be a common resolution. Then by the negativity lemma $p^*(K_{X_1}+B_1+\bbeta_{X_1})=q^*(K_{X^+_1}+B^+_1+\bbeta_{X^+_1})$. Since $(K_{X^+_1}+B^+_1+\bbeta_{X^+_1})\cdot C>0$, then $C$ is not   contained on the indeterminacy locus of $X_1^+\dasharrow X_1$ (i.e. it is not contained in the flipped locus). Let $\bar C\subset X$ be the strict transform of $C$, then $(K_{X_1}+B_1+\bbeta_{X_1})\cdot \bar C=(K_{X^+_1}+B^+_1+\bbeta_{X^+_1})\cdot C>0$ and so $(K_{X_1^+}+B_1^++\bbeta_{X^+_1})\cdot C=(K_{X_1}+B_1+\bbeta_{X_1})\cdot  \bar C\geq b$.

\end{proof}

\begin{lemma}\label{lem:potentials-of-pushforward}
Let $\phi:X\bir X'$ be a bimeromorphic map between two normal compact K\"ahler $3$-folds. Let $\omega'$ be a K\"ahler form on $X'$. If $X$ has strongly $\mbQ$-factorial klt singularities, then $\omega:=\phi^*\omega'$ is a closed positive $(1,1)$ current on $X$ with local potentials.  
\end{lemma}

\begin{proof}
Let $W$ be a resolution of the graph of $\phi$, and $p:W\to X$ and $q:W\to X'$ are the induced bimeromorphic morphisms. Then $\omega=\phi^*\omega'=p_*q^*\omega'$. Since  $X$ has strongly $\mbQ$-factorial klt singularities, by \cite[Lemma 2.32]{DH20}, there is a $\mbR$-divisor $E$ and an $(1,1)$ class $\alpha\in H^{1,1}_{\BC}(X)$ such that $q^*\omega'\equiv p^*\alpha+E$, where $E$ is $p$-exceptional. Then by the negativity lemma, $-E\>0$ is an effective divisor; in particular, $q^*\omega'-E$ is a positive current. Thus from \cite[Lemma 3.4]{HP16} it follows that $\omega=\phi^*\omega'=p_*q^*\omega'=p_*(q^*\omega'-E)$ has local potentials.  

\end{proof}

\appendix
\section{Boucksom-Zariski Decomposition}\label{sec:bz-decomposition}
We will use the definition of Boucksom-Zariski decomposition of a $(1,1)$ pseudo-effective class $\alpha\in H^{1,1}_{\BC}(X)$ on a compact complex manifold as in \cite[Definition 3.7]{Bou04}. We will also define the Lelong number of a pseudo-effective $(1,1)$ class $\alpha$ (on a manifold) as in \cite[Definition 3.1]{Bou04}. The main result of this section is Theorem \ref{thm:contracted-locus-of-mmp}.\\

We recall Boucksom's definition of the negative part of a pseudo-effective $(1,1)$ class.
\begin{definition}\label{def:negative-part}\cite[Definition 3.7]{Bou04}
    Let $X$ be a compact complex manifold and $\alpha$ a pseudo-effective $(1,1)$ class on $X$. Then we define the \textit{negative part} $N(\alpha)$ of $\alpha$ as follows:
\[
N(\alpha):=\sum_{P\subset X}\nu(\alpha, P)P,
\]
where $P$ is a prime Weil divisor on $X$. From \cite[Theorem 3.12(i)]{Bou04} it follows that $N(\alpha)$ is an effective $\mbR$-divisor.
\end{definition}

\begin{remark}\label{rmk:BZ-decomposition}
Let $X$ be a compact K\"ahler manifold and $\alpha$ a pseudo-effective $(1,1)$ class. If $N(\alpha)$ is the negative part of the Boucksom-Zariski decomposition and if $\alpha=\beta+D$, where $\beta$ is a modified nef class, and $D$ is an effective $\mbR$-divisor, then $N(\alpha)=N(\beta+D)\<N(\beta)+N(D)\< D$ by \cite[Pro. 3.2(ii) and Pro. 3.11(ii)]{Bou04}. In particular, for any prime Weil divisor $Q$ on $X$, $\nu(\alpha, Q)\<\mult_Q(D)$.
\end{remark}~\\

The following result will be useful for the proof of our main theorem in this section.
\begin{lemma}\label{lem:negativity}
Let $f:Y\to X$ be a proper bimeromorphic morphism of analytic varieties, where $X$ is relatively compact. Let $E$ be an effective $f$-exceptional $\mbR$-Cartier divisor on $Y$. Then there is a component $E'$ of $E$ such that $E'$ is covered by an analytic family of curves $\{C_t\}_{t\in T}$ such that $E\cdot C_t<0$ and $f_*C_t=0$ for all $t\in T$.
\end{lemma}

\begin{proof}
Let $\nu :\tilde Y\to Y$ be a resolution of singularities and $\tilde E=\nu ^*E$. Using Hironaka's relative Chow lemma \cite[Corollary 2]{Hir75} and then passing to a higher resolution we may assume that $\tilde f=f\circ \nu$ is a projective morphism.

Let $m:=\dim f(\Supp E)$. Replacing $X$ by a Stein open neighborhood we may assume that $X$ is a Stein space. Now we cut $X$ by $m$ general hyperplanes of $X$, and replace $Y$ and $\tilde Y$ by the corresponding inverse images. Then $f(\Supp E)$ is a finite set of points on $X$. Next since $\tilde f$ is projective, possibly shrinking $X$ further we may assume that there is a very ample divisor on $\tilde Y$. 
Thus cutting $\tilde Y$ by $n-2$ hyperplanes ($n=\dim Y$), we may assume that that $\tilde Y$ is a smooth surface. Next we replace $X$ by the Stein factorization of $\tilde f:\tilde Y\to X$ and thus assume that $X$ is a normal surface and $\tilde f$ is a projective bimeromorphic morphism from a smooth surface to a normal surface and $\tilde E$ is an effective $\tilde f$-exceptional divisor on $Y$. Let $\tilde E=\sum_{i=1}^\ell a_iC_i$. 
Since the intersection matrix of the exceptional curves of $\tilde f$ form a negative definite matrix by \cite[Lemma 3.40]{KM98}, we have $0>\tilde E^2=\sum_{i=1}^\ell a_i(E\cdot C_i)$, and thus $\tilde E\cdot C_i<0$ for some $1\<i\<\ell$. Note that $E\cdot \nu _* C_i=\tilde E\cdot C_i<0$ and hence $C_i$ is not $\nu $-exceptional.

Since $X$ is relatively compact, it can be covered by finitely many Stein open sets, and thus the lemma follows. 
\end{proof}

\begin{definition}
    Let $X$ be a normal analytic variety and $D:=\sum a_iD_i$ and $D':=\sum a'_iD_i$ two $\mbR$-divisors on $X$. Then we define the $\mbR$-divisor $D\wedge D'$ as 
    \[ 
    D\wedge D':=\sum_i\min\{a_i, a'_i\}D_i.
    \]
\end{definition}

\begin{lemma}\label{lem:basic-properties}
Let $f:Y\to X$ be a proper bimeromorphic morphism from a compact complex manifold $Y$ to a normal compact analytic variety $X$ and $\alpha$ a pseudo-effective $(1,1)$ class on $X$. If $E\>0$ is an effective $f$-exceptional $\mbR$-divisor, then $\nu(f^*\alpha+E, P)=\nu(f^*\alpha, P)+\mult_P(E)$ for every prime Weil divisor $P$ on $Y$. In particular, $N(f^*\alpha+E)=N(f^*\alpha)+E$.

\end{lemma}

\begin{proof} 
Let $E=\sum a_iE_i$. By \cite[Proposition 3.5]{Bou04}, we have $N(f^*\alpha+E)\leq N(f^*\alpha)+E$ and so $\nu(f^*\alpha+E, E)\<\nu(f^*\alpha, E)+a_i$.
To see the reverse inequality, suppose that $f^*\alpha+E=\beta +N$ is the Boucksom-Zariski decomposition of $f^*\alpha+E$ so that $N=\sum \nu (f^*\alpha+E,Q)Q$.
We claim that $E\leq N$. To see this, define $N':=N-N\wedge E$ and $E':=E-N\wedge E$, so that $f^*\alpha+E'=\beta +N'$.
We must show that $E'=0$. If this is not the case, then by Lemma \ref{lem:negativity}, there is a component $E_i$ of $E'$ which is covered by curves $\{C_t\}_{t\in T}$ such that $E'\cdot C_t<0$ and $f_*C_t=0$ for all $t\in T$. But then the family of curves $\{C_t\}_{t\in T}$ is not contained in the support of $N'$ and so
\[0> E'\cdot C_t=(f^*\alpha +E')\cdot C_t\geq N'\cdot C_t\geq 0.\] 
This is a contradiction, and hence $E\leq N$. Then $f^*\alpha =\beta +N'$, where $\beta $ is modified nef and $N':=N-E\geq 0$.
Thus from Remark \ref{rmk:BZ-decomposition} it follows that \[\nu (f^*\alpha,E_i)\leq {\rm mult}_{E_i}(N')={\rm mult}_{E_i}(N)-{\rm mult}_{E_i}(E)=\nu (f^*\alpha+E,E_i)-a_i.\]
Putting all of these together, we have that $\nu (f^*\alpha+E,E_i)=\nu (f^*\alpha,E_i)+a_i$ and hence $N(f^*\alpha+E)=N(f^*\alpha)+E$.

\end{proof}

Now we are ready to define the negative part of a pseudo-effective $(1,1)$ class on a normal variety and prove the main result of this appendix.
\begin{definition}\label{def:normal-BZ-decomposition}
Let $X$ be a normal compact analytic variety and $\alpha\in H^{1,1}_{\BC}(X)$ a pseudo-effective class. Let $f:Y\to X$ be a resolution of singularities $X$. Then we define the negative part $N(\alpha)$ as follows
\[ 
N(\alpha):=f_*(N(f^*\alpha)).
\]
The following Lemma \ref{lem:pullback-pushforward-of-negative-part} guarantees that this definition is independent of the choice of resolution $f$.
\end{definition}

\begin{lemma}\label{lem:pullback-pushforward-of-negative-part}
Let $X$ be a normal compact analytic variety and $\alpha\in H^{1,1}_{\BC}(X)$ a pseudo-effective class. Let $f:Y\to X$ and $g:Z\to X$ be two resolutions of singularities of $X$. Then
\[
f_*(N(f^*\alpha))=g_*(N(g^*\alpha)).
\]
\end{lemma}

\begin{proof} 
First assume that $X$ is a complex manifold.
Then it is easy to see from the definition of Lelong numbers that if $P$ is a prime divisor on $Y$ and $Q=f_*P\ne 0$, then $\nu (\alpha ,Q)=\nu (f^*\alpha, P)$, and hence
$f_*(N(f^*\alpha))=g_*(N(g^*\alpha))$.

Passing to the general situation, let $W$ be a common resolution of $Y$ and $Z$, and $p:W\to Y$ and $q:W\to Z$ are the projections. Then $p^*(f^*\alpha)=q^*(g^*\alpha)$, and thus by the previous argument $(f\circ p)_*N(p^*(f^*\alpha))=(g\circ q)_*N(q^*(g^*\alpha))$. But from our definition above we have $(f\circ p)_*N(p^*(f^*\alpha))=f_*(N^*(f^*\alpha))$ and $(g\circ q)_*N(q^*(g^*\alpha))=g_*(N(g^*\alpha))$, and hence $f_*N(f^*\alpha)=g_*N(g^*\alpha)$ and we are done.

\end{proof}

\begin{remark}\label{rmk:convexity-of-N}
From our definition above and Remark \ref{rmk:BZ-decomposition} it follows that
if $\alpha$ is a pseudo-effective class on a normal compact analytic variety $X$, then $N(\alpha)$ is an effective $\mbR$-divisor on $X$. Moreover, if $\alpha$ and $\beta$ are two pseudo-effective classes on $X$, then $N(\alpha+\beta)\<N(\alpha)+N(\beta)$.
\end{remark}

\begin{lemma}\label{lem:stability-of-negative-part2}
   Let $X$ be normal compact analytic variety, and $\alpha, \omega\in H^{1,1}_{\BC}(X)$ are pseudo-effective and nef classes, respectively. Then for $0<\epsilon\ll 1$, $\Supp N(\alpha+\epsilon\omega)=\Supp N(\alpha)$. 
\end{lemma}

\begin{proof}
    Let $f:X'\to X$ be a resolution of singularities of $X$. Note that if $\Supp (N(f^*(\alpha+\omega)))=\Supp (N(f^*\alpha))$, then applying $f_*$ both sides we get the required result; in particular, it is enough to prove the statement on a resolution of $X$. Thus replacing $X$ by $X'$ and $\alpha$ and $\omega$ by $f^*\alpha$ and $f^*\omega$, respectively, we may assume that $X$ is a compact complex manifold. Let $N:=N(\alpha)$ and $N_\eps:=N(\alpha+\eps\omega)$ for $\eps>0$. Then we have $N_\eps\leq N_{\eps'}\leq N$ for all $0<\eps'\leq\eps$, and thus $\Supp(N_\eps)\subset\Supp(N_{\eps'})\subset\Supp(N)$. In particular, there exists a $0<\eps_0<1$ such that $\Supp(N_\eps)$ is independent of $\epsilon$ for all $0<\eps\leq \eps_0$.  

    Now let $P$ be a prime Weil divisor not contained in $\Supp(N_\eps)$ for $0<\eps\leq \eps_0$. Let us denote $\alpha_\eps:=\alpha+\eps\omega$. Since $P$ is not in the support of $N_\eps$, $\nu(\alpha_\eps, P)=0$ for all $0<\eps\leq\eps_0$. The coefficient of $P$ in $N$ is $\nu(\alpha, P)$. Since by \cite[Proposition 3.5]{Bou04}, $\alpha\mapsto \nu(\alpha, P)$ is a lower semi-continuous function on the pseudo-effective cone, we have $\nu(\alpha, P)\leq \lim\inf_{\eps\to 0^+}\nu(\alpha_\eps, P)=0$, hence $\nu(\alpha, P)=0$. In particular, $P$ is not contained in the support of $N$. This shows that $\Supp(N_\eps)=\Supp (N)$ for all $0<\eps\leq\eps_0$.

\end{proof}

\begin{definition}\label{def:alpha-negative}
    Let $\phi:X\bir X'$ be a bimeromorphic contraction of normal compact analytic varieties. Let $\alpha\in H^{1,1}_{\BC}(X)$ and assume that $\alpha':=\phi_*\alpha\in H^{1,1}_{\BC}(X')$. We say that $\phi$ is $\alpha$-negative, if for any common resolution $p:W\to X$ and $q:W\to X'$, we may write 
    \[
p^*\alpha=q^*\alpha'+E,
    \]
    where $E\>0$ is an effective $\mbR$-divisor such that it is $q$-exceptional and $\Supp(p_*E)$ consists precisely the $\phi$-exceptional divisors on $X$. 
\end{definition}

The following theorem is the main result of this section.
\begin{theorem}\label{thm:contracted-locus-of-mmp}
Let $\phi:X\bir X'$ be a bimeromorphic contraction of normal compact analytic varieties. 
Let $(X, B+ \bbeta_X)$ and $(X', B'+\bbeta_{X'})$ be generalized dlt pairs such that $K_X+B+\bbeta_X$ is pseudo-effective and $B'+\bbeta_{X'}=\phi_*(B+\bbeta_X)$. If $\phi$ is $(K_X+B+\bbeta_X)$-negative, then the divisors contracted by $\phi$ are contained in the support of $N(K_X+B+\bbeta_X)$. In particular, if $K_{X'}+B'+\bbeta_{X'}$ is nef, then  the divisors contracted by $\phi$ are precisely the divisors in the support of $N(K_X+B+\bbeta _X)$.

\end{theorem}

\begin{proof}
Let $W$ be a compact complex manifold resolving the map $\phi$, and $p:W\to X$ and $q:W\to Y$ are the projections. Since $\phi$ is $(K_X+B+\bbeta_X)$-negative, we have
\begin{equation}\label{eqn:contracted-locus}
 p^*(K_X+B+\bbeta_X)=q^*(K_{X'}+B'+\bbeta_{X'})+E,   
\end{equation}
where $E\>0$ is an effective $q$-exceptional divisor and the support of $p_*E$ is the set of divisors contracted by $\phi$.\\
Then by Lemma \ref{lem:basic-properties} 
\[N(p^*(K_X+B+\bbeta_X))=N(q^*(K_{X'}+B'+\bbeta_{X'})+E)=N(q^*(K_{X'}+B'+\bbeta_{X'}))+E,\] 
and by Definition \ref{def:normal-BZ-decomposition},  $N(K_X+B+\bbeta_X)=p_*(N(q^*(K_{X'}+B'+\bbeta_{X'}))+E)$. In particular, the $\phi$-exceptional divisors are contained in the support of $N(K_X+B+\bbeta_X)$.\\
Moreover, if $K_{X'}+B'+\bbeta_{X'}$ is nef, then $N(q^*(K_{X'}+B'+\bbeta_{X'}))=0$, and so $N(K_X+B+\bbeta_X)=p_*E$ and we are done.

\end{proof}

\bibliographystyle{hep}
\bibliography{4foldreferences}

\end{document}